\documentclass[12pt]{amsart}
\pdfoutput=1

\usepackage[margin=1in]{geometry}
\usepackage{amsmath}
\usepackage{amssymb}
\usepackage{amsthm}
\usepackage{amscd}
\usepackage{color}
\usepackage{hyperref}
\usepackage{url}
\usepackage{mathtools}
\usepackage{graphicx}
\usepackage[all, knot]{xy}

\catcode `\@=11
\def\numberbysection{\@addtoreset{equation}{section}
         \renewcommand{\theequation}{\thesection.\arabic{equation}}}
\numberbysection
\def\subsubsection{\@startsection{subsubsection}{3}%
  \normalparindent{.5\linespacing\@plus.7\linespacing}{-.5em}%
  {\normalfont\bfseries}}

\theoremstyle{plain}
\newtheorem{theorem}{Theorem}[section]
\newtheorem{corollary}[theorem]{Corollary}
\newtheorem{proposition}[theorem]{Proposition}
\newtheorem{lemma}[theorem]{Lemma}

\theoremstyle{definition}
\newtheorem{definition}[theorem]{Definition}

\newtheorem{notation}[theorem]{Notation}
\newtheorem{remark}[theorem]{Remark}

\def\Seq{Seq^{\rm nr}}
\def\Np{\mathbb{N}^+}
\def\dSh{\mathrm{dSh}}
\def\R{\mathbb{R}}
\def\T{\mathcal{T}}
\def\Tn{{\mathcal T}_n}
\def\precT{\prec_{\T}}
\def\t{\tau}
\def\del{\partial}
\def\Tbr{\T_{\bullet}}
\def\Twr{\T_{\circ}}
\def\l{\mathbf{l}}
\def\CalF{\mathcal{F}}
\def\CalC{\mathcal{C}}

\begin{document}

\title[Permutahedral Structures of $E_2$ Operads]{Permutahedral Structures of $E_2$ Operads}

\author
[Ralph M.\ Kaufmann]{Ralph M.\ Kaufmann}
\email{rkaufman@math.purdue.edu}

\address{\textsc{Purdue University, Department of Mathematics,
 West Lafayette, IN 47907}
 and Max--Planck--Institute f\"ur Mathematik, Bonn, Germany}

\author
[Yongheng Zhang]{Yongheng Zhang}
\email{yzhang@amherst.edu }

\address{Amherst College, Department of Mathematics and Statistics,
Amherst, MA 01002
}
\begin{abstract}
There are basically two interesting breeds of $E_2$ operads, those that detect loop spaces and those that solve Deligne's conjecture.
The former deformation retract to Milgram's space obtained by gluing together permutahedra at their faces. We show how the second breed can be covered by permutahedra as well. Even more is true, the quotient is actually already an operad up to homotopy, which induces the operad structure on cellular chains adapted to prove Deligne's conjecture, while no such structure is known on Milgram's space.
We show, explicitely, that these two quotients are homotopy equivalent.
This gives a new topological proof that  operads of this type  are indeed 
of  the right homotopy type. 
It also furnishes a very nice clean description in terms of polyhedra, and with it PL topology, for the
whole story. The permutahedra and partial orders play a central role. This, in turn,  provides direct links to other fields of mathematics. We for instance find a new cellular decomposition of permutahedra using partial orders and that the permutahedra give the cells for the Dyer--Lashof operations.
\end{abstract}

\maketitle


\section*{Introduction}

 Over several decades different models of $E_2$ operads suitable for different purposes have been introduced:   the little 2-cubes operad $\mathcal{C}_2$ \cite{BV,May}, the little 2-discs operad $\mathcal{D}_2$,  the Steiner operad $\mathcal{H}_{\mathbb{R}^2}$ \cite{Ste,May2} which combines the good properties of $\mathcal{C}_2$ and $\mathcal{D}_2$, and more recently
 the Fulton-MacPherson operad $FM_2$ \cite{Kon99}.  Although as $E_2$--operads they are all quasi--isomorphic,  the individual homotopies are of interest. For the first list these are established by realizing that up to natural homotopy, (e.g.\ contracting intervals) these spaces are configuration spaces $\{F(\mathbb{R}^2,n)\}_{n\geq 1}$ of $n$ distinct ordered points on $\mathbb{R}^2$, whose homotopy type is known to be that of a $K(PBr,1)$.
 Renewed interest in $E_2$ operads stems from various solutions of Deligne's Hochschild cohomology conjecture \cite{BB,BF,Kau03,Kon99,Kon00,MS,MS03,Tar1,Tar2,Vor} and in the development of string topology \cite{CS}. In this setting cactus type operads were invented
\cite{Vor2,Kau02}. On the toplogical level, as we discuss, these are basically all isomorphic to the $E_2$ operad $Cact$ of spineless cacti introduced in \cite{Kau02}. The arity $n$ space $Cact(n)$ roughly consists of isotopy classes of embeddings of circles with positive radii into the  plane such that the images form a planted rooted planar tree picture of lobes modulo incidence parameters. For this operad and its other isomorphic versions, the proof of being and $E_2$ operad is rather indirect. It was shown using pure braid group technique of Fiedorowicz \cite{F}  and cellular operad technique of Berger \cite{Ber98}.
We now offer a direct  topological proof for the part of Fiedorowicz recognition concerning the homotopy type.

Using the different perspective of permutahedral covers, we prove the homotopy equivalence between $F(\mathbb{R}^2,n)$ and $Cact(n)$  explicitly by constructing a single homotopy equivalence between them. 
Permutahedra are an essential tool in 
the detection of loop spaces starting with Milgram \cite{Milgram66}, see \cite{Ber97,MSS} for nice reviews. They also appear in various other contexts, see  e.\ g.\  \cite{Kap, Tonks}. The full list would be too long to reproduce. They are still an active topic of research, especially through their connection to configuration spaces of points $F(\mathbb{R}^2,n)$, which is how they appear in the $E_2$ operad story, see \cite{Ber97}. From a totally different motivation, 
it has recently be shown how  $F(\mathbb{R}^2,n)$ deformation retracts to Milgram's permutahedral model $\mathcal{F}(n)$ obtained by gluing $n!$ copies of permutahedra $P_n$ along their proper faces \cite{BZ}. This applies directly to all the $E_2$ operads above based on cofiguration spaces,  giving them all a permutaheral structure, i.e.\ they appear as a quotient of permutahedral space $\coprod_{\sigma\in S_n} P_n$ and are homotopy equivalent to Milgrams model. Here and below, we write $S_n$ for the symmetric group on $n$ letters. The exception are spineless cacti $Cact$ and the models related to it, which are of a different breed.
While the configuration models, are  adapted to acting on loop spaces, through this connection spineless cacti and its relatives are adapted to acting on Hochschild complexes, or operads with multiplication.

We will prove that spineless cacti  and hence  all of its incarnations, see \S\ref{discussionpar}, have a permutahedral cover. The appearance of permutahedra in this model is very surprising, although the construction with hindsight looks very natural.
After passing to normalized spineless cacti, i.e.\ the spaces $Cact^1(n)$, we will show that they admit a presentation $\mathcal{C}(n)$ as the quotient of $n!$ copies of $P_n$.
It is important to note that here there is not only a gluing along faces, but
parts of the interior of the permutahedra are identified. We give an explicit description. Namely, $Cact^1(n)$ is a CW complex
whose cells are indexed by a certain type of labelled rooted (actually planted) planar trees. Each planar tree has an underlying  poset structure which transfers to the set of labels. We can succinctly state that each permutahedron corresponds to a  possible total order on $[n]=\{1,\dots,n\}$, viz.\ a permutation, and it is comprised of the sub-CW complex of cells indexed by partial orders on $[n]$ that are compatible with the given total order. The gluing is then along the cells that are indexed by non--total orders.
Going beyond this, there is an explicit relation between the codimension of the cells and
a partial order the partial orders. The highest co--dimension cells, that is cells of dimension $0$ are indexed by the partial order in which no elements are comparable. Since we are dealing with planar trees, see \cite{Kau02}, there
are again orders on the sets of equal height, which means that there are indeed $n!$ dimension $0$ cells, which are the vertices of the permutahedra. These combinatorics are all explained in detail below.

Due to the nature of the quotient, there is a  natural map $\mathcal{F}(n)\rightarrow Cact(n)$, whose description already yields a quasi--isomorphism. We will explicitly construct the homotopy inverse induced from compatible homotopies on the $n!$ $P_n$. In a sense, this map answers the question ``where are the centers of the lobes in cacti?''. This is not as straightforward as for the little discs, where the centers are given by the projection onto the factor of configuration space. For spineless cacti, $Cact^1$ corresponds to the centers. The quotient of $\CalF$ shows how this is related to configuration.

Recall, that $Cact^1$ has a topological operad structure, which is associative up to homotopy. An that this already induces an operad structure on the cellular level.
No such structure is known for $\CalF$. This also explains, why it was so difficult to find a proof of Deligne's conjecture. One can say that the operad structure only become apparent after taking quotients, see \S\ref{discussionpar}. This is astonishing, since instead of enlarging, we make things smaller by taking quotients.

The methods we use, are classical maps and homotopies, but for the combinatorics, we use partial orders, partitions and b/w planar trees. For these, we give a common treatment and introduce several new operators, which link our work to that of Connes and Kreimer.

Another upshot of our treatment is a new cellular decomposition of permutahedra, which has a cube at its core and then has $n-1$ shells for each $P_n$. In the tree language, the $k$--th shell is given by trees with initial branching number $k$. There is also a nice duality  between the outer faces in this decomposition of $P_{n-1}$ and the top--dimensional cells of $P_n$ leading to a recursion. This is established via the operators mentioned above.

The decomposition of the $P_n$ also allows us to recognize them as the cells responsible for the Dyer--Lashof operations.

In retrospect, spineless cacti are a natural geometric model for the sequence operad of \cite{MS03},  see \cite{Kau08}. We make this explicit  in \S\ref{isopar}.  This gives a way to show that the model of formulas \cite{MS} and hence sequences have the right homotopy type.
Our topological result also implies the result \cite{Tur} on the quasi-isomorphism between the cellular chains of $\mathcal{F}(n)$ and the cellular chains of $Cact^1(n)$. See \S\ref{discussionpar} for more details on these remarks.

The organization of the paper is as follows. Section \ref{permutasec} fixes frequently used notations and introduces the definition of unshuffless of a sequence. In it, we  also recall the definition and basic properties of the permutahedron $P_n$ and the permutahedral structure $\mathcal{F}(n)$ of $F(\mathbb{R}^2,n)$. Section \S\ref{spineslesssec} recalls the definition of spineless cacti and make explicit its polysimplicial structure. The  permutahedral structure $\mathcal{C}(n)$ of $Cact^1(n)$ is given in \S\ref{mainsec} using partial and total orders. This contains one direction of the homotopy equivalence. Here, we also introduce four operators $B^{\pm}_{b/w}$ acting on trees that are essential in keeping track of the combinatorics. These operators are analogous to those used in \cite{ConnesKreimer}. The homotopy equivalences between $\mathcal{F}(n)$ and $\mathcal{C}(n)$ is proven in Section \S\ref{homotopysec}, by giving
and explicit homotopy inverse. Some of the more tedious details are relegated to the Appendix. Finally, we give a more detailed discussion of $E_2$ operads and applications in \S\ref{discussionpar}.

\section*{Acknowledgements}
We would like  thank Clemens Berger for discussions.
RK would like to thank the Max-Planck Institute for Mathematics in Bonn for the hospitality and the Program of Higher Structures.
RK  thankfully  acknowledges  support  from  the  Simons foundation under collaboration grant \# 317149.

\section{Permutations, Permutohedra and Milgram's model}
\label{permutasec}
In this section, we start by recalling the definition of a permutahedron. We then set up the combinatorial language, which we will use for indexing. This is unavoidably a bit complex, as we will have to deal with lists of lists.
Thus we will introduce a short hand notation for these lists and manipulations on them. Besides reducing clutter, an additional benefit is an easy description of a poset structure and a grading.
 This allows us to encode the poset structure of the faces of permutahedra in this formalism.

\subsection{Permutohedra}

Before we recall the definition of our main actors, the permutahedra \cite{Mil}, we fix our notations for sequences in general and elements in the symmetric group $S_n$ in particular.

\begin{definition}
Let $\mathbb{N}^+$ the set of positive integers. For $n\in \mathbb{N}^+$,  set $[n]=\{1,2,\cdots,n\}$. A \textit{sequence} of length $n$ is a function $\phi:[n]\to \mathbb{N}^+$. $\phi$ is called a \textit{non--repeating sequence} (nr--sequence)
if this function is also injective.
We say the length $|\phi|$ of $\phi$ is $n$. By $\mathrm{Seq}_n$, we mean the set of all sequences of length $n$ and $\mathrm{\Seq}_n$, the set of all nr-sequences of length $n$.
\end{definition}

\begin{notation}
Any nr-sequence can be identified by a nonempty ordered list of distinct elements in $\mathbb{N}^+$ given by its images. Denote by $\phi_i:=\phi(i)$ the image of $i$ under $\phi$.
By abuse of notation, to specify $\phi$, we will use the following list notation, $\phi_1\phi_2\cdots \phi_n$,  where we do not use commas to separate the terms if no confusion arises.
We write $\{\phi\}$ for the image of $\phi$, which is the set $\{\phi(1),\phi(2),\cdots,\phi(n)\}$.
\end{notation}

\textit{Example.} The symmetric group $S_n$ consists of $n!$ bijective functions $\sigma: [n]\hookrightarrow [n]$, namely nr-sequences of length $n$ whose domain and codomain overlap. Each $\sigma$ can be identified with the list of its images $\sigma_1\sigma_2\cdots\sigma_n$. This is a short hand for the traditional notation $\left(
\begin{matrix}
 1&2&3&\dots&n\\
\sigma(1)&\sigma(2)&\sigma(3)&\dots&\sigma(n)
\end{matrix}\right)$. For example $1234$ is $id\in S_4$ and $2143$ is the product of the transpositions switching 1 and 2, and 3 and 4 respectively.

\begin{definition}
Given $\sigma \in S_n$, we define the vector $\mathbf{v}_{\sigma}$ in $\mathbb{R}^n$ as follows $$\mathbf{v}_{\sigma}:=
(\sigma^{-1}(1),\sigma^{-1}(2),\cdots,\sigma^{-1}(n)).$$
\end{definition}
\begin{remark}
Here, we follow the convention of labelling the vertices by the inverse permutations, see e.g.\ \cite{Kap}, which has the effect that the faces of permutohedra will be conveniently labelled by lists (or better unshuffles, see below), rather than by surjections.
\end{remark}
\textit{Example.} If $\sigma=3241$, then $\sigma^{-1}=4213$ and thus $\mathbf{v}_{\sigma}=(4,2,1,3)\in \mathbb{R}^4$.

\begin{definition}
The permutohedron $P_n$ is the convex hull of the set of points $\{\mathbf{v}_{\sigma}\in \mathbb{R}^n|\sigma\in S_n\}$, i.e.: $$P_n=\{\sum_{\sigma\in S_n}t_{\sigma}\mathbf{v}_{\sigma}\in \mathbb{R}^n|\sum_{\sigma\in S_n}t_{\sigma}=1,t_{\sigma}\geq 0\}.$$
See Figure \ref{permfig} for examples.

\begin{figure}
		\label{figure 1}
		\centering
		\includegraphics[width=150mm]{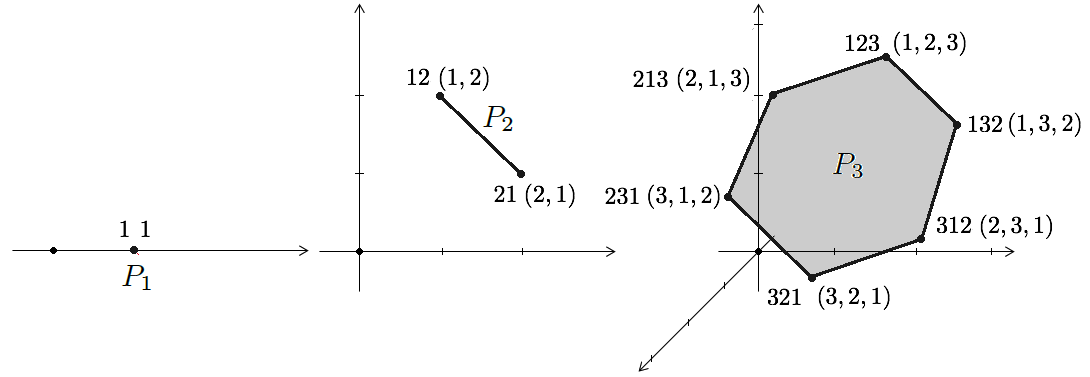}
		\caption{\label{permfig} The permutahedra $P_1$, $P_2$ and $P_3$.
		Each vertex is labelled $\sigma$ on the left and $\mathbf{v}_{\sigma}$ on the right.}
\end{figure}
\end{definition}

The permutohedron $P_n$ enjoys the following features which are readily checked:

\begin{enumerate}
\item $P_n$ is a polytope of dimension $n-1$.
\item The vertex set of $P_n$ is $\{\mathbf{v}_{\sigma}\in \mathbb{R}^n|\sigma \in S_n\}$.
\item $P_n$ is contained in the hyperplane $\{(x_1,\cdots,x_n)\in \mathbb{R}^n|x_1+\cdots+x_n=\frac{(n+1)n}{2}\}$.
\item
Two vertices $\mathbf{v}_{\sigma}$, $\mathbf{v}_{\tau}$ of $P_n$, $n\geq 2$ are adjacent if and only if $\mathbf{v}_{\tau}$ is obtained from $\mathbf{v}_{\sigma}$ by switching two coordinate values differing by 1 (or $\tau$ is obtained from $\sigma$ by switching two adjacent numbers in their image lists). In this case the Euclidean distance from $\mathbf{v}_{\sigma}$ to $\mathbf{v}_{\tau}$ is the minimal distance between two vertices, which is $\sqrt{2}$.
\item Its dimension $n-k$ faces are affinely isomorphic to $P_{m_1}\times P_{m_2}\times \cdots \times P_{m_k}$.
\end{enumerate}
\subsection{Notation for subsequences and unshuffles}
\subsubsection{Subsequences}
\label{permdefpar}

\begin{definition}
A \textit{subsequence} of length $k$ ($k\leq n$) of the sequence $\phi:[n]\rightarrow \mathbb{N}^+$ is a composite of functions $\phi\circ \psi$, where $\psi:[k]\hookrightarrow [n]$ is strictly increasing.  In the short hand notation,  $\phi\circ \psi$ is simply written as $\phi_{\psi_1}\phi_{\psi_2}\cdots \phi_{\psi_k}$. In particular,
for $n\geq 2$ and $\sigma\in S_n$, we define $\sigma\backslash \sigma_1$ be the subsequence $\sigma\circ \delta_1$ where $\delta_1:[n-1]\rightarrow [n]$ is the first face map  $\delta_1(i)=i+1$.
\end{definition}

So $\sigma\backslash \delta_1$ can also be written $\sigma_2\sigma_3\cdots \sigma_n$. Thus, $\sigma\backslash \sigma_1$ is obtained by removing the first term in the sequence $\sigma_1\sigma_2\cdots\sigma_n$.\\

\textit{Example.} If $\sigma\in S_4$ is the sequence $2341$ then $\sigma\backslash \sigma_1$ is $341$.

\subsubsection{Shuffles and unshuffles}
\begin{definition}
An \textit{unshuffle} of a sequence $\phi$ into $k$ subsequences of lengths $m_1,m_2,\cdots,m_k$ is an ordered list of subsequences $\mathbf{l}_1,\mathbf{l}_2,\cdots,\mathbf{l}_k$ of $\phi$
such that $|\mathbf{l}_i|=m_i$ and the disjoint union $\coprod_{i=1}^k\{\mathbf{l}_i\}$ equals $\{\phi\}$. We also call $\phi$ a \textit{shuffle} of $\mathbf{l}_1,\mathbf{l}_2,\cdots,\mathbf{l}_k$.

We define $\mathrm{dSh}_\phi[m_1,\cdots,m_k]$ to be the set of all unshuffless of
$\phi$ into subsequences of lengths $m_1,\cdots,m_k$,
$\mathrm{dSh}_\phi(k)=\coprod_{(m_1,\dots,m_k)}\mathrm{dSh}_\phi[m_1,\cdots,m_k]$ to be
the set of all unshuffless (or deshuffles) of $\phi$ into $k$ subsequences and $\mathrm{dSh}_\phi=\coprod_{k}\mathrm{dSh}_\phi(k)$ to be the set of all unshuffless of $\phi$.
\end{definition}

\begin{notation} We will use the following bar notation to give elements of $\mathrm{dSh}_\phi$(k). We write $\mathbf{l}_1|\mathbf{l}_2|\cdots|\mathbf{l}_k$, $k\geq 1$, for the list $\mathbf{l}_1, \mathbf{l}_2,\cdots, \mathbf{l}_k$, i.e.\ when $\phi$ is a shuffle of  $\mathbf{l}_1, \mathbf{l}_2,\cdots, \mathbf{l}_k$.

\end{notation}
\textit{Example.} Let $3214\in S_4$. Then $\mathrm{dSh}_{3214}[3,1]$ consists of the four elements: $321|4$, $324|1$, $314|2$ and $214|3$. And $\mathrm{dSh}_{3214}[2,2]$ consists of $32|14$, $31|24$, $34|21$, $21|34$,
$24|31$ and $14|32$.
\subsubsection{Grading and poset structure}
 We define the degree ($deg$) of elements in $\mathrm{dSh}_\phi(k)$ to be $|\phi|-k$.
This is the length of $\phi$ minus 1, minus the number of bars ($k-1$). It lies between
$|\phi|-1$ and $0$.

On lists there is the operation of merging lists. Given two sequences $\mathbf{l}_1,\mathbf{l}_2$
with disjoint domains $X_1,X_2$, we define $\mu(\mathbf{l}_1,\mathbf{l}_2):=\mathbf{l}_1\mathbf{l}_2$ to be the function $\mathbf{l}_2\amalg \mathbf{l}_2:X_1\amalg X_2\to \Np$. Note that in our shorthand notation the merging of two lists is exactly the juxtaposition given by removing
a bar.

The partial order $\prec$ on $\dSh_\phi$ is generated by removing bars and shuffling the lists.  More precisely, $\prec$ is the transitive closure of the relation

\begin{equation}
\label{partial order}
\mathbf{l}_1\cdots|\mathbf{l}_{i-1}|\mathbf{l}_i|\mathbf{l}_{i+1}|\mathbf{l}_{i+2}|\cdots|\mathbf{l}_k\mbox{ }\prec\mbox{ }\mathbf{l}_1|\cdots|\mathbf{l}_{i-1}|\mathbf{h}|\mathbf{l}_{i+2}|\cdots|\mathbf{l}_k,
\end{equation}
where $\mathbf{l}_{i-1}|\mathbf{l}_i\in dSh_{\mathbf{h}}(2)$ or simply $\mathbf{h}$ is a shuffle of $\mathbf{l}_{i-1} ,\mathbf{l}_i$.
It follows that the partial order decreases degree; that is if two elements that are in the relation $\mathbf{a}\prec \mathbf{b}$ then $deg(\mathbf{a})<deg(\mathbf{b})$.

\begin{notation}
\label{Jnota}
$\mathcal{J}_{\phi}$  will denote the poset $(\dSh_{\phi},\prec)$ and $\mathcal{J}^i_{\phi}$, $i=0,1,\cdots,|\phi|-1$ be the subset consisting
of elements of degree $i$ in $\mathcal{J}_{\phi}$.
\end{notation}

\textit{Example.} For $\sigma=145372896\in S_9$, we have $153|49|76|28 \mbox{ }\prec\mbox{ } 153|4796|28$, which are elements in $\mathcal{J}^5_{\sigma}$ and $\mathcal{J}^6_{\sigma}$, respectively.

\begin{remark}
Notice that any poset $\mathcal{J}_{\phi}$ represents a category
by setting $Hom(\mathbf{a},\mathbf{b}) = \{(\mathbf{a},\mathbf{b})\},$ if $(\mathbf{a}\preceq\mathbf{b})$. The category has a terminal element $\phi_1\cdots \phi_n$. One can formally add the one element set
$\mathcal{J}_{\phi}^{-1}$ and obtain an initial object.
\end{remark}

\subsubsection{Geometric realization}
\label{geometricpar}
We define a the geometric realization of $\mathcal{J}_{\phi}$
which is formally a functor $\mathcal{F}$ from $\mathcal{J}_{\phi}$ to the category of topological spaces and inclusions ---in fact,
 polytopes and face inclusions, which are inclusions of $\R^n\to \R^m$ and affine transformations. Although it would be more natural to order using $\phi$, to match the conventions 
 of faces given by lists, instead of surjections, we will use the inverse ordering. 
 
Let $\phi\in \Seq$ and $n=|\phi|$. Now $\phi$ is injective and
hence restricting it to its image, we get a map $\phi^{-1}:Im(\phi)\to [n]$. 
We let $\phi^{-1}_1,\dots ,\phi^{-1}_n$ be the ordered preimage,
that is $\phi^{-1}_1$ is $\phi^{-1}$ applied to the smallest image of $\phi$. In particular if $\phi=\sigma$ a permutation then $\phi^{-1}=\sigma^{-1}$, the inverse permutation and the notation agrees with the previous one.

Define  $\mathbf{v}_{\phi}=(\phi^{-1}_1,\phi^{-1}_2,\dots,\phi^{-1}_n)\in \R^n$.
Then $\mathcal{F}$ is  defined by $$\mathcal{F}(\phi_1|\phi_2|\cdots|\phi_n)=\mathbf{v}_{\phi}$$ on degree $0$ elements and $\mathcal{F}(\mathbf{a})$ is defined to be the convex hull of $\{\mathcal{F}(\mathbf{b}\in \mathbf{R}^n|\mathbf{b}\in \mathcal{J}^0_{\phi}, \mathbf{b} \preceq \mathbf{a}\}$ for general $\mathbf{a}\in \mathcal{J}_{f}$.
Finally, we define $\mathcal{F}$ on $\prec$ to be face inclusions. 
\begin{proposition}
 $\phi\in \Seq$ then $\mathcal{F}(\phi)$
is an $|\phi|-1$ dimensional polytope, whose dimension $i$ faces correspond to elements of $\mathcal{J}_{\phi}^i$.
In particular, for a permutation $\sigma:[n]\hookrightarrow[n]$:
$\mathcal{F}(\mathcal{J}_{\sigma})=P_n$.
\end{proposition}
In the latter equality the data of $\sigma$ is present in the labellings.

\begin{proof}
One can reduce to $\sigma=id$ and then
we refer the reader to  \cite{Ber97}, \cite{CM}, \cite{Kap}, \cite{Mil}, \cite{MSS} and \cite{Zie} for a proof.
\end{proof}
The example of the labelling of  $P_{\sigma}$ for $\sigma=1234\in S_4$ is given in Figure 3.2.

\begin{figure}
		\centering
		\includegraphics[width=150mm]{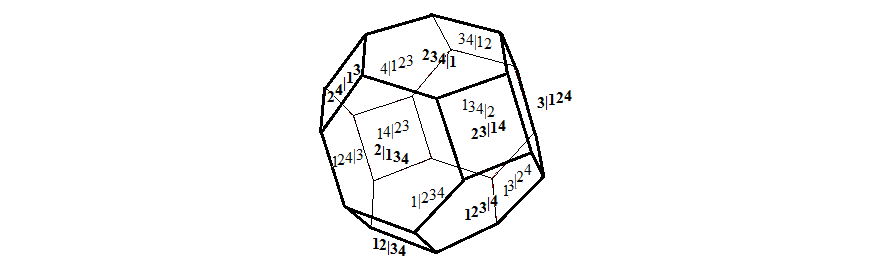}
		\caption{The codimension 1 faces of $P_4$ and their indexing elements in $\mathcal{J}^2_{1234}$. Visible faces are labelled by bold-faced numbers. The faces of the types $abc|d$, $ab|cd$ and $a|bcd$ are affinely isomorphic to $P_3\times P_1$, $P_2\times P_2$ and $P_1\times P_3$, respectively.}
\end{figure}

$F(\mathbb{R}^2,n)$ deformation retracts to a space which is obtained by gluing $n!$ copies of $P_n$. We first describe the gluing data through a poset $\mathcal{J}_n$ which contains all the $n!$ posets $\mathcal{J}_{\sigma}$ introduced in the previous chapter.

\subsection{Milgram's model via the poset $\mathcal{J}(n)$}

\begin{definition}
As a set, the poset $\mathcal{J}(n)$ equals the  union $\bigcup_{\sigma\in S_n}\mathcal{J}_{\sigma}$. The partial order of $\mathcal{J}(n)$ is defined the same way as that in \eqref{partial order}.
\end{definition}
Notice that we are dealing with the union and not the disjoint union. So elements of $\mathcal{J}_{\sigma}$ and $\mathcal{J}_{\sigma'}$ can become identified. This leads to a different poset structure.

\textit{Example.} $1234$ is the only element in $\mathcal{J}_{1234}$ that is greater than $13|24$. But in $\mathcal{J}(4)$, the elements greater than $13|24$ are $1324$, $1234$, $1243$, $2134$, $2143$ and $2413$.\\

\begin{definition}
We extend $\mathcal{F}$ from  $\mathcal{J}_{\sigma}$, $\sigma\in S_n$ to $\mathcal{J}(n)$ naturally by setting $$\mathcal{F}(n):=\mathrm{colim}_{\mathcal{J}(n)}\mathcal{F}.$$
\end{definition}

This means that as a topological space $\mathcal{F}(n)$ is obtained by gluing $n!$ copies of $P_n$ along their proper faces according to the partially order set $\mathcal{J}(n)$. Alternatively, we can write $$\mathcal{F}(n)=\left(\coprod_{\sigma\in S_n}P_n\right)/\sim_{\mathcal F},$$ where for $x \in P_n$ indexed by $\sigma$ and $y \in P_n$ indexed by $\tau$, $x\sim_{\mathcal{F}}y$ if there is $\mathbf{a}\in \mathcal{J}_{\sigma}\cap\mathcal{J}_{\tau}$ such that $x$ and $y$ have the same coordinates in $\mathcal{F}_{\mathbf{a}}$ (we simply write $x=y\in \mathcal{F}_{\mathbf{a}}$ in the future).\\

\subsection{Permutahedral structure of $F(\mathbb{R}^2,n)$: A theorem of Blagojevi\'{c} and Ziegler}

\begin{theorem}[\cite{BZ}]
$\mathcal{F}(n)$ is homeomorphic to a strong deformation retract of $F(\mathbb{R}^2,n)$.
\end{theorem}

\textit{Remark.} That $\mathcal{F}(n)$ and $F(\mathbb{R}^2,n)$ have the same homotopy type was known before this theorem. For example, \cite{Ber97} showed this by establishing a zig-zag   connecting $\mathcal{F}(n)$ and $F(\mathbb{R}^2,n)$. But this theorem is  stronger: it shows that one is actually the deformation retract of the other. In fact, \cite{BZ} described regular CW complex models which are homeomorphic to deformation retracts of the configuration spaces $F(\mathbb{R}^k,n)$ for all $k,n\geq 1$, which were used in their proof when $n$ is a prime power of the conjecture of Nandakumar and Ramana Rao that every polygon can be partitioned into $n$ convex parts of equal area and perimeter. The same CW complex models were also studied in \cite{Ber97} and \cite{GJ} and they were called the Milgram's permutahedral model in \cite{Ber97}.
We briefly review the proof of the above theorem here.

\begin{proof}[Sketch of proof according to \cite{BZ}.]
 First, $\mathbb{R}^{2n}$ deformation retracts to the subspace $W_n^{\oplus 2}$ in which the geometric center of each configuration is shifted to the origin. We denote this retraction by $r_1$. Then $W^{\oplus 2}_n\backslash 0^n$ is partitioned into relatively open infinite polyhedral cones. These cones give the Fox-Neuwirth stratification of $W^{\oplus 2}_n\backslash 0^n$ and they constitute a partially ordered set. Next, a relative interior point for each cone is chosen. These points yield the vertices of a star-shaped PL cell. Then $W^{\oplus 2}_n\backslash 0^n$ radially deformation retracts to the boundary of this PL cell. We denote this retraction by $r_2$. Finally, the Poincar\'{e}-Alexander dual complex of $r_2\circ r_1 (F(\mathbb{R}^2,n))$ relative to $r_2 \left(W^{\oplus 2}_n\backslash 0^n\backslash r_1(F(\mathbb{R}^2,n))\right)$ is constructed, which is a deformation retract of $r_2\circ r_1 (F(\mathbb{R}^2,n))$. Let this third retraction be $r_3$. In conclusion, $F(\mathbb{R}^2,n)$ deformation retracts to $r_3\circ r_2\circ r_1 (F(\mathbb{R}^2,n))$, which has a partially ordered set structure with the partial order the reverse of that of the Fox-Neuwirth stratification. This partially ordered set is precisely $\mathcal{J}(n)$ and $\mathcal{F}(n)$ is homeomorphic to $r_3\circ r_2\circ r_1 (F(\mathbb{R}^2,n))$.

\end{proof}

\section{The operad of Spineless Cacti}
\label{spineslesssec}
\subsection{The spineless cacti operad $Cact$ and its normalized version $Cact^1(n)$}
The operad of spineless cacti $Cact$ was introduced in \cite{Kau02}. We first briefly review $Cact=\{Cact(n)\}_{n\geq 1}$ using the intuitive picture of cacti from \cite{Vor2}. Although very intuitive, this description is unexpectedly hard to make precise topologically. A better way to
define the spaces
is to first define  CW complexes $Cact^1$, the spaces of normalized spineless cacti, which correspond to the restriction to lobes of radius $1$ and then extend to all positive radii by taking products with $\R_+^n$ \cite{Kau02} .

\subsubsection{Pictorial description}
Roughly a cactus \cite{Vor2} is an isotopy class of tree--like configuration of circles in the plane with a given base point. Here a circle is an orientation preserving embedding
$S^1_r\rightarrow \mathbb{R}^2$, where $S^1_r=\{x^2+y^2=r^2\}$ and the isotopies should preserve the incidence relations. The circles are also called lobes. The images of the base points are called local roots or zeros and the root is called a global zero. To be a spineless cactus means that any local zero is the  unique  intersection point of the lobe with  the lobe closer to the global zero (this exists due to the tree-like structure). An element of $Cact(5)$ is given in Figure \ref{cactfig}.

\begin{figure}
		\centering
		\includegraphics[width=150mm]{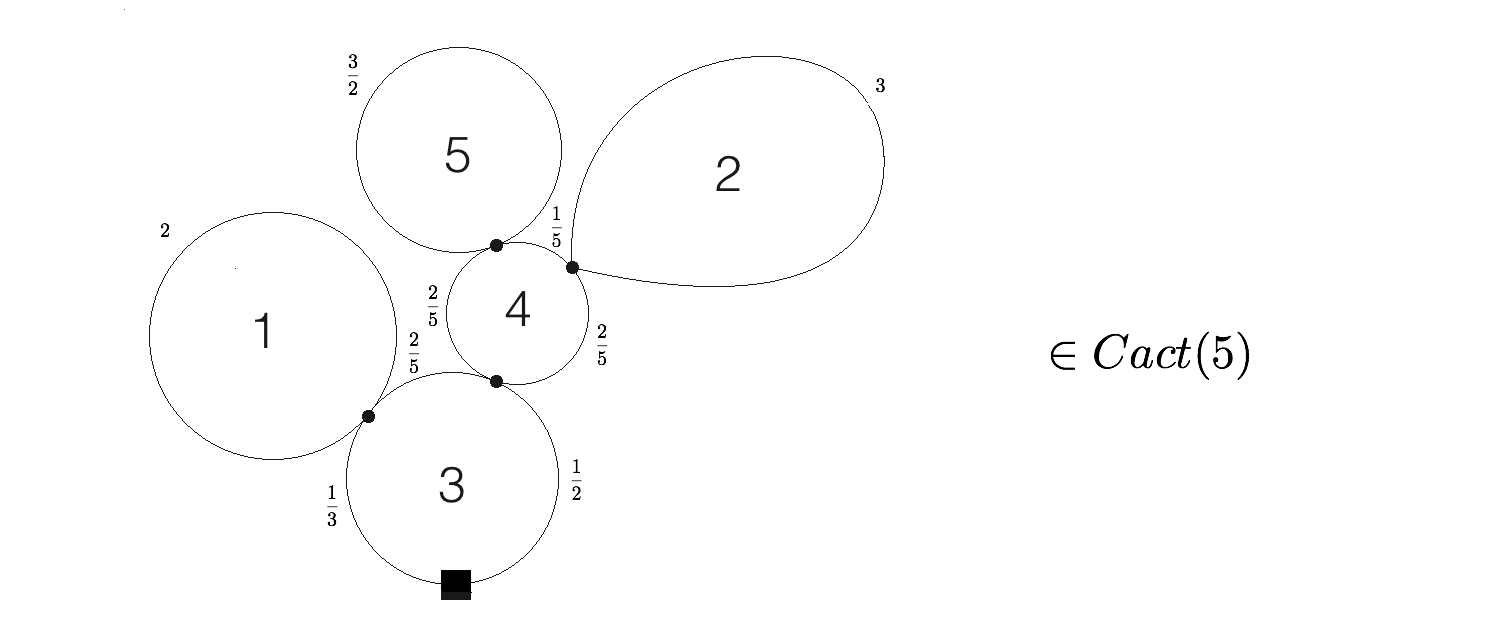}
		\caption{\label{cactfig}A representative of an element in $Cact(5)$  representing isotopy class of orientation and intersection parameter preserving embeddings of the five standard circles of radii $2,3,\frac{37}{30},1,\frac{3}{2}$ such that the images form a rooted planted tree-like configuration of circles. The local zeros are denoted by black dots. The global zero is denoted by a black square.}
\end{figure}

 Notice that for any $c\in Cact(n)$, if one starts from the root vertex (the black square) and travel around the perimeter of the configuration then one will eventually come back to the root vertex. The path travelled gives a map from $S^1$ to the configuration and is  called the outside circle.
\subsubsection{CW-complex}
First notice  that a configuration as described above gives rise to a b/w rooted bi--partite graph $\tau$. The white vertices are the lobes, the black vertices are the the zeros, with the global zero being the root.
A black vertex is joined by an edge to a white vertex if the corresponding point lies on the lobe. Tree--like means that the graph is a tree. This tree is also planar, since the configuration was planar. A cactus is spineless if the local zero is at the unique intersection point nearest the root, and hence can be ignored.  We now turn this observation around to make a precise definition.

Each $Cact^1(n)$ is a CW regular complex whose cells are indexed by
 planted planar black and white bi--partite trees with a black root and white leaves and a total of $n$ labelled white vertices.
  The open cell $\mathring{C}_{\tau}$ indexed by the tree $\tau$ is defined as the product of open simplices $\prod_{i=1}^n\mathring{\Delta}^{|v_i|-1}$,
  where $|v_i|$ is the number of incident edges of the white vertex labelled $i$. The number of incoming edges or the arity is then $|v_i|-1$. The closure $C(\tau)$ equals $\prod_{i=1}^n\Delta^{|v_i|-1}$ and it is attached by collapsing angles at white vertices, see \cite{Kau02} and Figure \ref{collapsefig} for details.  This angle collapse  corresponds to the contraction of an arc of a lobe, e.g.\ the arc labelled by $\frac{1}{2}$ or $\frac{2}{5}$ in  Figure \ref{cactfig}. These arc--labels correspond to the barycentric coordinates of the simplices. The attaching map can then be understood as sending one of these co--ordinates to zero, removing this co--ordinate and identifying the result with the barycentric coordinates of the tree obtained by collapsing the angle.
\begin{figure}[h]
		\centering
		\includegraphics[width=150mm]{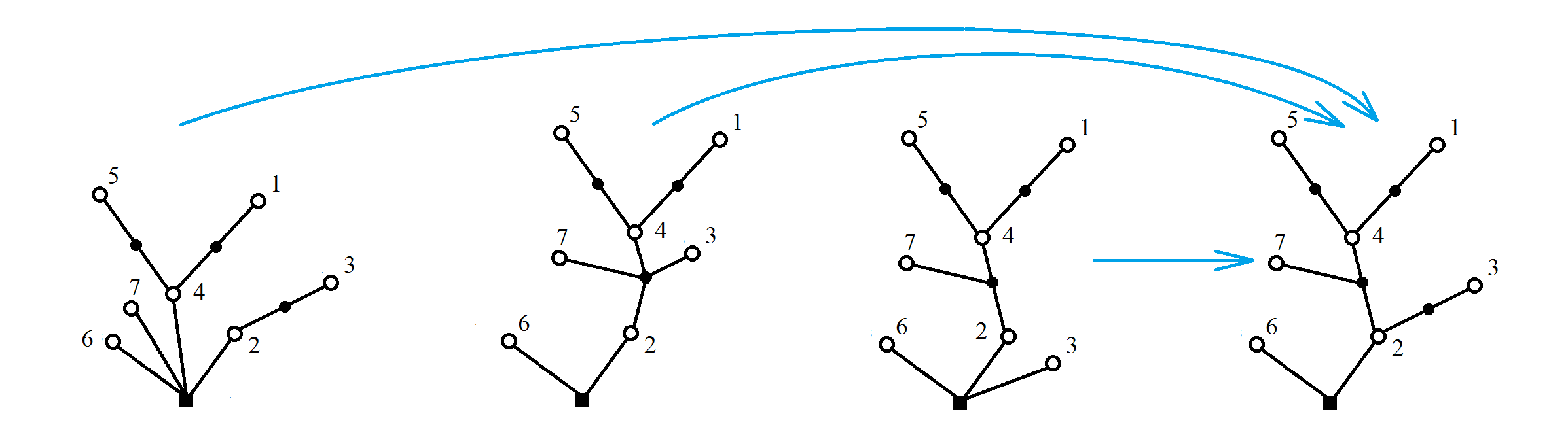}
		\caption{\label{collapsefig}Example of angle collapses having the same target $\tau$.}
\end{figure}

The space $Cact(n)$ is the product $Cact^1(n)\times \mathbb{R}^n_{>0}$ with the product topology.
 Note that $Cact(n)$ naturally deformation retracts to $Cact^1(n)$.

\subsubsection{Grading}

For a tree $\tau$ we define its degree as $i=dim(C(\tau))=\sum_i (|v_i|-1)$.
Let $\mathcal{T}^i_n$ be the subset of $\mathcal{T}_n$ of degree $i$. $\mathcal{T}^0_n$ consists of the minimal degree elements in $\mathcal{T}_n$ and $\mathcal{T}^{n-1}_n$ the maximal degree elements. $\mathcal{T}^0_n$ is also the set of trees indexing the spineless corolla cacti $SCC(n)$ \cite{Kau02}. We let $scc(\sigma)$ be the element in $\mathcal{T}^0_n$ shown in Figure \ref{sccfig}. 

\begin{figure}
		
		\centering
		\includegraphics[width=150mm]{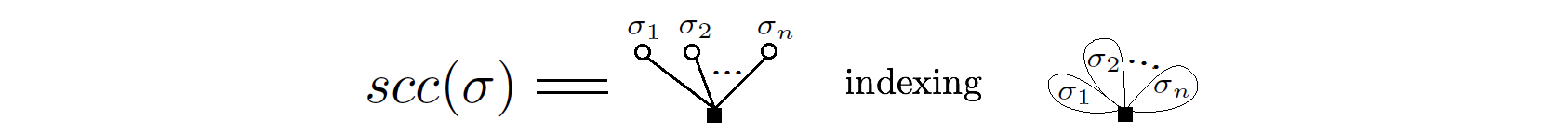}
		\caption{\label{sccfig}The spineless corolla cactus element $scc(\sigma)$ for $\sigma\in S_n$}
\end{figure}

 \subsubsection{Operadic structure} Although not strictly needed for the present discussion, we give the operad structure of this $E_2$ operad using the intuitive picture.
 Given $c_1\in Cact(m)$ and $c_2\in Cact(n)$, $c_1\circ_i c_2\in Cact(m+n-1)$ is obtained by rescaling the outside circle of $c_2$ to that of the $i$'th circle of $c_1$ and then identifying the outside circle of the resultant configuration to the $i$'th lobe of $c_1$. $S_n$ acts on $Cact(n)$ by permuting the labels.

One can check that the above structures make $Cact$ an operad (more precisely, pseudo-operad).

\begin{figure}[!h]
		\centering
		\includegraphics[width=150mm]{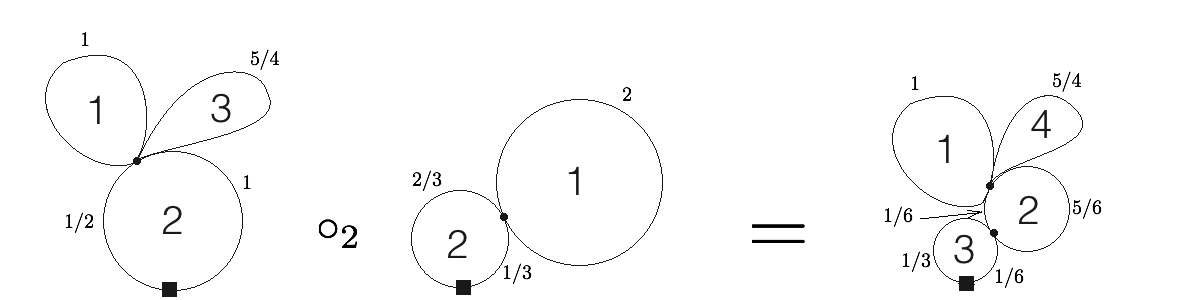}
		\caption{An example of operadic insertion $\circ_2: Cact(3)\times Cact(2)\rightarrow Cact(4)$.}
\end{figure}

\subsubsection{Summary}
Let $\mathcal{T}_n$ denote the partially ordered set of  planar planted bipartite (black and white) trees with white leaves, a black root, and $n$ white vertices labelled from $1$ to $n$. Denote by $\T_n^i$ the subset of trees with degree $i$. Then as a  stratified set
\begin{equation}
Cact^1(n)=\coprod_{\tau\in \mathcal{T}_n}\mathring{C}_{\tau}
\end{equation}
and as a space
\begin{equation}
\label{glueeq}
Cact^1(n)=\coprod_{\tau\in \mathcal{T}_n}C(\tau)/ \sim=\coprod_{\tau\in \mathcal{T}_n^{n-1}}C(\tau)/ \sim
\end{equation}
where $x\sim y$ is in the closure of the  relation induced by the attaching maps.
The last equation is true, since all points are included in some top--dimensional cell.

\subsection{Reformulation of $Cact^1(n)$ as the colimit of a poset}
The main result in this section is that the angle collapse actually gives a poset structure to $\Tn$. Moreover, since gluing procedures are alternatively described by relative co--products, we can ultimately describe $Cact^1$ as a colimit over a poset category of a realization functor.

We say that $\t{\angle} \t'$ if $\t$ can be obtained from $\t'$ by an angle collapse.

\begin{definition}
Let $\precT$ be the partial order obtained from the transitive closure of  the relation $\angle$ on  $\mathcal{T}_n$ induced by angle collapse.
\end{definition}

Again, $\t\precT\t'$ implies $deg(\t)<deg(\t')$ and the minimal elements form the set $SCC(n)$ and the maximal elements are those of $\Tn^{n-1}$.

Let $C$ be the following functor from the poset category $(\mathcal{T}_n,\precT)$ to the category of topological spaces. That is for each pair $\t\preceq_T\t'$ there is a unique arrow $\t\to\t'$.

\begin{enumerate}
\item For $\tau\in \mathcal{T}_n$, $C(\tau)$ is defined as before: $C(\tau)=\Delta^{w_1}\times \Delta^{w_2}\times \cdots \times \Delta^{w_n}$, where $w_i$ is the number of incoming edges to the white vertex labelled by $i$.
\item If $\tau\angle \tau'$, and $\tau$ is obtained from $\tau'$ by collapsing the angle between the  $j$th and the $(j+1)$th incoming edges of the white vertex $i$ (where we define the $0$th and the $(w_i+1)$th incoming edges to be the outgoing edge of this white vertex), then  
$$C(\t\angle\t'
)=\mathrm{id}_{\Delta^{w_1}}\times \cdots\times \mathrm{id}_{\Delta^{w_{i-1}}}\times s_j\times \mathrm{id}_{\Delta^{w_{i+1}}} \times\cdots\times \mathrm{id}_{\Delta^{w_n}},$$ where $s_j$ is the j--th degeneracy map
\begin{equation*}
\begin{array}{ccccc}
s_j&:& \Delta^{w_i-1} & \rightarrow & \Delta^{w_i}\\
          & & (t_0,t_1,\cdots,t_{w_i-1}) & \mapsto & (t_0,t_1,\cdots,t_{j-1},0,t_j,\cdots,t_{w_i-1}).
\end{array}
\end{equation*}
\end{enumerate}

Then it follows from  \eqref{glueeq} that:
\begin{proposition}
$$Cact^1(n)=\mathrm{colim}_{\mathcal{T}_n}C.$$
\qed
\end{proposition}
An example of the gluing is given in Figure \ref{gluingfig}.
\begin{figure}
		\centering
		\includegraphics[width=150mm]{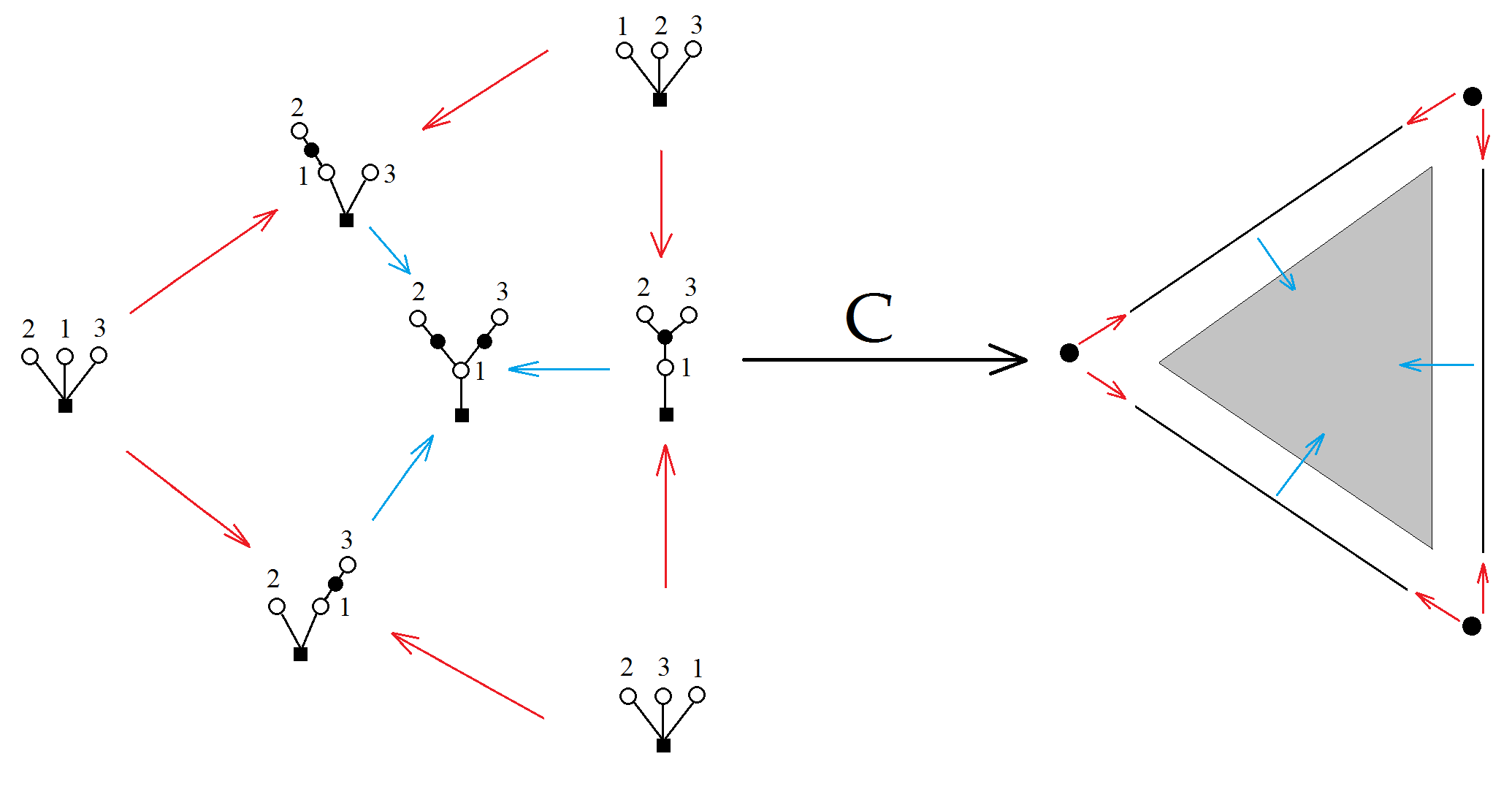}
		\caption{\label{gluingfig}A subposet of $\mathcal{T}_n$ and its image under the functor $C$.}
\end{figure}

\subsection{$Cact^1(n)$ as a poly--simplicial set}
What is actually obvious from this reformulation, but not stated explicitely in \cite{Kau02} is that $Cact^1$ is not only a regular CW complex, but the realization of a poly--semi--simplicial set. The poly--degeneracy maps are given by angle collapses.
\begin{proposition}
$\Tn$ is a poly--semi--simplicial set and $Cact^1=|\Tn|$. \qed
\end{proposition}

\section{A permutahedral cover for $Cact^1$}
\label{mainsec}

\subsection{Tools and setup}
In this section, we provide the necessary combinatorial tools for the statements and proofs. We introduce several partial and total orders in order to define $n!$ sub-posets, each of which corresponds to a permutahedron $P_n$.

\subsubsection{Total and partial orders}
For a given finite set $S$ with $|S|=n$ the linear orders on $S$ are in bijection with the set of bijective maps $\phi:[n]\to S$. In particular the linear orders on $[n]$ are in bijection with
 permutations $\sigma\in S_n$, the order being explicitly given by $\sigma_1< \cdots< \sigma_n$. We denote this linear order (total order) by $<_{\sigma}$.

Every rooted tree $\tau$ yields a partial order on its vertices by the height where the root is considered to be the lowest vertex.  The root is the unique minimal element and the leaves are the maximal elements.
For the trees in $\Tn$, by abuse of notation, we denote  by $\prec_{\tau}$ the induced partial order on the set of labels $[n]$ of the white white vertices.  We say $v \prec_{\tau} w$ if $w$ is above $v$. This is especially easy to read off the cactus picture.

\begin{definition}
On a given set $S$ a partial order $\prec$ is coarser than  $\prec'$,
if $a\prec b$ implies $a\prec'b$. If $\prec'$ is a total order $<$,
then we also say that $\prec$ is compatible with $<$.
\end{definition}
Notice that if $S$ is finite to show that $\prec$ is coarser that $\prec'$ one can simply check along
all maximal intervals $[s_1,s_2]$ w.r.t.\ $\prec$.

\subsubsection{The posets for $P_n$}

To describe the permutahedral structure of $Cact^1(n)$, for any $\sigma\in S_n$, we introduce the sub-poset $(\mathcal{T}_{\sigma},\precT)$ of $(\mathcal{T}_n,\precT)$ as follows.

\begin{definition}
\label{Tsigmadef}
The elements of $\mathcal{T}_{\sigma}$ are the trees in $\mathcal{T}_n$ such that $\prec_{\tau}$ is compatible with $<_{\sigma}$. The partial order of $\mathcal{T}_{\sigma}$ is the restriction of that of $\mathcal{T}_n$.
These sets inherit the degree splitting
 $\mathcal{T}^i_{\sigma}$, $i=0,1,\cdots,n-1$ of trees of  degree $i$.
\end{definition}

The  maximal intervals of $\prec_{\tau}$ correspond
the  leaf vertices and are given by the sequence of labels on the white vertices along the shortest path from the root vertex to this leaf vertex.
Thus $\prec_{\tau}$ is compatible with $<_\sigma$ if all these sequences are  subsequence of $\sigma_1\cdots\sigma_n$. Some examples of trees in $\mathcal{T}_{53214}$ are given in Figure \ref{Tsigmafig}.
\begin{figure}
		\centering
		\includegraphics[width=150mm]{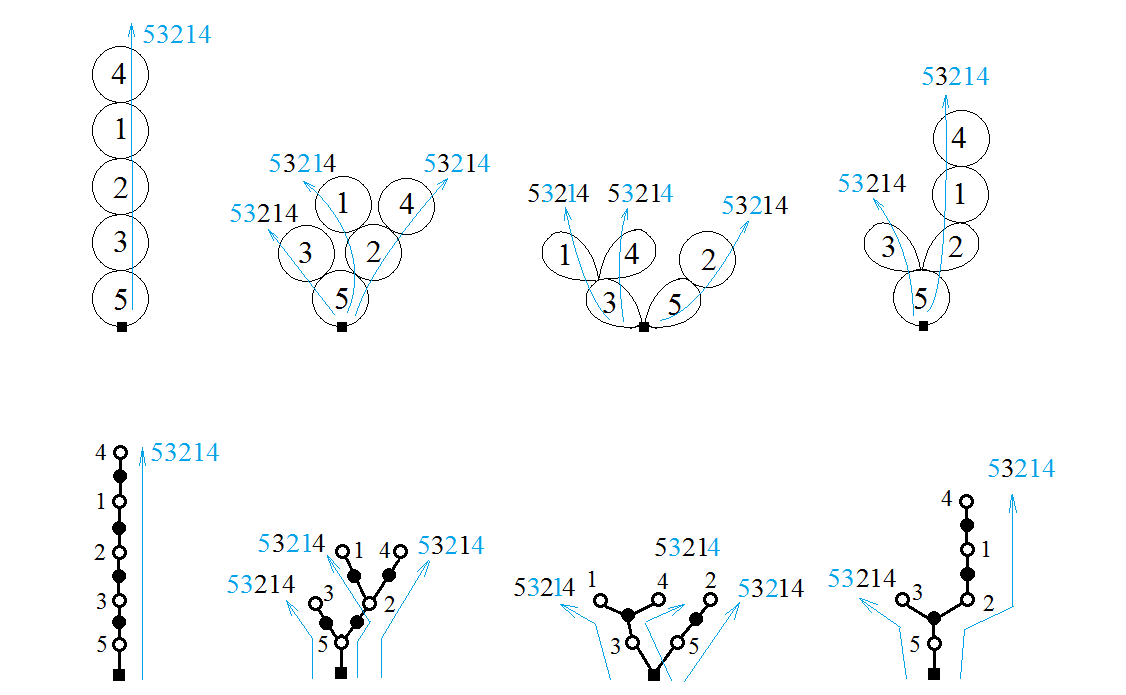}
		\caption{\label{Tsigmafig} Examples of trees in $\mathcal{T}_{53214}$ and the associated cacti pictures.}
\end{figure}
Notice that in $\T_\sigma$ there is a unique element which we call $\t_\sigma$ such
that the partial order $\prec_{\t_\sigma}=<_\sigma$.

Collapsing an angle between the leftmost/rightmost incoming edge with the outgoing edge of a white vertex makes the partial order on a tree coarser. Collapsing an angle between two adjacent incoming edges doesn't change the partial order on a tree. Thus, we have that:
\begin{lemma}
\label{mainlem}
The sub-posets $\T_\sigma$ are closed under angle collapses. That is if $\tau \in \T_\sigma$ and $\t'\prec \t$ then $\t'$ is also in $T_\sigma$\qed
\end{lemma}

In our later proofs, we also need trees whose white vertices are labelled with an arbitrary subset $S$ of $\Np$ and the corresponding orders, the generalization is intuitively clear from the example in Figure \ref{Tphifig}.

To be precise, we give the technical version. If $V_w$ is the set of white vertices, then a labelling $lab:S\subset \Np$ is a bijection $\phi:S\stackrel{\leftrightarrow}{\to} V_w$.  We let $\T_S$ be the set of  $S$--labelled planted planer b/w bipartite trees  with a black root and white leaves. If $|V_w|=n$, let $\phi:[n]\to S$ be a linear order on $S$. 

\begin{definition}
Given $S$ and an order $\phi$ on it we define the set $\T_{\phi}$ to be the subset of $\T_S$ of trees $\tau$ whose partial order $\prec_\tau$ is compatible with the order $<_\phi$.
\end{definition}
This directly generalizes Definition \ref{Tsigmadef} and Lemma \ref{mainlem} holds accordingly.

\textit{Example.} If $\phi:[3]\rightarrow S=
\{2,5,7\}\subset \mathbb{N}^+$ maps $1\mapsto 5$, $2\mapsto 7$ and $3\mapsto 2$, then we can consider the the $S-$labelled trees $\tau$ such that $\prec_\tau$ is compatible with $<_\phi$. This is depicted in Figure \ref{Tphifig}.

\begin{figure}
		\centering
		\includegraphics[width=150mm]{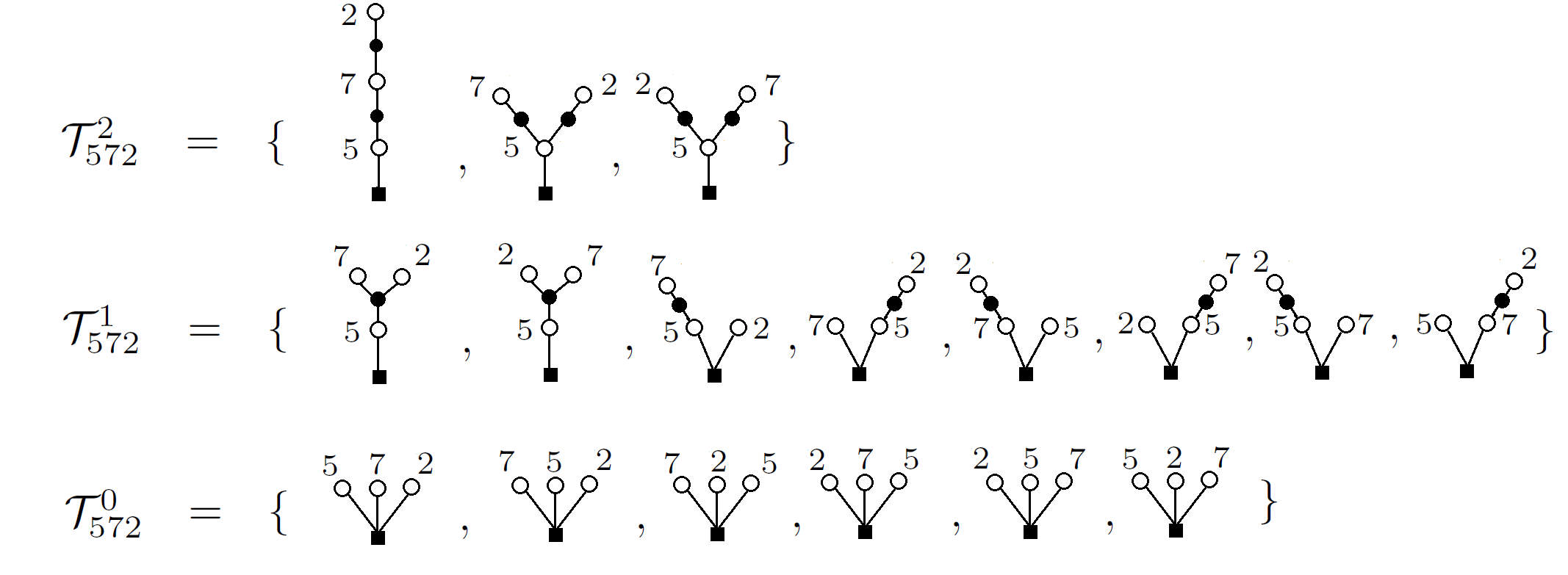}
		\caption{\label{Tphifig} The sets $\mathcal{T}^2_{572}$, $\mathcal{T}^1_{572}$, $\mathcal{T}^0_{572}$.}
\end{figure}

\subsubsection{Cutting and grafting trees: $B^{\pm}$ operators}
\label{Bpmpar}
There are two types of trees, those that have a unique lowest (i.e.\ closest to the root) white vertex, which we will call the white root. The set of these tree will be called the white rooted trees $\Twr$. The other type of tree has a several white vertices adjacent to the black root. These are,  by slight abuse of notation, called black rooted trees $\Tbr$. 
We will call {\it ordered collections} of such trees ``forests'' in $\T,\Tbr$ or $\Twr$. Here, we allow arbitrary labels on the white vertices.

\begin{definition}
The {\em initial branching number} of a tree $\t\in \Twr$ is the number of incoming edges of the unique white root.
\end{definition}

We will now define four operators:
\begin{enumerate}
\item $B^+_b$: ordered forests of $\T \to \T$.
This operation simply identifies all the black roots of the trees in the ordered forest into one black root. The linear order being the one coming from the trees and the order in the forrest. See Figure \ref{bplusfig} for an example.
\item $B^-_b$: $\T\to $ ordered forests in $\Twr$. This operations cuts all edges to the root vertex, takes the ordered collection of branches and puts one new black root on each branch. For an example, see Figure \ref{bminusfig}.
\item $B^{-}_w$: $\Twr\to$ ordered forests in $\T$. Cut off all the edges above the unique ${\it white}$ root vertex. Collect the branches in the order given by this white vertex, and add a black root to each of them. 
    
    NB: If one starts with an $\{\sigma\}$--labelled $\t\in \T_{\sigma}\cap \Twr$ and the white root is labeled by $\sigma_1$, then for some $\l_1,\dots,\l_k \in \mathrm{dSh}_{\sigma\backslash \sigma_1}[m_1,\dots, m_k]$: $B^-_w(\t)\in \T_{\l_1}\times \dots \times \T_{\l_k}$, where $k$ is the initial branching number of $\t$ and $m_i$ is the number of white vertices on the $i$-th branch.
\item $B^+_{s}$: $\T_{S_1}\times\dots\times \T_{S_k} \to T_{S}\cap \Twr$ whenever the $S_i$ are pairwise disjoint and none of them contain the singleton $\{s\}$. Here $S=\amalg_i S_i\amalg\{s\}$.
\begin{equation*}B^+_{s}(\tau_1,\tau_2,\cdots,\tau_k) =\{\text{$\tau$ obtained by grafting $\tau_1,\cdots,\tau_k$ to $scc(s)$}\} \subset \T_{n},
\end{equation*}
 where $scc(s)$
is only element in  $\T_{\{s\}}$. Here grafting means that each $\t_i$ is connected to the unique white vertex of $scc(s)$ by an additional edge in the order starting with $\t_1$. This is illustrated in Figure \ref{bpfig}.
We will use this operator when $(S_1,\dots, S_n)$ is a partition of the set $[n]\setminus \{s\}$.

\begin{figure}
		\centering
		\includegraphics[width=150mm]{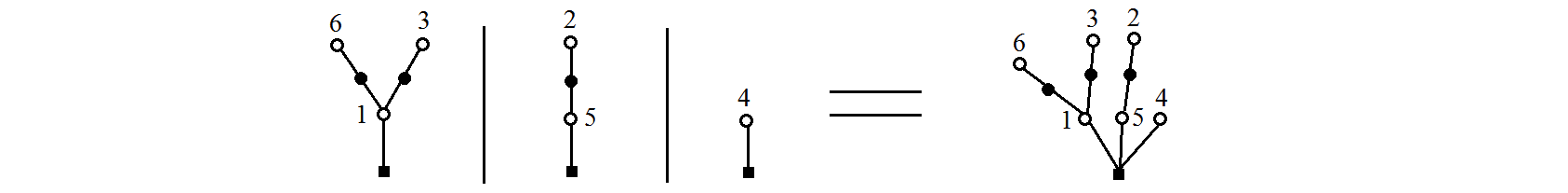}
\caption{\label{bplusfig}An example of $B_b^+$ and the bar notation}
\end{figure}
\begin{figure}
		\centering
		\includegraphics[width=150mm]{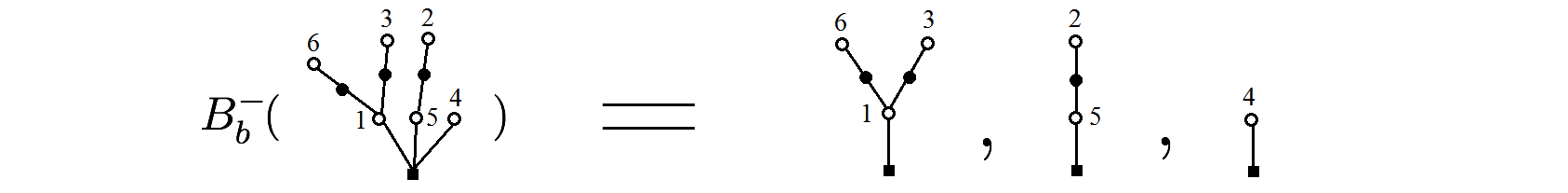}
\caption{\label{bminusfig} An example of $B^-_b$}
\end{figure}

\begin{figure}
		\centering
		\includegraphics[width=150mm]{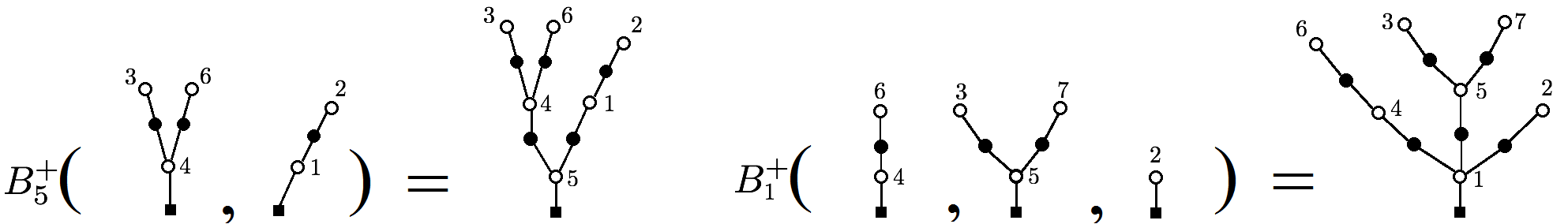}
		\caption{\label{bpfig} Two examples of $B^+_s$.}
\end{figure}

\end{enumerate}
\begin{notation}
To make contact with the permutahedra, especially the notation of \S\ref{permutasec}, 
we will use a vertical bar notation for the $B^+_b$ operator. That is, we will denote
$B^+_b(\t_1,\t_2,\dots,\t_k)$ by $\t_1|\t_2|\dots|\t_k$.  See Figure \ref{bplusfig}.
\end{notation}
\begin{remark}
\label{inversermk}
It is clear that $B^+_b$ and $B^-_b$ are inverses of each other.
Since a label is forgotten by $B^-_w$, $B_s^+$ is a left inverse for $B^-_w$ on the subset of $\Twr$, whose white roots are labelled by $s$.
Furthermore, $B^-_w$ is a left inverse for $B_s^+$ on the domain of definition of $B_s^+$.

Lastly if $s$ is not in the labelling set of $\t$: $B^+_sB^-_b$ switches the color of the root from black to white, labels it by $s$ and adds a new black root.
\end{remark}

\subsubsection{Decompositions and filtrations}

\label{decomppar}

$\T^{n-1}_n$ consists of the maximal elements in $\mathcal{T}_n$, i.e.\ exactly those elements that index the top-dimensional cells in $Cact^1(n)$. By \eqref{glueeq},  these cover $Cact^1$. These trees are all in $\Twr$, since
 otherwise, the tree would not have maximal degree.
 
To provide the setup for later inductive proofs, for each $\sigma\in S_n$, we will partition and then filter $\mathcal{T}_{\sigma}^{n-1}$ according to the initial branching number $k$. For trees in
$\T^{n-1}_{\sigma}$, $k$ can take values from $1$ to $n-1$. 
Let  $\mathcal{T}^{n-1}_{\sigma}(k)\subset\T^{n-1}_\sigma$ be the subset containing all the trees with initial branching number $k$. Then we have the following decomposition:
\begin{equation}
\mathcal{T}^{n-1}_{\sigma}=\coprod_{k=1}^{n-1}\mathcal{T}^{n-1}_{\sigma}(k).
\end{equation}
This decomposition gives rise to an ascending  filtration of
$\mathcal{T}^{n-1}_{\sigma}$:
\begin{equation}
\label{filtrationeq}
\mathcal{T}^{n-1}_{\sigma}(1)=\mathcal{T}^{n-1}_{\sigma,1}\subset \mathcal{T}^{n-1}_{\sigma,2}\subset\cdots \subset \mathcal{T}^{n-1}_{\sigma,n-2}\subset\mathcal{T}^{n-1}_{\sigma,n-1}=\mathcal{T}^{n-1}_{\sigma},
\end{equation}
where
\begin{equation}
\T^{n-1}_{\sigma,k}=\coprod_{q\leq k}\mathcal{T}^{n-1}_{\sigma}(q).
\end{equation}

We can further decompose each $\mathcal{T}^{n-1}_{\sigma}(k)$ using the $B^-_w$ or the $B_{\sigma_1}^{+}$ operator. The following observation is the key: since the $B_{\sigma_1}^+$ operator lands in $\Twr$, it is in general not surjective, but
it is surjective on the top degree trees.

\begin{definition}

Fix $\sigma\in S_n$ , $k\in \Np$ with $1\leq k\leq n-1$, and $m_1,m_2,\cdots, m_k$ be $k$ positive integers such that $m_1+\cdots +m_k=n-1$.
 Let $\mathbf{l}=\mathbf{l}_1,\mathbf{l}_2,\cdots,\mathbf{l}_k\in \mathrm{dSh_{\sigma \backslash \sigma_1}}[m_1,m_2,\cdots,m_k]$.

  We define $\mathcal{T}^{n-1}_{\sigma}[\mathbf{l}]$
to be the set of all trees $B^+_{\sigma_1}(\tau_1,\tau_2,\cdots,\tau_k)$ in $\mathcal{T}_{\sigma}$ obtained by grafting $\tau_1,\cdots,\tau_k$, with $\tau_i\in \T_{\mathbf{l}_i}^{m_i-1}$ to $scc(\sigma_1)$. Since, the order of the branches is recorded, it follows that indeed the image under
$B^+_{\sigma_1}$ is in $\mathcal{T}^{n-1}_{\sigma}$  
and furthermore $\t\in \T_{\sigma}^{n-1}[\l]$ if and only if
$B^-_w(\tau)\in \T_{\mathbf{l}_1}^{m_1-1}\times\dots\times  \T_{\mathbf{l}_k}^{m_k-1}$.

To extend this decomposition to all degrees,
we now define $\mathcal{T}_{\sigma}[\mathbf{l}]$ be the subset of $\mathcal{T}_{\sigma}$ such that each element in $\mathcal{T}_{\sigma}[\mathbf{l}]$ is less than or equal to an element in $\mathcal{T}^{n-1}_{\sigma}[\mathbf{l}]$. Similarly, we define the pieces of the filtration $\mathcal{T}_{\sigma,k}$.

 Since angle collapse only potentially decreases the initial branching number, we also have the inherited poset structures on $\mathcal{T}_{\sigma}[\mathbf{l}]$ and $\mathcal{T}_{\sigma,k}$.

\end{definition}

\textit{Example.} The elements of $\mathcal{T}^5_{532146}[\mathbf{l}_1,\mathbf{l}_2]$ where $\mathbf{l}_1=36$, $\mathbf{l}_2=214$ are shown in Figure \ref{aritytwofig}.
\begin{figure}
		\centering
		\includegraphics[width=150mm]{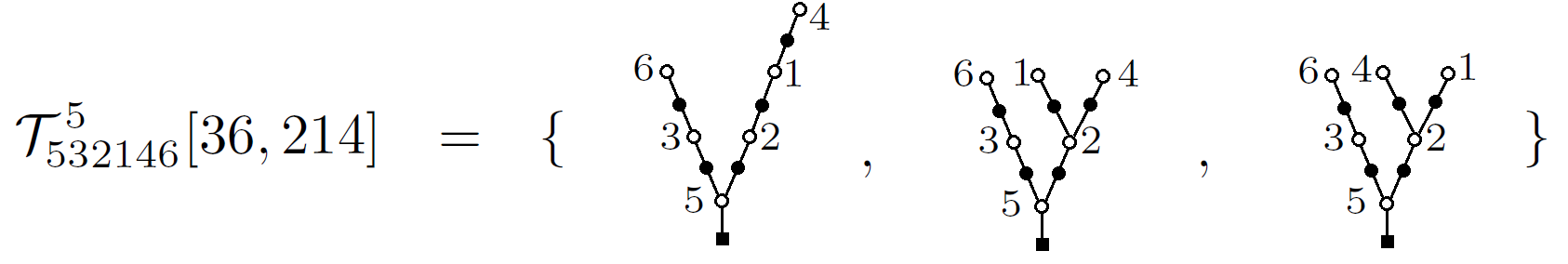}
		\caption{\label{aritytwofig} The elements of $\mathcal{T}^5_{532146}[\mathbf{l}]$ where $\mathbf{l}=\mathbf{l}_1,\mathbf{l}_2$ and $\mathbf{l}_1=36$, $\mathbf{l}_2=214$.}
\end{figure}

Summing up, we have the decomposition

\begin{equation}
\label{kdecompeq}
\mathcal{T}^{n-1}_{\sigma}(k)=\coprod_{m_1,\cdots,m_k}\coprod_{\mathbf{l}\in \mathrm{dSh}_{\sigma\backslash \sigma_1}[m_1,\cdots,m_k]}\mathcal{T}_{\sigma}^{n-1}[\mathbf{l}]
\end{equation}
and
\begin{equation}
\label{filtrationeq2}
\mathcal{T}_{\sigma,1}\subset \mathcal{T}_{\sigma,2} \subset\dots \subset \mathcal{T}_{\sigma,n-2}\subset \mathcal{T}_{\sigma,n-1}.
\end{equation}

The realization functor $C$ on $\mathcal{T}_n$ restricts to $(\mathcal{T}_{\sigma},\precT)$, $(\mathcal{T}_{\sigma}[\mathbf{l}],\precT)$ and $(\mathcal{T}_{\sigma,k},\precT)$, respectively.
\subsection{Permutohedral covering of  $Cact^1(n)$.}
We can now prove that indeed $Cact^1(n)$ is covered by $n!$ permutahedra $P_{\sigma}:=P_n$, $\sigma\in S_n$ in $Cact^1(n)$ as shown below.

\begin{theorem}
\label{CactP_n}
For any $\sigma\in S_n$, $\mathrm{colim}_{\mathcal{T}_{\sigma}}C$ is a polytope, which is piecewise linearly homeomorphic ($\cong$) to $P_n$.
\end{theorem}

\begin{proof}
We proceed by nested induction. When $n=1,2$, $\mathrm{colim}_{\mathcal{T}_{\sigma}}C$ are a point and a closed line segment, respectively. So the statement is true in these two cases.

Suppose the statement is true for all $m$ and all $\sigma\in S_m$ where $m<n$.
Let $\sigma \in S_n$. We will first show that $\mathrm{colim}_{\mathcal{T}_{\sigma}}C$ is a PL (piecewise linear) cell of dimension $n-1$. We will simply say that $\mathrm{colim}_{\mathcal{T}_{\sigma}}C$ is a PL $D^{n-1}$.

 We will iteratively use the following observation. The connected sum of two PL $D^{n-1}$'s along a sub PL $D^{n-2}$ is a PL $D^{n-1}$.
More precisely: if $X$ and $Y$ are both PL $D^{n-1}$'s and $i:D^{n-2}\hookrightarrow X$ and $j:D^{n-2}\hookrightarrow Y$ are injective PL maps such that $i(D^{n-2})$ is the connected union of some facets of $X$ and $j(D^{n-2})$ is the connected union of some facets of $Y$ (so both $i(D^{n-2})$ and $j(D^{n-2})$ are PL $D^{n-2}$), then the glued object (pushout of $X\hookleftarrow D^{n-2} \hookrightarrow Y$) is again a PL $D^{n-1}$.

Also, notice that by the induction hypothesis and the definition of the realization functor $C$, for $\mathbf{l}\in \mathrm{dSh}_{\sigma\backslash \sigma_1}[m_1,\cdots,m_k]$,

\begin{equation}
\label{celleq}
\mathrm{colim}_{\mathcal{T}_{\sigma}[\mathbf{l}]}C \cong P_{m_1}\times P_{m_2}\times \cdots \times P_{m_k}\times \Delta^k
\end{equation}
So $\mathrm{colim}_{\mathcal{T}_{\sigma}[\mathbf{l}]}C$ is a PL $D^{n-1}$.

We now use a second induction on $k$, to show that
\begin{equation}
\label{claim1}
\mathrm{colim}_{\mathcal{T}_{\sigma,k}}C\mbox{ is a PL } D^{n-1}.
\end{equation}

When $k=1$ we know that
$\T^{n-1}_{\sigma,1}=\T^{n-1}_{\sigma}(1)=\T^{n-1}_{\sigma}[\mathbf{l}]$,
with $\mathbf{l}=\sigma_2\cdots\sigma_n$, (see \S\ref{decomppar}),
and hence $\mathrm{colim}_{\mathcal{T}_{\sigma,1}}C=\mathrm{colim}_{\mathcal{T}_{\sigma}[\mathbf{l}]}C\cong P_{n-2}\times \Delta^1$ is a PL $D^{n-1}$.

Now suppose for $2\leq k\leq n-1$, $\mathrm{colim}_{\mathcal{T}_{\sigma,k-1}}C$ is a PL $D^{n-1}$.

For each $\mathbf{l}\in \mathrm{dSh}_{\sigma\backslash \sigma_1}[m_1,\cdots,m_k]$, $\mathrm{colim}_{\mathcal{T}_{\sigma}[\mathbf{l}]}C$, which is a PL $D^{n-1}$, is glued to the PL $D^{n-1}$ given by  $\mathrm{colim}_{\mathcal{T}_{\sigma,k-1}}C$ along $P_{m_1}\times \cdots P_{m_k}\times \bigcup_{i=2}^k\del_i\Delta^k$.
Here $\del_i$ is the $i$--th face map which on the simplex in the vertex notation $\Delta^k=v_1\cdots v_{k+1}$ can be written as $v_1\cdots \widehat{v_i}\cdots v_{k+1}$. In the cactus picture, this corresponds to the contraction of the $i$-th arc on the root lobe. Notice that since $\partial \Delta^k=\bigcup_{i=1}^{k+1}v_1\cdots \widehat{v_i}\cdots v_{k+1}$ is a PL
$S^{k-1}$, $(v_2v_3\cdots v_{k+1})\bigcup(v_1v_2\cdots v_k)$ is a PL $D^{k-1}$ and $\bigcup_{i=2}^kv_1\cdots \widehat{v_i}\cdots v_{k+1}$ is also a PL $D^{k-1}$. Thus, $P_{m_1}\times \cdots P_{m_k}\times (\bigcup_{i=2}^kv_1\cdots \widehat{v_i}\cdots v_{k+1})$ is a PL $D^{n-2}$. 
And hence we are gluing two PL $D^{n-1}$'s along a common PL $D^{n-2}$ and the result is
a PL $D^{n-1}$. This is true for each $\mathbf{l}\in S_{\sigma}[m_1,\cdots,m_k]$ in \eqref{kdecompeq} individually, so we can glue in these $\mathrm{colim}_{\mathcal{T}_{\sigma}[\mathbf{l}]}C$ one by one and end up with a PL $D^{n-1}$ and obtain \eqref{claim1}.

From this it follows that: $\mathrm{colim}_{\mathcal{T}_{\sigma}}C=\mathrm{colim}_{\mathcal{T}_{\sigma,n-1}}C$ is a PL $D^{n-1}$,  by applying $C$ to the filtration \eqref{filtrationeq2}.
Indeed, we have the following filtration of the PL cell $\mathrm{colim}_{\mathcal{T}_{\sigma,n-1}}C$ by PL cells: $$\mathrm{colim}_{\mathcal{T}_{\sigma,1}}C\subset \mathrm{colim}_{\mathcal{T}_{\sigma,2}}C\subset \cdots \subset \mathrm{colim}_{\mathcal{T}_{\sigma,n-2}}C\subset \mathrm{colim}_{\mathcal{T}_{\sigma,n-1}}C.$$
An example is illustrated in Figure \ref{gluepnfig}.

\begin{figure}
		\centering
		\includegraphics[width=150mm]{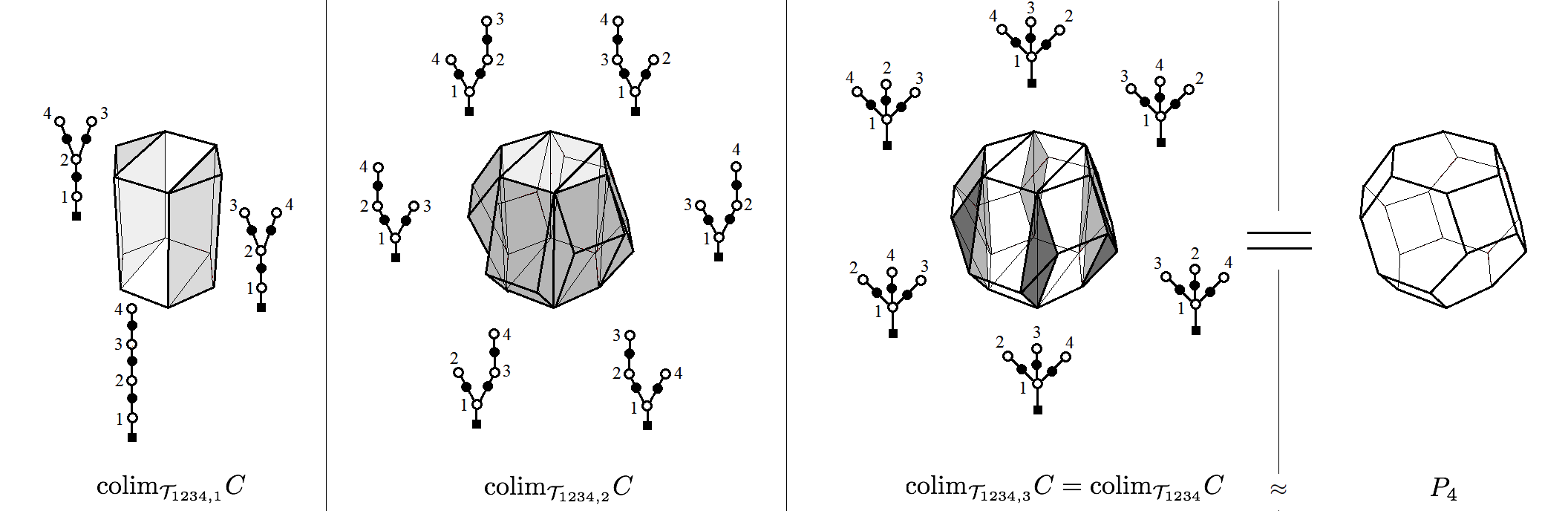}
		\caption{\label{gluepnfig}$\mathrm{colim}_{\mathcal{T}_{1234,i}}C$, $i=1,2,3$.}
\end{figure}

Next, we show that the PL cell $\mathrm{colim}_{\mathcal{T}_{\sigma}}C$ is indeed piecewise linearly isomorphic to $P_n$. Let us define a new functor. For any $\sigma\in S_n$, let $\mathcal{C}$ be the realization functor from $\mathcal{T}_{\sigma}$ to the category of PL topological spaces defined by $$\mathcal{C}_{scc(\sigma)}=\mathbf{v}_{\sigma}\in \mathbb{R}^n$$ on degree $0$ elements and $\mathcal{C}_{\tau}$ to be the convex hull of $\{\mathcal{C}_{\tau'}|\tau'\in \mathcal{T}^0_{\sigma},\tau'\precT \tau\}$ for general $\tau\in \mathcal{T}^{i}_{\sigma}$ where $i>0$. Again,  the image of $\precT$ under $\mathcal{C}$ are defined to be be face inclusions.

There is hence is a piecewise linear homeomorphism from $\mathrm{colim}_{\mathcal{T}_{\mathbf{l}_i}}C$ to $\mathrm{colim}_{\mathcal{T}_{\mathbf{l}_i}}\mathcal{C}$ by extending the vertex correspondences $C_{scc(\sigma)}\mapsto \mathcal{C}_{scc(\sigma)}$. It remains to identify the face structure.

Each cell on the boundary of $\mathrm{colim}_{\mathcal{T}_{\sigma}}C$ is indexed by a tree obtained as $B_b^+(\tau_1,\tau_2,\cdots,\tau_k)$, where $\tau_i \in \mathcal{T}^{m_i-1}_{\mathbf{l}_i}$ such that $\mathbf{l}_1,\mathbf{l}_2,\cdots,\mathbf{l}_k\in \mathrm{dSh}_{\sigma}[m_1,m_2,\cdots,m_k]$. As mentioned previously, we denote such a tree by  $\tau_1|\cdots|\tau_k$.

Let $\mathcal{T}_{\mathbf{l}_1}|\mathcal{T}_{\mathbf{l}_2}|\cdots|\mathcal{T}_{\mathbf{l}_k}=
B^+_b(\mathcal{T}_{\mathbf{l}_1}\times \dots\times \mathcal{T}_{\mathbf{l}_k})
=\{\tau_1|\tau_2|\cdots|\tau_k : \tau_i\in \mathcal{T}_{\mathbf{l}_i}\}$. We shall consolidate the cells indexed by all $\tau_1|\tau_2|\cdots|\tau_k\in \mathcal{T}^{m_1-1}_{\mathbf{l}_1}|\mathcal{T}^{m_2-1}_{\mathbf{l}_2}|\cdots | \mathcal{T}^{m_k-1}_{\mathbf{l}_k}$ together to form the faces.
We can then again use induction on $n$ as previously. Namely, by the induction hypothesis and the way that $B_b^+$ is defined, we know for each $\mathbf{l}_1|\mathbf{l}_2|\cdots|\mathbf{l}_k$, $\mathrm{colim}_{\mathcal{T}_{\mathbf{l}_1}| \mathcal{T}_{\mathbf{l}_2}|\cdots |\mathcal{T}_{\mathbf{l}_k}}\mathcal{C}=P_{m_1}\times P_{m_2} \times \cdots \times P_{m_k}$. But this is the characterization of the cells of $P_n$.
Therefore, $\mathrm{colim}_{\mathcal{T}_{\sigma}}\mathcal{C}=P_n$ and thus $\displaystyle \mathrm{colim}_{\mathcal{T}_{\sigma}}C \cong P_n$.

 Notice, that the colimits, can be taken before realization, and all the combinatorics can also be taken on the level of polytopes.
This gives the strengthening of the statement.
\end{proof}

\begin{remark}
Let $\sigma\in S_n$, since $\mathrm{colim}_{\mathcal{T}_{\sigma}}C\cong P_n$, we say that $P_n$ has the decomposition into cactus cells (products of simplices) associated to $\sigma$.

 For $n\geq 2$, we have $n!/2$ different decompositions of $P_n$. The number is $n!/2$ instead of $n!$ because $\sigma$ and $\sigma\circ s$ give the same decomposition, where $s:[n]\rightarrow [n]$ is defined by $s\big|_{[n-2]}=\mathrm{id}_{[n-2]}$, $s_{n-1}=n$ and $s_n=n-1$.
\end{remark}

\subsection{Further consequences. Recursion and Dyer--Lashof operations}

\subsubsection{Operadic generation and Dyer--Lashof operations}

The indexing set $\mathcal{T}^{n-1}_{12\cdots n}$ of the top-dimensional cells of the decomposition of $P_n$ associated to $12\cdots n$ can be generated from the single tree $B^+_{1}(scc(2))$ with two white vertices using the operadic composition $\circ_1$  for the cellular chain operad $CC_*(Cact^1)$. This observation allows us to link permutahedra to Dyer--Lashof operations.

Let $c=\sum_{i\in I}n_ic_i$, where $n_i\in \mathbb{Z}$, be a chain in $CC_{n-1}(Cact^1(n))$ for some $n$. Let $\{c\}$ be the set of support of $c$, i.e., $\{c\}=\{c_i\big| n_i\neq 0\}$.  Let $\tau=B^+_1(scc(2))$. Define $\{\tau\circ_1\mathcal{T}^{i-1}_{12\cdots i}\}$ to be the (disjoint) union of the sets $\{\tau\circ_1 \tau'\}$ where $\tau'\in \mathcal{T}^{i-1}_{12\cdots i}$. It can be readily checked that the following holds.

\begin{lemma}
$\{\tau\circ_1\mathcal{T}^{i-1}_{12\cdots i}\}=\mathcal{T}^i_{12\cdots (i+1)}$.\qed
\end{lemma}

\begin{theorem}
Let $\tau^{n-1}=\underbrace{\{\tau\circ_1\{\cdots\{\tau\circ_1\{\tau\circ_1\{\tau\}\}\}\cdots\}\}}_{\text{There are $n-1$ $\tau$ and thus $n-2$ $\circ_1$.}}$. Then $\tau^{n-1}=\mathcal{T}^{n-1}_{12\cdots n}$ and moreover, the multiplicity of each summand in 
$\tau\circ_1(\cdots(\tau\circ_1(\tau\circ_1(\tau)))\cdots))$
is $1$. So $\tau^{n-1}$ indexes the top-dimensional cells of the decomposition of $P_n$ associated to $12\cdots n$. This is the cell for the Dyer--Lashof operation.
\end{theorem}
\begin{proof}
By iterating $i=2,3,\cdots,n-1$, where $n\geq 3$, we see that indeed, we get all the cells indexing $P_{id}=P_n$. By \cite{Kau10}[Proposition 2.13] this iteration also has coefficients $1$ and yields the cell for the Dyer--Lashof operation.
\end{proof}

\begin{remark}
This also allows us to give a concrete homotopy between the right iteration above and the left iteration  $$\underbrace{(\tau\circ_2(\cdots(\tau\circ_2(\tau\circ_2(\tau)))
\cdots))}_{\text{There are $n-1$ $\tau$ and thus $n-2$ $\circ_2$}}=
\t_{id}=B^+_{\sigma_1}(B^+_{\sigma_2}(\cdots B^+_{\sigma_{n-1}}(scc(\sigma_n))\cdots))$$
Here  the support is the single tree $\t_{id}$ in $\T_{id}$, whose cell is the hypercube $I^{n-1}$ that sits at the center of the permutahedron $P_{\sigma}$. 

\end{remark}

\subsubsection{Iterative decomposition into cactus cells}
There is an interesting duality in the cactus decomposition. 

On one hand,
recall from the proof of Theorem  \ref{CactP_n}, that each codim $k-1$ face of $P_n$ is labelled by 
$\mathcal{T}^{m_1-1}_{\mathbf{l}_1}|\mathcal{T}^{m_2-1}_{\mathbf{l}_2}|\cdots | \mathcal{T}^{m_k-1}_{\mathbf{l}_k}$ and the subdivision is given by the elements of this set. More precisely:  for $n\geq 2$, $1\leq k\leq n-1$, fix $m_1,m_2,\cdots,m_k$ satisfying $m_i\geq 1$ and $m_1+m_2+\cdots+m_k=n-1$ and let $\mathbf{l}=\mathbf{l}_1,\mathbf{l}_2,\cdots,\mathbf{l}_k\in \mathrm{dSh}_{\sigma}[m_1,m_2,\cdots,m_k]$.
Then the elements in $\mathcal{T}^{m_1-1}_{\mathbf{l}_1}|\mathcal{T}^{m_2-1}_{\mathbf{l}_2}|\cdots | \mathcal{T}^{m_k-1}_{\mathbf{l}_k}$ are $\tau_1|\tau_2|\cdots|\tau_k$ where each $\tau_i\in \mathcal{T}_{\mathbf{l}_i}^{m_i-1}$ is a tree with the maximal number  ($m_i-1$) of white edges and its partial order is compatible with the total order $\mathbf{l}_i$.

On the other hand, recall from \eqref{kdecompeq}, the top cells of $P_{\sigma}$ are naturally indexed by the fibers of $B^-_{\sigma_1}$
$$
\mathcal{T}^{n-1}_{\sigma}(k)=\coprod_{m_1,\cdots,m_k}\coprod_{\mathbf{l}\in \mathrm{dSh}_{\sigma\backslash \sigma_1}[m_1,\cdots,m_k]}\mathcal{T}_{\sigma}^{n-1}[\mathbf{l}]
$$

To sum this up, 
for fixed $k$, we define 
$\mathcal{T}^{face}_{\sigma\backslash \sigma_1}(k)$ as follows.
$$\mathcal{T}^{face}_{\sigma\backslash \sigma_1}(k)=\coprod_{m_1+\cdots+m_k=n-1}\coprod_{\mathbf{l}_1,\cdots,\mathbf{l}_k \in \mathrm{dSh}_{\sigma\backslash \sigma_1}[m_1,\cdots,m_k]} \mathcal{T}_{\mathbf{l}_1}^{m_1-1}|\mathcal{T}_{\mathbf{l}_2}^{m_2-1}|\cdots|\mathcal{T}_{\mathbf{l}_k}^{m_k-1}\subset 
T_{\sigma\backslash \sigma_1}^{n-1-k}$$
This is the set of trees indexing all the cells making up the codim $k-1$--faces of $P_{\sigma\backslash\sigma_1}$. Then  we have the diagram:
\begin{equation}
\xymatrix{
\mathcal{T}_{\mathbf{l}_1}^{m_1-1}|\mathcal{T}_{\mathbf{l}_2}^{m_2-1}|\cdots|\mathcal{T}_{\mathbf{l}_k}^{m_k-1}\ar@{^{(}->}[d]\ar@<-1ex>[r]_{B^-_b}&
\ar@<-1ex>[l]_{B_b^+}\mathcal{T}_{\mathbf{l}_1}^{m_1-1}\times \mathcal{T}_{\mathbf{l}_2}^{m_2-1}\times \cdots\times \mathcal{T}_{\mathbf{l}_k}^{m_k-1}
\ar@<1ex>[r]^>>>>{B^+_{\sigma_1}}&\ar@<1ex>[l]^<<<<{B^-_{\sigma_1}}\mathcal{T}_{\sigma}^{n-1}[\mathbf{l}]\ar@{^{(}->}[d]\\
\mathcal{T}^{face}_{\sigma\backslash \sigma_1}(k)\ar[rr]^{B_{\sigma_1}^+\circ B_b^-}&& \mathcal{T}^{n-1}_{\sigma}(k)
}
\end{equation}

We  set $\mathcal{T}^{face}_{\sigma\backslash \sigma_1}=\coprod_{k=1}^{n-1}\mathcal{T}^{face}_{\sigma\backslash \sigma_1}(k)$, which is the set indexing all the cells making up all  faces of $P_{\sigma\backslash\sigma_1}$

\begin{proposition}
The map obtained by taking the disjoint union over $k$ of the lower arrows $B_{\sigma_1}^+\circ B_b^-:\mathcal{T}^{face}_{\sigma\backslash \sigma_1}(k)\to  \mathcal{T}^{n-1}_{\sigma}(k)$,
is a bijection:
$B_{\sigma_1}^+\circ B_b^-:\mathcal{T}^{face}_{\sigma\backslash \sigma_1}\to  \mathcal{T}^{n-1}_{\sigma}$.
\end{proposition}
\begin{proof}
From Remark \ref{inversermk}, we see that in the upper row all the arrows are bijections and this proves the claim.
\end{proof}

The elements in the codomain $\mathcal{T}^{n-1}_{\sigma}$ of $B_{\sigma_1}^+\circ B_b^-$ label the top dimensional cells of the decomposition of $P_n$ into cactus cells. The above proposition means the top dimensional cells of $P_n$ can instead be labelled by the the top dimensional cells of the decomposition of each face of $P_{n-1}$. The figure below uses color to illustrate this from $P_{n-1}$ to $P_n$ for $n=2,3,4$.\\

\begin{figure}[!h]
		\centering
		\includegraphics[width=150mm]{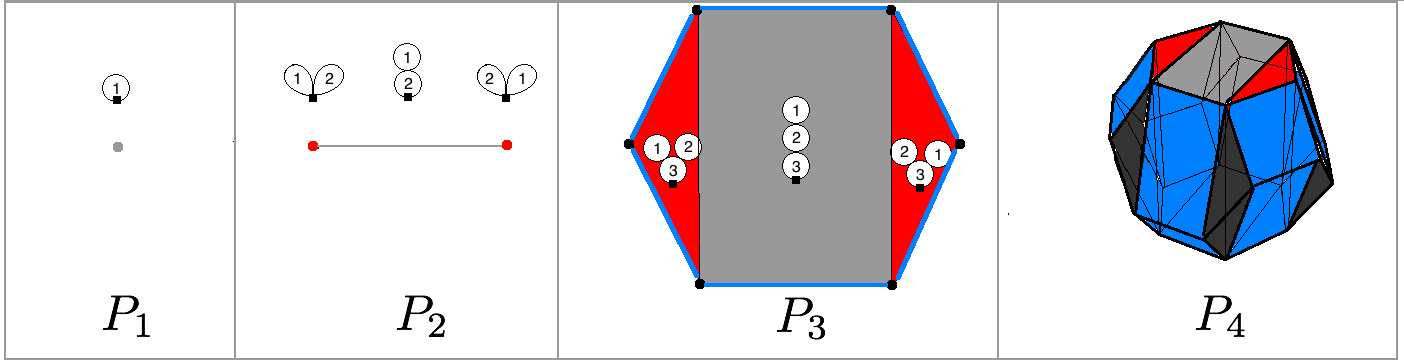}
		\caption{The subdivisions of $P_1$ by $1$, $P_2$ by $21$, $P_3$ by $321$ and $P_4$ by $4321$}
\end{figure}

Even though we are not able to draw the subdivision of $P_5$, we can at least compute the number of top-dimensional cells of it $|\mathcal{T}^4_{54321}|$ using the above bijection, where $|X|$ denote the number of elements of the set $X$. We have

\begin{itemize}
\item $|\mathcal{T}^{face}_{4321}(1)|=|\mathcal{T}^3_{4321}|=15$.
\item $|\mathcal{T}^{face}_{4321}(2)|=30$.
\item $|\mathcal{T}^{face}_{4321}(3)|=36$.
\item $|\mathcal{T}^{face}_{4321}(4)|=4!$.
\end{itemize}

So $|\mathcal{T}^4_{54321}|=\sum_{k=1}^4|\mathcal{T}^{face}_{4321}(k)|=105$.

\subsubsection{Remark}

Our construction is related to a statement \cite{Ber98} [Remark 1.10] .

\textit{``Jim McClure and Jeff Smith construct an $E_2$-operad which acts on topological Hochschild cohomology ... Its multiplication uses prismatic decomposition of the permutahedra $P_k$ (labelled by ``formulae'') which can be described as follows: The image of $P_2\times P_1\times P_{k-1}\rightarrow P_k$ is a prism $\Delta^1\times P_{k-1}$, thus by induction endowed with a prismatic decomposition; it turns out that the (closure of the) complement of the image also admits a prismatic decomposition labelled by the set of proper faces of $P_{k-1}$ ...''}

Namely, the above comment is almost true. It is true that $P_n$ can first be decomposed into two parts: $P_{n-1}\times I$ where $I$ is a closed interval of length $\sqrt{n(n-1)}$  and the closure of the complement of $P_{n-1}\times I$ in $P_n$, then $P_{n-1}\times I$ has the decomposition induced from that of $P_{n-1}$. But the closure of the complement of $P_{n-1}\times I$ in $P_n$ has the decomposition into pieces not labelled by the proper faces of $P_{n-1}$, but by the top dimensional cells from subdivisions of each proper face of $P_{n-1}$.

\subsection{The Permutaheral cover of $Cact^1(n)$}

\begin{definition}
We extend $\mathcal{C}$ from  $\mathcal{T}_{\sigma}$ to $\mathcal{T}_n$ and then let $$\mathcal{C}(n):=\mathrm{colim}_{\mathcal{T}_n}\mathcal{C}.$$ By construction, the resulting space is homeomorphic to $Cact^1(n)$, that is there is a homeomorphism $L_n:\mathcal{C}(n)\approx Cact^1(n)$.
This homeomorphism is actually almost the identity. It is just two different realizations of the same complex, which is why we will write 
$$\mathcal{C}(n)= Cact^1(n).$$
\end{definition}

\begin{proposition}
$Cact^1$ is a quotient of the permutahedral space $\coprod_{\sigma\in S_n}P_n$.
 \end{proposition}
 \begin{proof}
By taking the colimit iteratively that is first over each $T_{\sigma}$ and then gluing the resulting spaces further and using  Theorem \ref{CactP_n}, we can write 
\begin{equation}
Cact^1(n)=\mathcal{C}(n)=\left(\coprod_{\sigma\in S_n}P_n\right)/\sim_{\mathcal{C}},
\end{equation}
Here explicitly, for $P_n$ indexed by $\sigma$,  the subdivision is indexed by elements in $\mathcal{T}_{\sigma}$ and for $x\in P_n$ indexed by $\sigma$ and $y\in P_n$ indexed by $\nu$, $x\sim_{\mathcal{C}}y$ if there is $\tau\in \mathcal{T}_{\sigma}\cap \mathcal{T}_{\nu}$ such that $x=y$ in $\mathcal{C}_{\tau}$. 
 \end{proof}
 Examples when $n=3$ and $n=4$ are shown in Figure \ref{fig:6P3} and Figure~\ref{fig:24P4}, respectively.

\begin{figure}
		\centering
		\includegraphics[width=150mm]{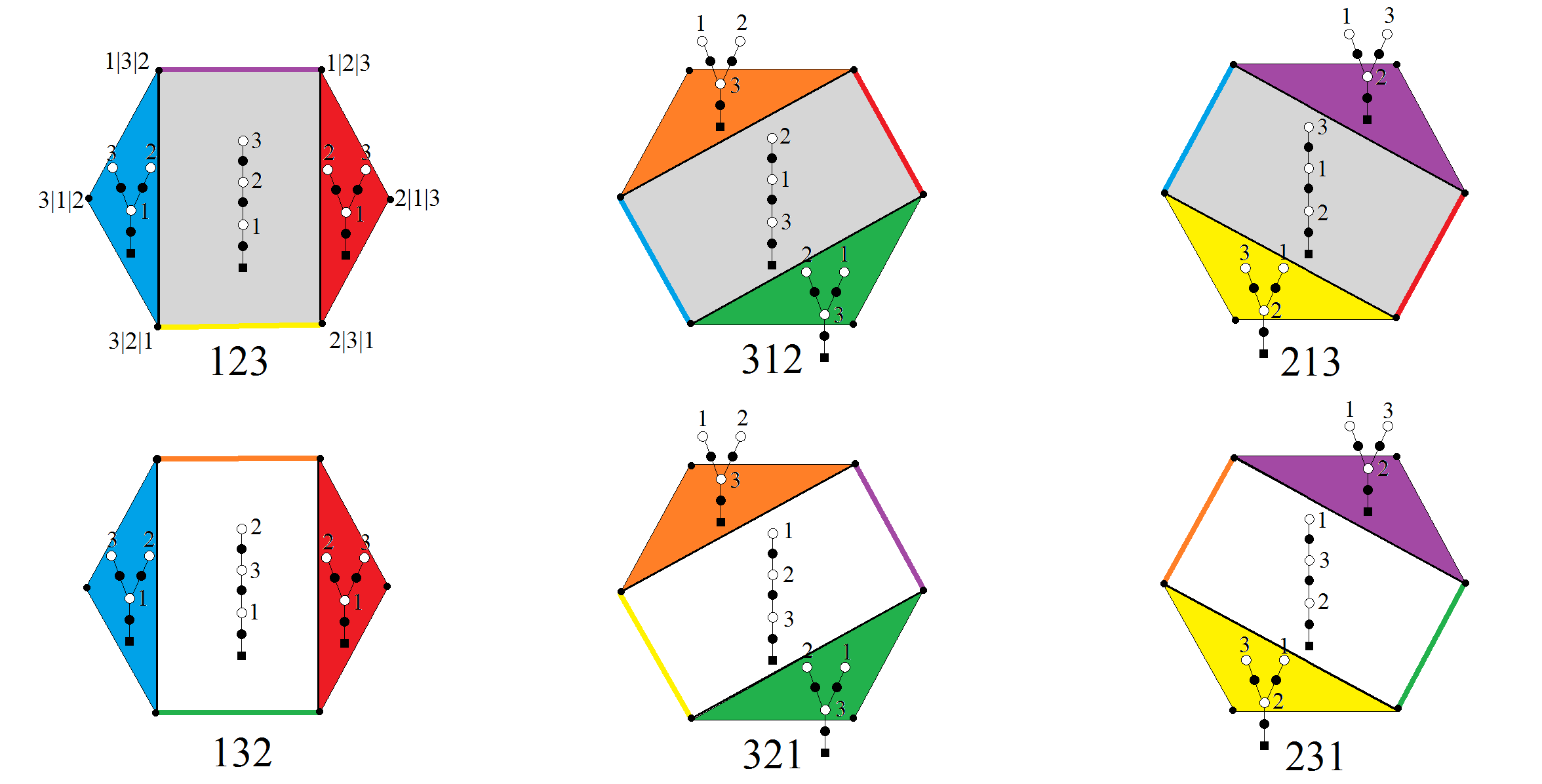}
		\caption{  \label{fig:6P3}$\mathcal{C}(3)$ is obtained by gluing $6$ copies of $P_3$, one for each $\sigma\in P_3$. For simplicity, the indexing elements from $\mathcal{J}^0_{\sigma}$ for the vertices are only shown for the first $P_3$ ($\sigma=123$). The points that are to be glued are labelled by the same color and put in the same position.}
\end{figure}

\begin{figure}
		\centering
		\includegraphics[width=150mm]{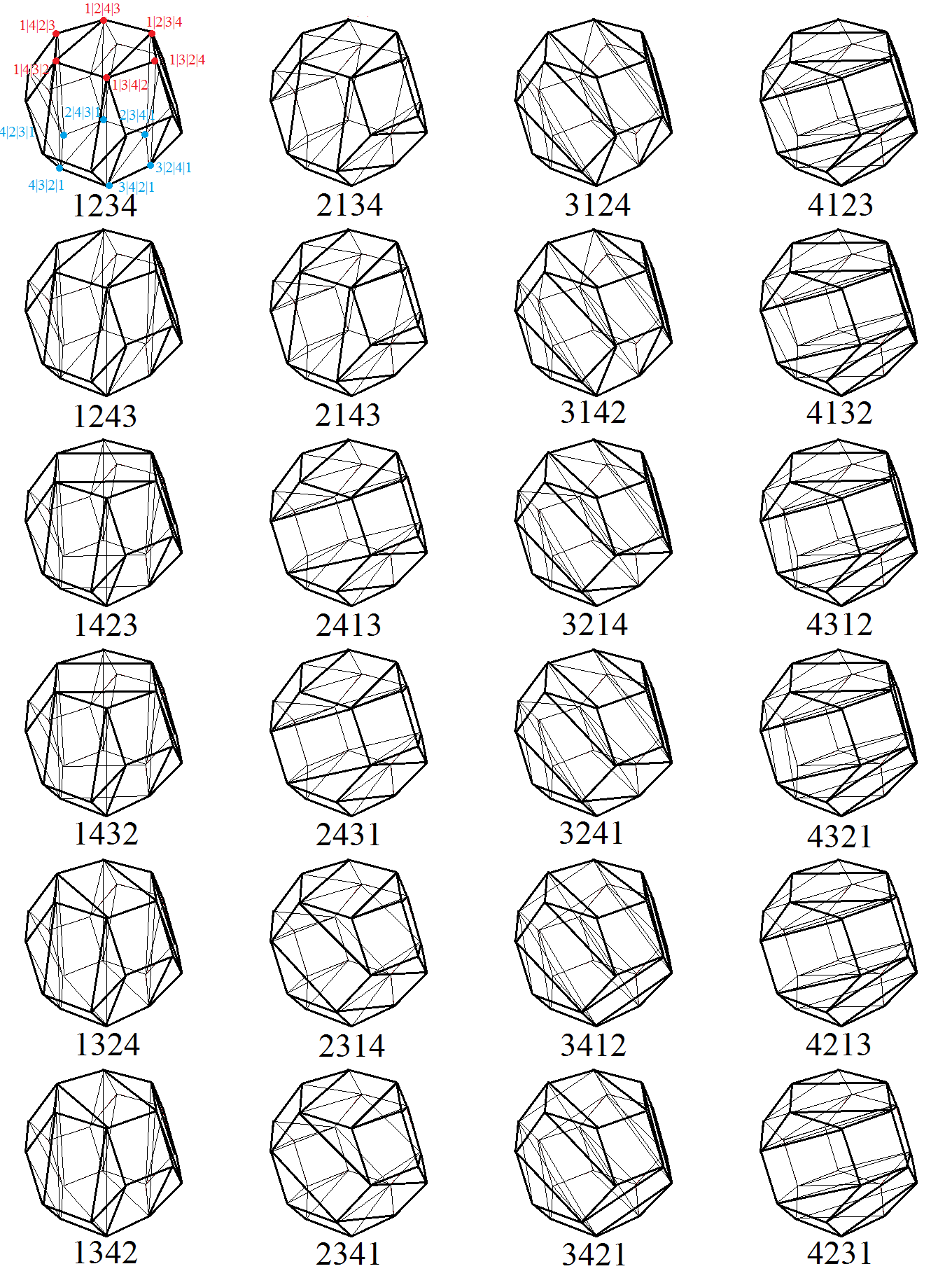}
		\caption{\label{fig:24P4}$\mathcal{C}(4)$ is obtained by gluing $24$ copies of $P_4$, one for each $\sigma\in P_4$. For simplicity, only twelve of the indexing elements from $\mathcal{J}^0_{\sigma}$ for the vertices are shown for the first $P_4$ ($\sigma=1234$). One can find out which cells are glued.}
		
\end{figure}

\section{Homotopy equivalence between the permutahedral spaces $Cact^1(n)$ and $\mathcal{F}(n)$}
\label{homotopysec}

The two spaces $\mathcal{F}(n)$ and $\mathcal{C}(n)$ are closely related as quotients of permutahedral space.

$$\xymatrix{
&\coprod_{\sigma\in S_n}P_n\ar[dl]_{p_{\CalF}}\ar[dr]^{p_{\CalC}}&\\
\CalF(n)&&\CalC(n)}
$$
 But the gluings for $\mathcal{F}(n)$ only occur on the proper faces of $P_n$ while those for $\mathcal{C}(n)$ also happen in the interior of $P_n$. In fact, only the interiors of the hyper-cubes $\mathcal{C}_{B^+_{\sigma_1}(B^+_{\sigma_2}(\cdots B^+_{\sigma_{n-1}}(scc(\sigma_n))\cdots))}$ in each of the $n!$ copies of $P_n$ are not glued. The gluings for $\CalC(n)$ are cell-wise and we identify the cells $C(\tau)$ in the decomposition of $P_{\sigma}$ and $P_{\nu}$ if the order $\prec_{\tau}$ is
compatible with both $<_{\sigma}$ and $<_{\nu}$.

 \begin{lemma}
$p_{\CalC}$ is constant on fibers of $p_{\CalF}$ and hence there is an induced map $\overline{1}_n$
\begin{equation}
\xymatrix{
&\coprod_{\sigma\in S_n}P_n\ar[dl]_{p_{\CalF}}\ar[dr]^{p_{\CalC}}&\\
\CalF(n)\ar[rr]^{\overline{1}_n}&&\CalC(n)}
\end{equation}
\end{lemma}

\begin{proof}

 If $x\sim_{\mathcal{F}}y$ where $x\in P_n$ indexed by $\sigma$ and $y\in P_n$ indexed by $\nu$, let $\mathbf{l}=\mathbf{l}_1|\cdots|\mathbf{l}_k$ be an element in $\mathcal{J}_n$ such that $\mathbf{l}\in \mathcal{J}_{\sigma}\cap\mathcal{J}_{\nu}$ and $x=y$ in the interior of $\mathcal{F}(\mathbf{l})$. Then we can find $\tau=\tau_1|\cdots|\tau_k\in \mathcal{T}_{\sigma}\cap\mathcal{T}_{\nu}$ where $\tau_i\in \mathcal{T}_{\mathbf{l}_i}$, $i=1,\cdots,k$  by using cactus decomposition of each $P_{|\mathbf{l}_i|}$ associated to $\mathbf{l}_i$ such that $x=y$ in $\mathcal{C}_{\tau}$. So $x\sim_{\mathcal{C}}y$. 

\end{proof} 
It is easily seen that this map is again a quotient map.

\begin{proposition}
The map $\overline{1}_n$ is a quasi--isomorphism. Furthermore
it induces a map on the level of cellular chains 
$CC_*(\overline{1}_n):CC_*(\CalF_n)\to CC_*(\CalC_n)=CC_*(Cact^1(n))$ where  $CC_*(P_{\sigma})\to \sum_{\prec_\t\text{ compatible with } <_{\sigma}} C(\t)$.

\end{proposition}
\begin{proof}
The map on the cellular level is clear from the description above.
It is well known that $\CalF_n$ has the homotopy type of 
$K(PB_n,1)$ and it is proved in \cite{Kau02} that the same holds for $Cact^1(n)$, which shows that it is a quasi--isomorphism. 
\end{proof}

In the remainder of the section, we will prove a little more, namely we will prove that $\overline{1}_n$ is a homotopy equivalence by constructing an explicit homotopy inverse 
$\overline{h}_n$.  The qausi--isomorphism part of the above proposition then follows without resorting to abstract recognition principles.

\begin{theorem}
\label{2nd}
$\overline{1}_n:\mathcal{F}(n)\rightarrow \mathcal{C}(n)$ is a homotopy equivalence with explicit homotopy inverse $\overline{h}_n$
constructed in \S\ref{homotopypar}.
\end{theorem}

\begin{proof}
This follows from Proposition \ref{homotopyprop}  below.
\end{proof}

The way the maps are constructed is by considering lifts along one projection, then a map:   $f:\coprod_{\sigma\in S_n}P_n \to \coprod_{\sigma\in S_n}P_n$ followed by the other projection.
We will call the resulting map the map induced by $f$.
For the induced map to exist, of course $f$ should be suitably constant along fibers.
In particular, the map  $\overline{1}_n$ is defined by lifting along $p_{\CalF}$
and then simply projecting along $p_{\CalC}$. Thus it is induced by the identity map $1_n$ which is the identity on all of the $P_{\sigma}$.

\begin{remark}
\label{homotopyrmk}
We will  describe the homotopy inverse  $\overline{h}_n:\mathcal{C}(n)\rightarrow \mathcal{F}(n)$ as a map induced from $h_n:=\coprod_{\sigma\in S_n}h_{\sigma}:\coprod_{\sigma\in S_n}P_n\rightarrow \coprod_{\sigma\in S_n}P_n$.

That is, we consider the diagram
$$
\xymatrix{
\coprod_{\sigma\in S_n}P_n \ar[d]^{p_\CalF}&\ar[l]_{h_n=\coprod h_{\sigma}} \coprod_{\sigma\in S_n}P_n\ar[d]_{p_\CalC}\\
\CalF(n)&\ar[l]_{\overline{h_n}} \CalC(n)
}
$$ 
with the condition  that $h_n(x)\sim_{\mathcal{F}}h_n(y)$ if $x\sim_{\mathcal{C}}y$.

This will be achieved by having each $h_{\sigma}$ map all points in $P_n$ other than those in the interior of $\mathcal{C}_{B^+_{\sigma_1}(B^+_{\sigma_2}(\cdots B^+_{\sigma_{n-1}}(scc(\sigma_n))\cdots))}$ to proper faces of $P_n$ and then analogous conditions on the the proper faces of $P_n$ are inductively satisfied.

We will define $h_{\sigma}$ and the homotopy showing it is a homotopy inverse at the same time. That is, we will define  $H_{\sigma}: P_n\times I\rightarrow P_n$  and then set $h_{\sigma}=H_{\sigma}(\cdot,1)$ for each $\sigma$. 

To prove the homotopy equivalence, we notice that the two maps  $\overline{h}_n\circ \overline{1}_n:\mathcal{F}(n)\rightarrow\mathcal{F}(n)$ and 
$\overline{1}_n\circ \overline{h}_n:\mathcal{C}(n)\rightarrow \mathcal{C}(n)$ are both induced from $h_n:\coprod_{\sigma\in S_n}P_n\rightarrow \coprod_{\sigma\in S_n}P_n$, in the sense that we have the diagrams
$$
\xymatrix{
\coprod_{\sigma\in S_n}P_n \ar[d]^{p_\CalF}\ar[r]^{h_n\circ 1_n=h_n}& \coprod_{\sigma\in S_n}P_n\ar[d]_{p_\CalF}\\
\CalF(n)\ar[r]^{\overline{h_n}\circ\overline{1_n}} &\CalF(n)
}
\quad
\xymatrix{
\coprod_{\sigma\in S_n}P_n \ar[d]^{p_\CalC}&\ar[l]_{1_n\circ h_n=h_n} \coprod_{\sigma\in S_n}P_n\ar[d]_{p_\CalC}\\
\CalC(n)&\ar[l]_{\overline{1_n}\circ \overline{h_n}} \CalC(n)
}
$$ 

This means that if the homotopies proving the homotopy equivalence are   $H_{\mathcal{F}}$ and $H_{\mathcal{C}}$, i.e.\ $1_{\mathcal{F}(n)}\simeq_{H_{\mathcal{F}}}\overline{h}_n\circ \overline{1}_n$ and $1_{\mathcal{C}(n)}\simeq_{H_{\mathcal{C}}}\overline{1}_n\circ \overline{h}_n$, we can look for a common homotopy 
$\coprod_{\sigma\in S_n}H_{\sigma}$ inducing both $H_{\mathcal{F}}$ and $H_{\mathcal{F}}$.

This homotopy has to and will satisfy the following conditions

\begin{itemize}
\item[$(*1)$] $H_{\sigma}(\cdot,0)=1_{\sigma}:P_n\rightarrow P_n$.

\item[$(*2)$] If $x\sim_{\mathcal{F}}y$ where $x$ is in $P_n$ indexed by $\sigma$ and $y$ is in $P_n$ indexed by $\nu$, then $H_{\sigma}(x,t)\sim_{\mathcal{F}}H_{\nu}(y,t)$ for all $t\in I$.
\item [$(*3)$] If $x\sim_{\mathcal{C}}y$ where $x$ is in $P_n$ indexed by $\sigma$ and $y$ is in $P_n$ indexed by $\nu$, then
\begin{itemize}
\item[$(*3a)$] $H_{\sigma}(x,t)\sim_{\mathcal{C}}H_{\nu}(y,t)$ for all $t\in I$, and
\item[$(*3b)$] $H_{\sigma}(x,1)\sim_{\mathcal{F}}H_{\nu}(y,1)$.\\
\end{itemize}
\end{itemize}
\end{remark}

\subsection{Rough sketch of a proof or Theorem \ref{2nd}}
Before delving into the intricate details of fully constructing the homotopy, we will present a short argument.
First, we know that the individual $P_n$s are homotopic to their core $I^{n-1}$s, abstractly. More concretely, by Theorem\ref{CactP_n}, we know that the cells are glued iteratively in $n-1$ steps, parameterized by the initial branching number. We obtain a retract $r:P_n\to I_n$,
by collapsing the cells in reverse order to the piece of the boundary
that is attached to the lower shell. 
 That is, first we look at
$$
\xymatrix{
I_n\ar@<3pt>@{^{(}->}^i[r]&\ar@<3pt>[l]^rP_n
}
$$
It is not hard to show that this is a deformation retract.
When the gluing maps are added, however it will be more 
convenient to realize that there is actually a map going the other way around. Although it is constructed a bit differently, the idea
is that if $V$ is the vertex set of $P_n$ and 
then $W=r(V)$ contains the vertex set of $I_n$ and additional points
in the boundary. Mapping back $W$ to $V$ linearly, gives a map the other way around. This map can be extended to the whole of $P_n$,
which is the sought after map $h_n$. It maps $I_n$ homeomorphically onto $P_n$ and is homotopic to the identity.

$$
\xymatrix{
P_n\ar@<3pt>^{\approx h_n}[r]&\ar@<3pt>[l]^{id}P_n
}
$$

On the cellular level, we contract all the cells that are not of the type $I^n$ and then obtain a complex which is isomorphic to $\CalF(n)$. 
\subsection{Explicit construction of the homotopy}
\label{homotopypar}
For the actual homotopy,
the idea is that one retracts the cells building up the $P_{\sigma}$ to the part of their boundary that is {\em not glued}. These cells are
given by
 \eqref{celleq}
as $P_{m_1}\times P_{m_2}\times \cdots \times P_{m_k}\times \Delta^k$ which, for concreteness, can be viewed as in $\mathbb{R}^{m_1-1}\times \mathbb{R}^{m_2-1}\times \cdots\times \mathbb{R}^{m_k-1}\times \mathbb{R}^k$. They are attached to cells of lower initial branching numbers along $P_{m_1}\times P_{m_2}\times \cdots \times P_{m_k}\times(\bigcup_{i=2}^kv_1\cdots \widehat{v_i}\cdots v_{k+1})\subset P_{m_1}\times P_{m_2}\times \cdots \times P_{m_k}\times \partial\Delta^k \subset \partial(P_{m_1}\times P_{m_2}\times \cdots \times P_{m_k}\times \Delta^k)$.

The basic homotopy is the following:
\begin{proposition}
\label{ProductRetract}

Let $n\geq 3$. For $m_1,\cdots,m_k\geq 1$, where $k\geq 2$ and $m_1+\cdots+m_k=n-1$, 
$\partial(P_{m_1}\times P_{m_2}\times \cdots \times P_{m_k}\times\Delta^k)\backslash \mathrm{Int}(P_{m_1}\times P_{m_2}\times \cdots \times P_{m_k}\times(\bigcup_{i=2}^kv_1\cdots \widehat{v_i}\cdots v_{k+1}))$ is a deformation retract of $P_{m_1}\times P_{m_2}\times \cdots \times P_{m_k}\times\Delta^k$.
\end{proposition}
\begin{proof}
A short argument is as follows.
Consider  $P_{m_1}\times P_{m_2}\times \cdots \times P_{m_k}\times\Delta^k$ as fibered  over $P_{m_1}\times P_{m_2}\times \cdots \times P_{m_k}\times(\bigcup_{i=2}^kv_1\cdots \widehat{v_i}\cdots v_{k+1})$. The  fibers are singletons along $\partial (P_{m_1}\times P_{m_2}\times \cdots \times P_{m_k}\times(\bigcup_{i=2}^kv_1\cdots \widehat{v_i}\cdots v_{k+1}))$ and the fibers over the points not in the previous set are closed intervals. Then we can contract $P_{m_1}\times P_{m_2}\times \cdots \times P_{m_k}\times\Delta^k$ onto $\partial(P_{m_1}\times P_{m_2}\times \cdots \times P_{m_k}\times\Delta^k)\backslash \mathrm{Int}(P_{m_1}\times P_{m_2}\times \cdots \times P_{m_k}\times(\bigcup_{i=2}^kv_1\cdots \widehat{v_i}\cdots v_{k+1}))$.
The full proof is in the  Appendix.
\end{proof}
\subsubsection{Extended products and extended homotopies}

These homotopies cannot be used directly, since we have to take care of the attaching maps. For this, we have to slightly thicken the cell and while retracting the interior of the cell to the boundary, ``pull'' the thickening into the interior.

Let $I_{\epsilon}$ be a closed line segment with small length $\epsilon$. Consider $(P_{m_1}\times \cdots \times P_{m_k}\times (\bigcup_{i=2}^kv_1\cdots\widehat{v_i}\cdots v_k))\times I_{\epsilon}$ to be embedded  into $\mathbb{R}^{n-1}$.
Let $\mathrm{Ext}_{\epsilon}(P_{m_1}\times \cdots \times P_{m_k}\times \Delta^k)$ be the union of $P_{m_1}\times \cdots \times P_{m_k}\times \Delta^k$, which we call the {\em basic cell} and  $(P_{m_1}\times \cdots \times P_{m_k}\times (\bigcup_{i=2}^kv_1\cdots\widehat{v_i}\cdots v_{k+1}))\times I_{\epsilon}$, which we call the {\em tab}, along
  $P_{m_1}\times \cdots \times P_{m_k}\times (\bigcup_{i=2}^kv_1\cdots\widehat{v_i}\cdots v_{k+1})$.

\begin{proposition}
\label{extended}
There is a homotopy $H_{\mathrm{Ext}_{\epsilon}(P_{m_1}\times \cdots \times P_{m_k}\times \Delta^k)}:\mathrm{Ext}_{\epsilon}(P_{m_1}\times \cdots \times P_{m_k}\times \Delta^k)\times I\rightarrow \mathrm{Ext}_{\epsilon}(P_{m_1}\times \cdots \times P_{m_k}\times \Delta^k)$, satisfying the following conditions:

\begin{enumerate}
\item it contracts $P_{m_1}\times P_{m_2}\times \cdots \times P_{m_k}\times\Delta^k$ onto $\partial(P_{m_1}\times P_{m_2}\times \cdots \times P_{m_k}\times\Delta^k)\backslash \mathrm{Int}(P_{m_1}\times P_{m_2}\times \cdots \times P_{m_k}\times(\bigcup_{i=2}^kv_1\cdots \widehat{v_i}\cdots v_{k+1}))$,

\item  it maps $(P_{m_1}\times \cdots \times P_{m_k}\times (\bigcup_{i=2}^kv_1\cdots\widehat{v_i}\cdots v_{k+1}))\times I_{\epsilon}$ homeomorphically to $\mathrm{Ext}_{\epsilon}(P_{m_1}\times \cdots \times P_{m_k}\times \Delta^k)$.

\item Let $X$ be space and $I'$ a closed subinterval of $I$, then we call a map $G:X\times I'\rightarrow X$ the identity homotopy on $X$ if $G(\cdot,t)=1_X$ for all $t\in I'$. Then $H_{\mathrm{Ext}_{\epsilon}(P_{m_1}\times \cdots \times P_{m_k}\times \Delta^k)}$ is the identity homotopy on $\partial \mathrm{Ext}_{\epsilon}(P_{m_1}\times \cdots \times P_{m_k}\times \Delta^k)$.
\end{enumerate}
\end{proposition}

\begin{proof}
 See appendix.
\end{proof}

\subsubsection{Embedding the extensions}
Now we describe how to embed 
the extended products inside each $P_n$.

Let $n\geq 3$. We will do the construction for the $P_n$ corresponding the identity element $1_n=12\cdots n\in S_n$.  Once we have this, we can push it forward by $\sigma$ to obtain the construction for $P_n$ corresponding to $\sigma$.

Let $2\leq k\leq n-1$. For any $m_1,\cdots,m_k\geq 1$ with $m_1+\cdots+m_k=n-1$, we again first  consider the standard partition $\mathbf{j}=\mathbf{j}_1,\cdots,\mathbf{j}_k\in \mathrm{dSh}_{23\cdots n}[m_1,\cdots,m_k]$ where $\mathbf{l}_1=23\cdots (m_1+1)$ and $\mathbf{l}_i=(m_1+\cdots+m_{i-1}+2)\cdots (m_1+\cdots+m_i+1)$ for $i=2,\cdots,k$. Notice that the sequence $1\mathbf{j}_1\cdots\mathbf{j}_k$ is the sequence $12\cdots n$ for $1_n$. We know that $\mathrm{colim}_{\mathcal{T}_{1_n}[\mathbf{j}]}\mathcal{C}\cong P_{m_1}\times \cdots \times P_{m_k}\times \Delta^k$. Let $\mathrm{Ext}_{\phi_{m_1,\cdots,m_k}}(\mathrm{colim}_{\mathcal{T}_{1_n}[\mathbf{l}]}\mathcal{C})$ be the image in $P_n$ of $\mathrm{Ext}_{\epsilon}(P_{m_1}\times \cdots \times P_{m_k}\times \Delta^k)$ under an homeomorphism $\phi_{m_1,\cdots,m_k}$ first satisfying the following conditions. 
\begin{enumerate}
\item It is a homeomorphism onto its image.
\item It maps the basic cell to the corresponding cell in $P_n$
\item It maps the tab into the cells that the corresponding cell is attached to in the iteration.

\end{enumerate}
Maps like this exist in abundance, which is easily seen by regarding
a neighborhood of the common boundary.

Now for general $\sigma\in S_n$, and $\mathbf{l}=\mathbf{l}_1,\cdots,\mathbf{l}_k\in \mathrm{dSh}_{\sigma\backslash \sigma_1}[m_1,\cdots,m_k]$, let $\omega=\sigma_1\mathbf{l}_1\cdots\mathbf{l}_k\in S_n$. We let  $\mathrm{Ext}_{\phi_{m_1,\cdots,m_k}}(\mathrm{colim}_{\mathcal{T}_{\sigma}[\mathbf{l}]}\mathcal{C})$ be the image of $\mathrm{Ext}_{\phi_{m_1,\cdots,m_k}}(\mathrm{colim}_{\mathcal{T}_{1_n}[\mathbf{j}]}\mathcal{C})$ under the linear map, which permutes it into the right position.

$$\sum_{\nu\in S_n}t_{\nu}\mathcal{C}_{scc(\nu_1,\cdots,\nu_n)}\mapsto \sum_{\nu\in S_n}t_{\nu}\mathcal{C}_{scc((\omega\nu)_1,\cdots,(\omega\nu)_n)}.$$

Since for fixed $\sigma$ and $k$, any two from the collection of $\mathrm{colim}_{\mathcal{T}_{\sigma}[\mathbf{l}]}\mathcal{C}$ for all $m_1,\cdots,m_k$ and all $\mathbf{l}\in \mathrm{dSh}_{\sigma\backslash \sigma_1}[m_1,\cdots,m_k]$ either are disjoint or share a subspace homeomorphic to $D^i$, where $i$ is at most $n-3$, after possibly shrinking and perturbing the image of the tabs,
 we can choose the homeomorphisms $\phi_{m_1,\cdots,m_k}$ such that 
\begin{enumerate}
\setcounter{enumi}{3}
\item  
 the interiors of $\mathrm{Ext}_{\phi_{m_1,\cdots,m_k}}(\mathrm{colim}_{\mathcal{T}_{\sigma}[\mathbf{l}]}\mathcal{C})$ are pairwise disjoint for fixed $\sigma$ and $k$ and
 \item the same top dimensional cells in $P_n$ for different $\sigma$ have the same extended products. 
\end{enumerate}

\begin{notation}
Since for fixed $n$ and the homeomorphisms $\phi_{m_1,\cdots,m_k}$, $\mathrm{Ext}_{\phi_{m_1,\cdots,m_k}}(\mathrm{colim}_{\mathcal{T}_{\sigma}[\mathbf{l}]}\mathcal{C})$ only depends on $\mathbf{l}$, for simplicity, we denote it by $E_{\mathbf{l}}$.
For the example $n=3$, see Figure \ref{extendedfig}.

\begin{figure}
		\centering
		\includegraphics[width=150mm]{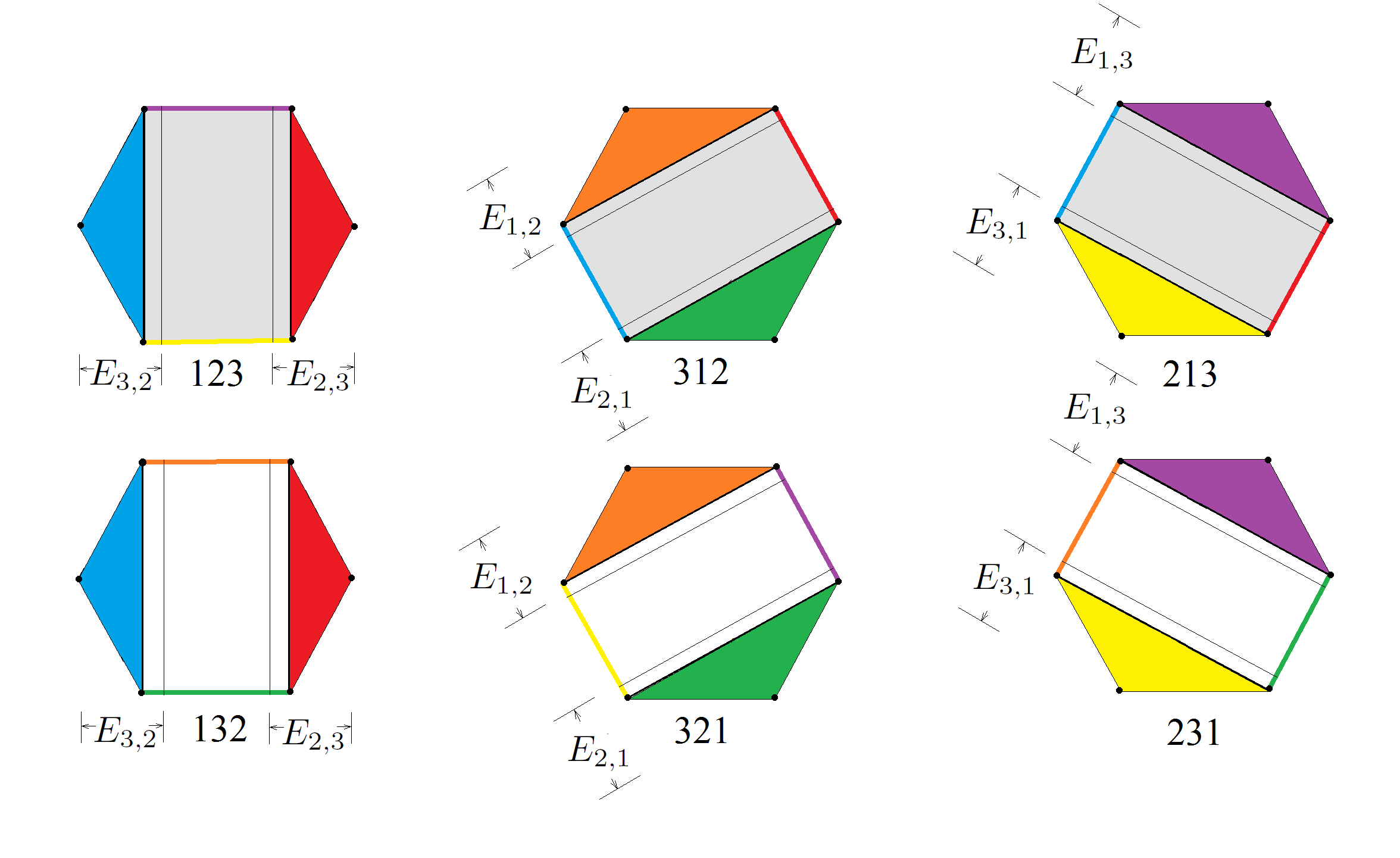}
		\caption{\label{extendedfig}$E_{\mathbf{l}}$ when $n=3$.}
\end{figure}

Under the homeomorphisms $\phi_{m_1,\cdots,m_k}$ and the linear maps, we transfer the homotopies $H_{\mathrm{Ext}_{\epsilon}(P_{m_1}\times \cdots\times P_{m_k}\times \Delta^k)}$ from $\mathrm{Ext}_{\epsilon}(P_{m_1}\times \cdots\times P_{m_k}\times \Delta^k)$ to each $E_{\mathbf{l}}$. We call this homotopy $H_{\mathbf{l}}:E_{\mathbf{l}}\times I\rightarrow E_{\mathbf{l}}$. \end{notation}

It is clear that these homotopies have  the three properties transferred from those in Proposition \ref{extended}.
\begin{corollary}
\label{simple}
The homotopies $H_{\mathbf{l}}:E_{\mathbf{l}}\times I\rightarrow E_{\mathbf{l}}$ satisfy the following conditions:

\begin{enumerate}
\item Choose $\sigma\in S_n$ such that $\mathbf{l}=\mathbf{l}_1,\cdots,\mathbf{l}_k\in \mathrm{dSh}_{\sigma\backslash \sigma_1}[m_1,\cdots,m_k]$. For each $i=1,2,\cdots,k-1$, let $$d_i\mathcal{T}_{\sigma}^{n-2}[\mathbf{l}]=\{B^+_{\sigma_1}(\tau_1,\cdots,\tau_i|\tau_{i+1},\cdots,\tau_k)| \tau_j\in \mathcal{T}^{m_j-1}_{\mathbf{l}_j}\}$$ and $$d\mathcal{T}^{n-2}_{\sigma}[\mathbf{l}]=\bigcup_{i=1}^{k-1}d_i\mathcal{T}_{\sigma}^{n-2}[\mathbf{l}].$$ Let $d\mathcal{T}_{\sigma}[\mathbf{l}]$ be the set of trees in $\mathcal{T}_{\sigma}$ such that each tree is smaller than or equal to an element in $d\mathcal{T}^{n-2}_{\sigma}[\mathbf{l}]$. Then $H_{\mathbf{l}}$ contracts $\mathrm{colim}_{\mathcal{T}_{\sigma}[\mathbf{l}]}\mathcal{C}$ onto $(\partial \mathrm{colim}_{\mathcal{T}_{\sigma}[\mathbf{l}]}\mathcal{C})\backslash \mathrm{Int(colim}_{d\mathcal{T}_{\sigma}[\mathbf{l}]}\mathcal{C})$;

\item it maps the closure of $E_{\mathbf{l}}\backslash \mathrm{colim}_{\mathcal{T}_{\sigma}[\mathbf{l}]}$ homeomorphically to $E_{\mathbf{l}}$;

\item $H_{\mathbf{l}}$ is the identity homotopy on $\partial E_{\mathbf{l}}$.
\end{enumerate}
\qed
\end{corollary}

\subsubsection{Iterated cone construction}

Starting from $n=4$, we have to take care of the boundaries. For this we will use the so--called iterated cones, which we now define. 

Let $n\geq 4$ and $1_n=12\cdots n\in S_n$. Notice that $\mathrm{colim}_{\mathcal{T}_{1_n,1}}\mathcal{C}$, as a subspace of $\mathcal{F}_{1_n}=\mathrm{colim}_{\mathcal{T}_{1_n}}\mathcal{C}=P_n$, is the cartesian product of $P_{n-1}$ with $I_{\sqrt{n(n-1)}}$, where $I_{\sqrt{n(n-1)}}$ is the interval $[0,\sqrt{n(n-1)}]$. We construct the iterated cones of the faces of $P_n$ of dimension $i:2\leq i\leq n-2$, where $P_n$ is seen as the realization of $1_n$ under $\mathcal{F}$, as follows. We then transfer the cones to all the $P_\sigma$ symmetrically by using the $S_n$ action.
The cones are specified, by giving their cone vertex which will be a point inside $P_n$.  

There is a choice for such cone vertices. We will choose these vertices ``as close to the base as needed'' and the cones at each step are mutually disjoint. This is technically done by requiring that the line determined by the vertex $v$ and the geometric center of the base is perpendicular to the face and the distance from $v$ to the geometric center is small. Since there are only finitely many cones in each cactus cell it follows that this is possible. The distance and perpendicularity do not fix $v$ uniquely starting at codimension 2. In that case, we will choose $v$ such that the entire cone lies {\em inside} a union of top dimensional cactus cells with maximal possible initial branching number $k$.

\textbf{Step $1$.}  For any dimension $2$ face $\mathcal{F}_{\mathbf{l}_1|\cdots|\mathbf{l}_{n-2}}$ of $P_n$, let $v$ be a point in the interior of $P_n$ with distance $\epsilon_2$ to the geometric center of $\mathcal{F}_{\mathbf{l}_1|\cdots|\mathbf{l}_{n-2}}$. We choose this $v$ as explained above and such that $v$ is not in $\mathrm{colim}_{\mathcal{T}_{1_n,1}}\mathcal{C}$.  Then we form the join $\mathcal{F}_{\mathbf{l}_1|\cdots|\mathbf{l}_{n-2}}*v$, which is a cone with base $\mathcal{F}_{\mathbf{l}_1|\cdots|\mathbf{l}_{n-2}}$. We call such a cone $C^1(\mathcal{F}_{\mathbf{l}_1|\cdots|\mathbf{l}_{n-2}})$. These cones are $3$-dimensional.

\textbf{Step $i-1$, $3\leq i\leq n-3$.} For any dimension $i$ face $\mathcal{F}_{\mathbf{l}_1|\cdots|\mathbf{l}_{n-i}}$ of $P_n$, we consider its union with the $(i-2)$-cones of its codimension $1$ faces: $$\mathcal{F}_{\mathbf{l}_1|\cdots|\mathbf{l}_{n-i}}\cup\bigcup_{\mathbf{k}\in \mathcal{J}_{1_n}^{i-1},\mathbf{k}<\mathbf{l}_1|\cdots|\mathbf{l}_{n-i}}C^{i-2}(\mathcal{F}_{\mathbf{k}}).$$ Let $v$ be a point in the interior of $P_n$ with distance $\epsilon_i$ below the face $\mathcal{F}_{\mathbf{l}_1|\cdots|\mathbf{l}_{n-i}}$ as explained above.
If necessary, we move the previous cone vertices, so that the line segments from $v$ to any point in the union do not contain any other point. Again this is possible, since there are only finitely many cones and we can vary the distance and the position of the cone points for lower dimensions.

 Then we form the join of $v$ with this union and denote it by $C^{i-1}(\mathcal{F}_{\mathbf{l}_1|\cdots|\mathbf{l}_{n-i}})$ and call it the $(i-1)$-cone of $\mathcal{F}_{\mathbf{l}_1|\cdots|\mathbf{l}_{n-i}}$.
 Its dimension is $(i+1)$.

\textbf{Step $n-3$.} For any dimension $n-2$ (codimension $1$) face $\mathcal{F}_{\mathbf{l}_1|\mathbf{l}_2}$ of $P_n$, we form the union $$B(\mathbf{l}_1|\mathbf{l}_{2}):=\mathcal{F}_{\mathbf{l}_1|\mathbf{l}_{2}}\cup\bigcup_{\mathbf{k}\in \mathcal{J}_{1_n}^{n-3},\mathbf{k}<\mathbf{l}_1|\mathbf{l}_{2}}C^{n-4}(\mathcal{F}_{\mathbf{k}}).$$

For the special elements $1|2\cdots n$ and $2\cdots n|1$, we let $Cyl(\mathcal{F}_{1|2\cdots n})=Cyl(\mathcal{F}_{2\cdots n|1})$ be the space as the union of line segments such that each line segment joins a point in $B(1|2\cdots n)$ to the corresponding point in $B(2\cdots n|1)$. We call it the cylinder of $\mathcal{F}_{1|2\cdots n}$ (or of $\mathcal{F}_{2\cdots n|1}$).

For general $\mathbf{l}_1|\mathbf{l}_2$ (including $1|2\cdots n$ and $2\cdots n|1$), let $v$ be a point in the interior of $P_n$ with distance $\epsilon_{n-2}$ directly below the geometric center of $\mathcal{F}_{\mathbf{l}_1|\mathbf{l}_2}$ such that if $\mathbf{l}_1|\mathbf{l}_2$ is neither $1|2\cdots n$ nor $2\cdots n|1$, $v$ is not in the interior of $Cyl(\mathcal{F}_{1|2\cdots n})$. We form the join $B(\mathbf{l}_1|\mathbf{l}_2)*v$ and denote it by $C^{n-3}(\mathcal{F}_{\mathbf{l}_1|\mathbf{l}_2})$. We call it the $(n-3)$-cone of $\mathcal{F}_{\mathbf{l}_1|\mathbf{l}_2}$.\\

We choose the values of $\epsilon_2,\cdots,\epsilon_{n-2}$ small enough and the orientation of each $v$ such that
\begin{enumerate}
\item all the previous conditions are satisfied;
\item each iterated cone is contained in the union of the elements of a collection $\{\mathcal{C}_{\mathbf{l}_l}\}$ such that each $\mathbf{l}_l\in \mathcal{T}^{n-1}_{1_n}$ has the largest possible arity number;
\item the union of all $(n-3)$-cones exhibits maximal symmetry.
\end{enumerate}

To construct the iterated cones for $P_n$ as $\mathcal{F}_{\sigma}$ for any $\sigma\in S_n$, we take as images of the iterated cones of $\mathcal{F}_{1_n}$ under the linear map $$\sum_{\nu\in S_n}t_{\nu}\mathcal{F}_{\nu_1|\cdots|\nu_n}\mapsto \sum_{\nu\in S_n}t_{\nu}\mathcal{F}_{(\sigma\nu)_1|\cdots|(\sigma\nu)_n}.$$

\subsubsection{Transferring homotopies from faces to cones}

We will use the following observation to transfer the previously defined homotopies as homotopies on faces into the ambient $P_n$ by inducing a homotopy on  a cone over the face  with vertex $v$ which lies inside $P_n$.

Let $P$ be a polytope and $H_P:P\times I\rightarrow P$ a homotopy with $H_P(\cdot,t)|_{\partial P}=1_{\partial P}$ for all $t\in I$. 
We define $H_{P\times I}:(P\times I)\times I\rightarrow P\times I$ by $((x,s),t)\mapsto (H_P(x,t),s)$. So $H_{P\times I}(\cdot,\cdot,t)|_{(\partial P)\times I}=1_{(\partial P)\times I}$ for all $t\in I$. Thus, if we  identify $P\times I/P\times \{1\}$ with a cone $P*v$,
the homotopy $H_{P\times I}$ induces a homotopy on $P\times I/P\times \{1\}\approx P*v$, which we call it $H_{P*v}$.  This homotopy then satisfies.

\begin{equation}
\label{boundary}
H_{P*v}(\cdot,t)|_{(\partial P)*v}=1_{(\partial P)*v}
\end{equation}

\subsubsection{The homotopy}
\label{homotopy}

Now we describe the homotopies $H_{\sigma}$ where $\sigma \in S_n$ for some $n\geq 1$.

For $n=1,2$, we let $H_{\sigma}$ be the identity homotopies. In fact, $\mathcal{C}(n)\approx\mathcal{F}(n)$ in these two cases. For $n=3$, we let $H_{\sigma}(\cdot,t)$ be $H_{\sigma_2,\sigma_3}(\cdot,t)$ on $E_{\sigma_2,\sigma_3}$ and $H_{\sigma_3,\sigma_2}(\cdot,t)$ on $E_{\sigma_3,\sigma_2}$; we let $H_{\sigma}$ be the identity homotopy on $P_3
\backslash \mathrm{Int}(E_{\sigma_2,\sigma_3}\cup E_{\sigma_3,\sigma_2})$. By Corollary \ref{simple}, $H_{\sigma}$ is a well-defined homotopy on $P_3$. See Figure \ref{fig:h3}

\begin{figure}
		\centering
		\includegraphics[width=150mm]{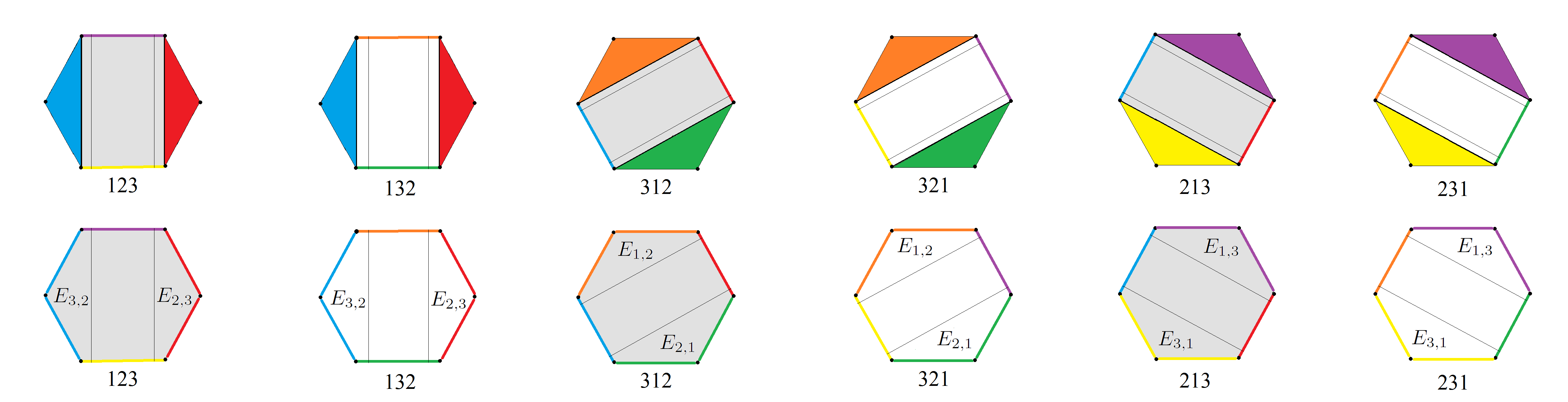}
		\caption{\label{h3fig} Top row: the domains of $h_{\sigma}=H_{\sigma}(\cdot,1)$; Bottom row: the images of $h_{\sigma}=H_{\sigma}(\cdot,1)$.}
		\label{fig:h3}
\end{figure}

For each $n\geq 4$, we describe $H_{\sigma}$ in $n$ steps, assuming we know the homotopies for all $m<n$. 

In the first $n-1$ steps the homotopy is applied according to the initial
branching number starting with $k=1$ and ending with $k=n-1$. After
that, in the last step (Step $n$), the homotopy is done on the faces.The homotopy on the faces
uses the iterated cones of degree $\leq n-3$ while in Steps $2,\dots, n-1$, we use the extensions. The first step is more complicated
since part of the boundary of the cells that we are moving is connected to other cells of other $P_{\sigma}$. So in this step, there is an additional use of iterated cones.

In the steps with iterated cones, the homotopy is done in several substeps. For this we need the following technical details:
 For $\mathbf{l}=\mathbf{l}_1|\cdots|\mathbf{l}_k\in \mathcal{J}_{\sigma}$, let $l_i$ be the length of the string $\mathbf{l}_i$ and $l$ the maximum of $\{l_i\big| i=1,\cdots,k\}$.  The homotopies over the cones of the faces will be performed inductively over $l$.

Correspondingly, we subdivide the interval $I$ in such a way that the homotopy can take place at different times in $I$.
Let $f^m_0:I\rightarrow I$ be defined by $t\mapsto \frac{1}{m}t$ and $f^m_1:I\rightarrow I$ by $t\mapsto \frac{1}{m}(t+m-1)$, where $m\geq 4$. Let $f^m_{0}:I\rightarrow I$ and $f^m_1:I\rightarrow I$ be the identity maps if $m=3,2,1$. Then we let $f_{i_1i_2\cdots i_k}=f^n_{i_1}\circ f^{n-1}_{i_2} \circ \cdots \circ f^{n-k+1}_{i_k}$ for $1\leq k\leq n$, where $i_j=0,1$, $j=1,\cdots,k$. We also let $I_{i_1i_2\cdots i_k}=f_{i_1i_2\cdots i_k}(I)$. For an example, see Figure \ref{Ifig}.
  From now on,  $i_1=0$ corresponds to step 1 while $i_1=1$ corresponds to step $n$. 

 \begin{figure}
		\centering
		\includegraphics[width=150mm]{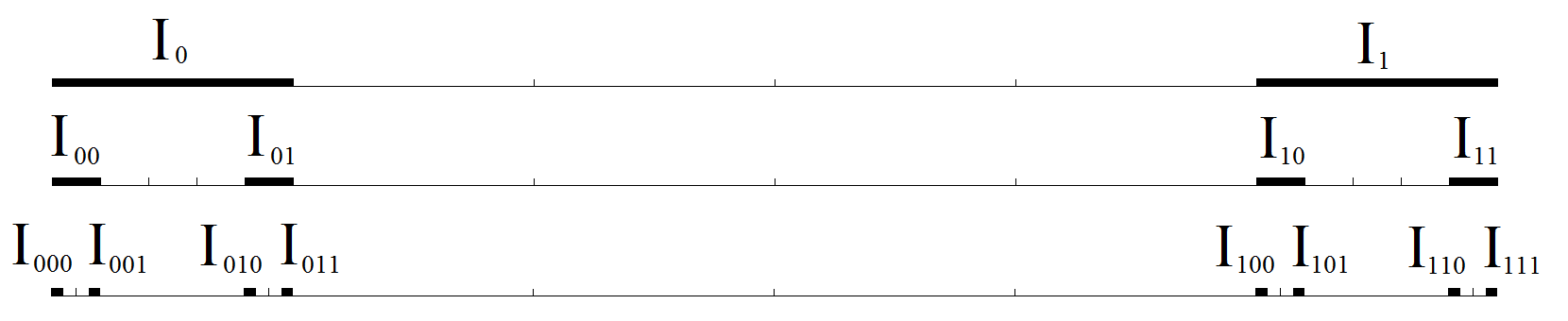}
		\caption{\label{Ifig}$I_{i_1\cdots i_k}$ when $n=6$. Notice that $I_{i_1i_2i_3i_4i_5i_6}=I_{i_1i_2i_3}$.}
\end{figure}

\textbf{Step $1$: } $t\in [0,\frac{1}{n}]$, $i_0=0$.

For $\mathbf{l}=\mathbf{l}_1|\cdots|\mathbf{l}_{n-2}\in \mathcal{J}_{\sigma}$, if $l=2$, then all but two of $\mathbf{l}_1,\cdots,\mathbf{l}_{n-2}$ are of length $1$ and we let $H_{C^1(\mathcal{F}_{\mathbf{l}})}:C^1(\mathcal{F}_{\mathbf{l}})\times [0,\frac{1}{n}]\rightarrow C^1(\mathcal{F}_{\mathbf{l}})$ be the identity homotopy. If $l=3$, then $\mathbf{l}=\cdots|\mathbf{l}_l|\cdots$ where $\mathbf{l}_l=ijk$ and all the other $\mathbf{l}_m$s are of length $1$. 
We define $H_{\mathcal{F}_{\mathbf{l}}}:\mathcal{F}_{\mathbf{l}}\times [0,\frac{1}{n}]\rightarrow \mathcal{F}_{\mathbf{l}}$ by $$(x,t)\mapsto \psi_\mathbf{l}^{-1}\left(*\times \cdots \times *\times H_{\mathcal{F}_{123}}(\psi_{\mathbf{l}_l}(x),f^{-1}_{i_1\cdots i_{n-3}}(t))\times *\times \cdots\times *\right)$$
 if $t\in I_{i_1i_2\cdots i_{n-3}}$ for some $i_1,\cdots,i_{n-3}$ and $(x,t)\mapsto x$ otherwise, where $\psi_{\mathbf{l}}=(*,\cdots,*,\psi_{\mathbf{l}_l},*,\cdots,*):\mathcal{F}_{\mathbf{l}}\rightarrow *\times\cdots \times *\times \mathcal{F}_{123}\times *\times \cdots \times *$ is the homeomorphism defined under $i\mapsto 1$, $j\mapsto 2$ and $k\mapsto 3$. Then we get the induced homotopy $H_{C^1(\mathcal{F}_{\mathbf{l}})}:C^1(\mathcal{F}_{\mathbf{l}})\times [0,\frac{1}{n}]\rightarrow C^1(\mathcal{F}_{\mathbf{l}})$.

Let $3\leq i\leq n-3$. Suppose we have described the homotopies for $t\in [0,\frac{1}{n}]$ on the $(i-2)$-cones. For $\mathbf{l}=\mathbf{l}_1|\cdots|\mathbf{l}_{n-i}$, define $H_{\mathcal{F}_{\mathbf{l}}}:\mathcal{F}_{\mathbf{l}}\times [0,\frac{1}{n}]\rightarrow \mathcal{F}_{\mathbf{l}}$ by $$(x,t)\mapsto \psi_{\mathbf{l}}^{-1}\left(H_{\mathcal{F}_{12\cdots l_1}}(\psi_{\mathbf{l}_1}(x),f^{-1}_{i_1i_2\cdots i_{n-l_1}}(t)),\cdots, H_{\mathcal{F}_{12\cdots l_{n-i}}}(\psi_{\mathbf{l}_{n-i}}(x),f^{-1}_{i_1i_2\cdots i_{n-l_{n-i}}}(t))\right),$$ where  $\psi_{\mathbf{l}}=(\psi_{\mathbf{l}_1},\cdots,\psi_{\mathbf{l}_{n-i}}):\mathcal{F}_{\mathbf{l}}\rightarrow \mathcal{F}_{12\cdots l_1}\times \cdots\times \mathcal{F}_{12\cdots l_{n-i}}$ is the homeomorphism under the assignments $(\mathbf{l}_j)_k\mapsto k$, for $k=1,2,\cdots,l_j$ and $j=1,2,\cdots,n-i$, and we define $H_{\mathcal{F}_{12\cdots l_j}}(\psi_{\mathbf{l}_j}(x),f^{-1}_{i_1i_2\cdots i_{n-l_j}}(t))$ to be $\psi_{\mathbf{l}_j}(x)$ if $f^{-1}_{i_1\cdots i_{n-l_j}}(t)=\emptyset$. The homotopies $H_{\mathcal{F}_{\mathbf{l}}}$ and $H_{C^{i-1}(\mathcal{F}_{\mathbf{k}})}$, where $\mathbf{k}\in \mathcal{J}^{i-1}_{\sigma}$ and $\mathbf{k}<\mathbf{l}$, agree on their overlaps.\\

Thus, we get a well-defined homotopy on $\mathcal{F}_{\mathbf{l}_1|\cdots|\mathbf{l}_{n-i}}\cup\bigcup_{\mathbf{k}\in \mathcal{J}_{1_n}^{i-1},\mathbf{k}<\mathbf{l}_1|\cdots|\mathbf{l}_{n-i}}C^{i-2}(\mathcal{F}_{\mathbf{k}})$, which induces the homotopy $H_{C^{i-1}(\mathcal{F}_{\mathbf{l}})}:C^{i-1}(\mathcal{F}_{\mathbf{l}})\times [0,\frac{1}{n}]\rightarrow C^{i-1}(\mathcal{F}_{\mathbf{l}})$.\\

Lastly, for $\mathbf{l}=\mathbf{l}_1|\mathbf{l}_2$, define $H_{\mathcal{F}_{\mathbf{l}}}:\mathcal{F}_{\mathbf{l}}\times [0,\frac{1}{n}]\rightarrow \mathcal{F}_{\mathbf{l}}$ by $$(x,t)\mapsto \psi^{-1}_{\mathbf{l}}\left(H_{\mathcal{F}_{12\cdots l_1}}(\psi_{\mathbf{l}_1}(x),f^{-1}_{i_1i_2\cdots i_{n-l_1}}(t)),H_{\mathcal{F}_{12\cdots l_2}}(\psi_{\mathbf{l}_2}(x),f^{-1}_{i_1i_2\cdots i_{n-l_2}}(t))\right)$$ as above. Again, we get well-defined 
homotopy on $$B(\mathbf{l}_1|\mathbf{l}_2)=\mathcal{F}_{\mathbf{l}_1|\mathbf{l}_{2}}\cup\bigcup_{\mathbf{k}\in \mathcal{J}_{1_n}^{n-3},\mathbf{k}<\mathbf{l}_1|\mathbf{l}_{2}}C^{n-4}(\mathcal{F}_{\mathbf{k}}).$$ Then we get induced homotopies $H_{C^{n-3}(\mathcal{F}_{\mathbf{l}})}:C^{n-3}(\mathcal{F}_{\mathbf{l}})\times [0,\frac{1}{n}]\rightarrow C^{n-3}(\mathcal{F}_{\mathbf{l}})$ where $\mathbf{l}$ is neither $1|2\cdots n$ nor $2\cdots n|1$. Let $\theta$ be the homeomorphism $\theta=(\theta_1,\theta_2):Cyl(\mathcal{F}_{1|2\cdots n})\rightarrow B(1|2\cdots n)\times I$. Then we also have the homotopy $H_{Cyl(\mathcal{F}_{1|2\cdots n})}:Cyl(\mathcal{F}_{1|2\cdots n})\times [0,\frac{1}{n}]\rightarrow Cyl(\mathcal{F}_{1|2\cdots n})$ defined by $(x,t)\mapsto \theta^{-1}\left(H_{B(1|2\cdots n)}(\theta_1(x),t),\theta_2(x)\right)$. The homotopies $H_{Cyl(\mathcal{F}_{1|2\cdots n})}$ and the $H_{C^{n-3}(\mathcal{F}_{\mathbf{l}})}$s agree on 
their overlaps. On the other hand, we let $H_{\sigma}:P_n\times [0,\frac{1}{n}]\rightarrow P_n$ be the identity homotopy on the complement of the interior of \\$Cyl(\mathcal{F}_{1|2\cdots n})\cup\bigcup_{\mathbf{l}\in \mathcal{J}^{n-2}_{\sigma}\backslash\{1|2\cdots n,2\cdots n|1\}}C^{n-3}(\mathcal{F}_{\mathbf{l}})$. By  \eqref{boundary}, $H_{\sigma}:P_n\times [0,\frac{1}{n}]\rightarrow P_n$ is a well-defined homotopy.

\textbf{Step $j$, $j=2,\cdots,n-1$}: $t\in [\frac{j-1}{n},\frac{j}{n}]$. Let $H_{\sigma}(\cdot,t)$ be $H_{\mathbf{l}}(\cdot,nt-j+1)$ on $E_{\mathbf{l}}$, where $\mathbf{l}=\mathbf{l}_1,\cdots,\mathbf{l}_j$ is compatible with $\sigma$ (recall this means each $\mathbf{l}_i$ is a subsequence of $\sigma_2\cdots\sigma_n$), and let $H_{\sigma}(\cdot,t)$ be the identity homotopy on the complement of the interior of the union of these $E_{\mathbf{l}}$s in $P_n$. By Corollary \ref{simple}, each $H_{\sigma}:P_n\times [\frac{j-1}{n},\frac{j}{n}]\rightarrow P_n$ is a well-defined homotopy.\\

\textbf{Step $n$:} $t\in [\frac{n-1}{n},1]$. We get the induced homotopies $H_{C^{n-3}(\mathcal{F}_{\mathbf{l}})}:C^{n-3}(\mathcal{F}_{\mathbf{l}})\times [\frac{n-1}{n},1]\rightarrow C^{n-3}(\mathcal{F}_{\mathbf{l}})$ for all $\mathbf{l}\in \mathcal{J}^{n-2}_{\sigma}$ as those in the Step 1 except that we let $i_1=1$ and we don't consider the cylinders. On the other hand, we let $H_{\sigma}:P_n\times [\frac{n-1}{n},1]\rightarrow P_n$ be the identity homotopy on the complement of the interior of $\bigcup_{\mathbf{l}\in \mathcal{J}^{n-2}_{\sigma}}C^{n-3}(\mathcal{F}_{\mathbf{l}})$. By \eqref{boundary}, $H_{\sigma}:P_n\times [\frac{n-1}{n},1]\rightarrow P_n$ is a well-defined homotopy.\\

The above $n$ homotopies agree on their overlaps, thus we get a well-defined homotopy $H_{\sigma}:P_n\times I\rightarrow P_n$.

\begin{proposition}
\label{homotopyprop}
For any $n\geq 1$, the homotopies $H_{\sigma}$ satisfy the conditions (*1),(*2) and (*3) of Remark \ref{homotopyrmk} and hence induce the 
homotopies $1_{\mathcal{F}(n)}\simeq_{H_{\mathcal{F}}}\overline{h}_n\circ \overline{1}_n$ and $1_{\mathcal{C}(n)}\simeq_{H_{\mathcal{C}}}\overline{1}_n\circ \overline{h}_n$ detailed in Remark \ref{homotopyrmk}.
\end{proposition}
\begin{proof}

There is nothing to check for $n=1,2$.

By our previous discussion, it suffices to prove that the homotopies $H_{\sigma}$, $\sigma\in S_n$ satisfy $(*1)$, $(*2)$ and $(*3)$.\\

As a warm-up, let us consider the case for $n=3$ in detail first. $(*1)$ holds by Corollary \ref{simple}. Let $x\sim_{\mathcal{F}}y$ where $x\in P_n$ indexed by $\sigma$ and $y\in P_n$ indexed by $\nu$. Then $x=y$ in $\mathrm{Int}(\mathcal{F}_{\mathbf{l}})$ for some $\mathbf{l}\in \mathcal{J}_{\sigma}\cap \mathcal{J}_{\nu}$. If the order of $\mathbf{l}$ is 2, then $\sigma=\nu$. So $H_{\sigma}(x,t)=H_{\nu}(x,t)=H_{\nu}(y,t)$ in $\mathcal{F}_\mathbf{l}$. Otherwise (the order of $\mathbf{l}$ is smaller than 2), $H_{\sigma}(x,t)=x=y=H_{\nu}(y,t)$ in $\mathcal{F}_{\alpha}$ because we have the identity homotopy on the $1$-skeleton. So $H_{\sigma}(x,t)\sim_{\mathcal{F}}H_{\nu}(y,t)$ for any $t$, verifying $(*2)$. Let $x\sim_{\mathcal{C}}y$ where $x\in P_n$ indexed by $\sigma$ and $y\in P_n$ indexed by $\nu$. Then there is $\mathbf{l}\in \mathcal{T}^2_{\sigma}\cap \mathcal{T}^2_{\nu}$ with the lowest arity number such that $x=y$ in $\mathcal{C}_{\mathbf{l}}$. Then for any $t$, either $H_{\sigma}(x,t)=H_{\nu}(y,t)$ in $\mathcal{C}_{\mathbf{l}}$ or $H_{\sigma}(x,t)=H_{\nu}(y,t)$ in $\mathcal{C}_{\mathbf{k}}$ where $\mathbf{k}\in \mathcal{T}^2_{\sigma}\cap \mathcal{T}^2_{\nu}$ and $\mathbf{k}$ has arity number one greater than or equal to that of $\mathbf{l}$. Thus, $(*3a)$ holds. Lastly, if $x=y$ in $\mathrm{Int}(\mathcal{C}_\mathbf{l})$ where $\mathbf{l}$ has arity number $1$, then $\sigma=\nu$ and so $H_{\sigma}(x,1)=H_{\sigma}(y,1)=H_{\nu}(y,1)$ in $\mathcal{F}_{\sigma}=\mathcal{F}_{\nu}$; otherwise, $H_{\sigma}(x,1)=H_{\nu}(y,1)$ in $\mathcal{F}_{\mathbf{k}}$ for some $\mathbf{k}\in \mathcal{J}^1_{\sigma}\cap \mathcal{J}^1_{\nu}$. Hence, $(*3b)$ holds.\\

For $n\geq 4$, from the description of $H_{\sigma}$ when $t\in [0,\frac{1}{n}]$, we see that $(*1)$ holds. Now let $x\sim_{\mathcal{F}}y$ where $x\in P_n$ indexed by $\sigma$ and $y\in P_n$ indexed by $\nu$. Then $x=y$ in $\mathrm{Int}(\mathcal{F}_{\mathbf{l}})$ for some $\mathbf{l}\in \mathcal{J}_{\sigma}\cap \mathcal{J}_{\nu}$. Then $H_{\sigma}(x,t)=H_{\nu}(y,t)$ in $\mathcal{F}_{\mathbf{l}}$ (but not necessarily in $\mathrm{Int}(\mathcal{F}_{\mathbf{l}})$). Thus, $(*2)$ holds. Now let $x\sim_{\mathcal{C}}y$ where $x\in P_n$ indexed by $\sigma$ and $y\in P_n$ indexed by $\nu$. Then there is $\mathbf{l}\in \mathcal{T}^{n-1}_{\sigma}\cap \mathcal{T}^{n-1}_{\nu}$ with the lowest arity number such that $x=y$ in $\mathcal{C}_{\mathbf{l}}$. Then for any $t$, either $H_{\sigma}(x,t)=H_{\nu}(y,t)$ in $\mathcal{C}_{\mathbf{l}}$ or $H_{\sigma}(x,t)=H_{\nu}(y,t)$ in $\mathcal{C}_{\mathbf{k}}$ where $\mathbf{k}\in \mathcal{T}^{n-1}_{\sigma}\cap \mathcal{T}^{n-1}_{\nu}$ and $\mathbf{k}$ has arity number greater than or equal to that of $\mathbf{l}$. Therefore, $(*3a)$ holds. Finally, let $x\sim_{\mathcal{C}}y$ where $x\in P_n$ indexed by $\sigma$ and $y\in P_n$ indexed by $\nu$, so there is $\mathbf{l}\in \mathcal{T}_{\sigma}\cap \mathcal{T}_{\nu}$ such that $x=y$ in $\mathcal{C}_{\mathbf{l}}$. Then there is $\mathbf{k}=\mathbf{k}_1|\cdots|\mathbf{k}_k\in \mathcal{F}_{\sigma}\cap \mathcal{F}_{\nu}$ such that $H_{\sigma}(\mathcal{C}_{\mathbf{l}},\frac{n-1}{n})=H_{\nu}(\mathcal{C}_{\mathbf{l}},\frac{n-1}{n}) \subset \mathcal{F}_{\mathbf{k}}$. By the definition of the homotopy, there is a $\mathbf{k'}<\mathbf{k}$ such that $H_{\sigma}(x,1)=H_{\nu}(y,1)$ in $\mathcal{F}_{\mathbf{k}'}$, establishing $(*3b)$.

\end{proof}

\section{Discussion: Relations to other $E_2$ operads, application and Outlook}
\label{discussionpar}
There seem to be two breeds of $E_2$ operads. The first, and older ones, are useful for the recognition of loop spaces, like the little
discs, the little cubes and the Steiner operad. The other, and the newer generation, are good for solving Deligne's conjecture. Of course there are the cofibrant models, which are by definition a hybrid. These have the drawback that they are usually a bit too abstract to handle to give actual operadic operations, by which we mean they act only through a factorization via a more concrete operad.

The first type usually has configuration spaces as deformation retracts, namely, as we have discussed, the have Milgram's models $\{\CalF(n)\}$ as retracts, which is classically used in the loop space program \cite{Milgram66,Ber97}. See \cite{MSS}[Chapter2.4] as a good survey. One feature of $\CalF$, however, is that it is not an operad in any known way. 
There are some remnants \cite{MSS,Ber97} using convex hulls,
but there is not even a closed cellular operad structure.

On the other hand, the other models have an algebraic aspect, which allows one to define operations on the Hochschild complex.
Their diversity is actually not as big as one would think.
On the cell/chain level they are variations of the Gerstenhaber--Voronov's GV-operad of braces and multiplication  \cite{GV}. On the topological level, they all retract to $Cact$, which deformation retracts to $\CalC=Cact^1$.
The difference to the above is that $\CalC$ is actually a chain level operad, while $\CalF$ is not. Moreover, $\CalC$ has the b/w tree structure,
making it ready to give operations, such as in Deligne's conjecture.

Let us briefly go through the list and history. Most of the operads
where actually constructed first on the combinatorial level and then realized as topological spaces, using totalizations, realizations or condensations. For the operad $Cact$ the story was the inverse. It existed first as a topological operad \cite{KLP}, and then it was realized that it is $E_2$ and it has an operadic cellular chain model $CC_*(Cact^1)$ \cite{Kau02} corresponding to GV. 

The first operad is the Kontsevich--Soibelman minimal operad $M$, which is the generalization of the GV operad to the $A_{\infty}$ case. It gave the first solution to Deligne's conjecture and works over $\mathbb{Z}$. The procedure here is a little different as the Fulton--MacPherson compactification and a W-construction were used. The fact that it can be realized as a version of cacti is contained in \cite{KSchw}. If the $A_{\infty}$ algebra is strict, it contacts to $CC_*(Cact)$, see \cite{Kau10}. The next operad was by McClure and Smith \cite{MS} and it gives a cosimplicial description of the GV--operad and hence can realize a topological operad using totalization.
The current paper also fixes the homtopy type of formulas in \cite{MS} directly. 

In their second paper, McClure and Smith \cite{MS03} introduced sequence notation for the GV operad, which has been one of the most influential constructions in the theory. 
The sequence operad $S_2$ is with hindsight actually isomorphic to the description by b/w bi--partite trees \cite{Kau02,Kau03} as we recall below. The totalization of \cite{MS} then reconstructs spineless cacti on the topological level.
Alternatively, \cite{MS03} used Berger's machine. The operad structure of
this comparison was further clarified in \cite{BF}.

The third paper \cite{MS04} contains the very fruitful functor operad approach. This gives back the usual operations if
certain degeneracy maps are applied.  That this procedure is also operadic follows from a more general theorem  proved in \cite{hoch2}[Theorem 4.4]. The degeneracies are given as angle markings. An intermediate step to the full realization is given by partitioning, see  \cite{hoch2}. This provides a colored set operad structure. This was cast into a cosimplicial/simplicial picture in \cite{BB} as the lattice path operad. Here the map from $\Delta^n$ specifies the partitioning in the sense of \cite{hoch2} and is equivalent
to using the foliation operator of \cite{Kau03}. After doing the condensation,
one again obtains spineless cacti on the topological level.

The quotient map we discuss is probably related to the surjection of
\cite{BF}, but this is not clear at the moment.
What is clear is that the quotient makes the non--operad $\CalF$ into an operad on the cell level and even a topological quasi--operad, i.\ e.\ associative up to homotopy.

For the cyclic Deligne conjecture, which is an application of the $E_2$ identification,  the story is similar.
The first proof is in \cite{cyclic} using $Cacti$. The first full proof that $Cacti$ are indeed equivalent to $fD_2$ is in \cite{Kau02}.
Here one uses the $E_2$ structure of spinless cacti and the fact that $Cacti$ is a bi--crossed product.
The paper \cite{BB} adds
the simplicial structure on the chain level explaining the usual operations as action of a colored operad. Another version of this partitioning is contained in \cite{hoch2}. Condensation then reconstructs cacti.
Another chain level action was given slightly   later than \cite{cyclic} in \cite{TradlerZeinalian}. They did not, however, show that the relevant chain level operads are models for the framed little discs. 
 So, on the topological level the operads adapted to prove the cyclic version are basically all cacti. The added difficulty, is that for cacti one is dealing with a bi--crossed product and not just a semi--direct product. The cyclic $A_{\infty}$ version again based on cacti is contained in \cite{Ward}.

\subsection{Bijection of $CC_*(Cact^1)$ and $S_2$}
\label{isopar}
The isomorphism is clear from the isomorphisms 
$$\xymatrix{S_2\ar@{<->}[r]^>>>>{\simeq}&{\rm formulas}\ar@{<->}[r]^>>>>{\simeq}&GV\ar@{<->}[r]^>>>>{\simeq}&CC_*(Cact^1)\ar@{<->}[r]^>>>>{\simeq}&\mathcal{T}_{bp}^{pp,nt}
}
$$
contained in \cite{MS,MS03,Kau02}.

This isomorphism  between $CC_*(Cact^1)$ and $S_2$ is actually explicitly given in \cite{Kau08}, which also contains the idea of cacti with stops, i.e.\ monotone parameterizations as explained in \cite{letter}, used by Salvatore in the cyclic case, to rewrite the existing proofs in the language of McClure--Smith.
 Given a cactus or equivalently a b/w tree, one
obtains a sequence as follows. Go around the outside circle and record the lobe number you see. Equivalently, for a planar planted b/w tree all the white angles (i.e. pairs of subsequent flags to a white vertex) come in a natural order by embedding in the plane.
Reading off the labels give the direct morphism, which is easily seen to be an isomorphism.

In fact, \cite{Kau08}[Proposition 4.11] actually contains a generalization to the full $E_{\infty}$ structures. Here one obtains an obviously surjective map. The injectivity is clear on the $E_2$ level.

\subsection{Lifing of a cellular quasi--isomorphism}
Let $CC_*(\mathcal{F}(n))$ be the cellular chains of $\mathcal{F}(n)$. In \cite{Tur}, Tourchine constructed a homomorphism $I_*: CC_*(\mathcal{F}(n))\rightarrow CC_*(Cact^1(n))$ of complexes and showed that $I_*$ is a quasi-isomorphism using homological algebra. By checking the definition of $I_*$ on Page 882 of \cite{Tur}, one immediately has the following.

\begin{proposition}
The homomorphism $I_*: CC_*(\mathcal{F}(n))\rightarrow CC_*(Cact^1(n))$ is induced from the homotopy equivalence $\overline{1}_n: \mathcal{F}(n)\rightarrow Cact^1(n)$. Thus, $I_*$ is a quasi-isomorphism.
\end{proposition}

\subsection{Further connections}
Furthermore, since both $\mathcal{F}(n)$ and $\mathcal{C}(n)$ are obtained by gluing $n!$ contractible polytopes, our work also has connections to Batanin's theory of the symmetrization of contractible $2-$operads (and $n$-operads in general) \cite{Bat}. It is worth pointing out that instead of making the spaces bigger by compactifying the various deformation retract models of the configuration spaces $F(\mathbb{R}^2,n)$ to get operadic structures, one can subdivide the constituent contractible pieces ($P_n$) and then do further gluing to get the quasi-operad $\{Cact^1(n)\}_{n\geq 1}$, which become a bona-fide operad $\{Cact(n)\}_{n\geq 1}$ after taking the semi-direct product with the scaling operad \cite{Kau02}.

Another $E_2$ operad which is quite different from $\mathcal{C}_2$ or $\mathcal{D}_2$ is $\{|C_2S_n|\}_{n\geq 1}$, the realization of the second term of the Smith filtration of the simplicial universal bundle $WS_n$. A reformulation of $C_2S_n$ resembles the presentation of $\mathcal{F}(n)$: $C_2S_2$ is built up out of $n!$ copies of the nerve $\mathcal{N}(S_n,<)$ where $<$ is the weak Bruhat order on the set $S_n$  \cite{Ber99}. It can be readily checked that for $n=1,2,3$, $|\mathcal{N}(S_n,<)|$ deformation retracts to $P_n$ and $C_2S_n$ deformation retracts to $\mathcal{F}(n)$. The retraction can be explicitly described. The same is hoped for $n\geq 4$, even though the dimension of $|C_2S_n|$ grows quadraticly as $\frac{n(n-1)}{2}$ and difficulty already arises when $n=4$.

One can also try relating the higher dimensional little discs (cubes etc.) operad to higher dimensional cacti operad through products of permutahedra. We refer the reader to \cite{Kau08} for a higher dimensional version of the cacti operad.

\section{Appendix. Proof of Proposition \ref{ProductRetract} and Proposition \ref{extended}.}

\subsection{Proof of Proposition \ref{ProductRetract}}

Proposition \ref{ProductRetract} follows from two lemmas.

\begin{lemma}
\label{BallRetract}
Let $S^{n-2}_{\pm}$ be the upper(lower)-hemisphere $\{(x_1,\cdots,x_{n-1})\in \mathbb{R}^{n-1}\big| x_1^2+\cdots+x_{n-1}^2=1, x_{n-1}\geq 0 (x_{n-1}\leq 0)\}$ in $\mathbb{R}^{n-1}$. Then $S^{n-2}_+$ is a deformation retract of the unit cell $D^{n-1}$.
\end{lemma}

\begin{proof}
Define $$H_{D^{n-1}}:D^{n-1}\times I\rightarrow D^{n-1}$$ by $$H_{D^{n-1}}((x_1,x_2,\cdots,x_{n-1}),t)=(x_1,x_2,\cdots,x_{n-2},(1-t)x_{n-1}+t\sqrt{1-x_1^2-\cdots-x_{n-2}^2}).$$ It can be readily checked that $H_{D^{n-1}}$ is a well-defined homotopy. Geometrically, $H_{D^{n-1}}$ contracts each fiber over a point $(x_1,\cdots,-\sqrt{1-x_1^2-\cdots-x_{n-2}^2})$ on $S^{n-2}_-$ to the point $(x_1,\cdots,\sqrt{1-x_1^2-\cdots-x_{n-2}^2})$ on $S^{n-2}_+$. In addition, $H_{D^{n-1}}(\cdot,t)$ is the identity map on $S^{n-2}_+$ for all $t\in I$. So $H_{D^{n-1}}$ is a deformation retraction from $D^{n-1}$ onto $S^{n-2}_+$ ($S^{n-2}_+$ is a deformation retract of $D^{n-1}$).

\begin{figure}[!h]
		\centering
		\includegraphics[width=150mm]{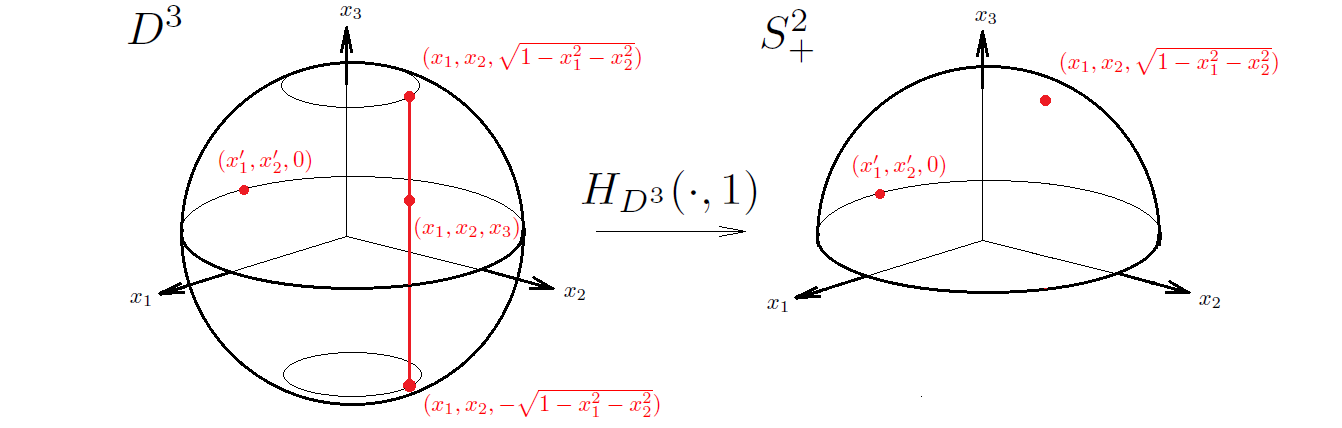}
		\caption{Deformation retraction from $D^3$ to $S^2_+$.}
\end{figure}
\end{proof}

\begin{lemma}
\label{ProductToBall}
There is a homeomorphism from $P_{m_1}\times P_{m_2}\times \cdots \times P_{m_k}\times\Delta^k$ to $D^{n-1}$ mapping $P_{m_1}\times P_{m_2}\times \cdots \times P_{m_k}\times(\bigcup_{i=2}^kv_1\cdots \widehat{v_i}\cdots v_{k+1})$ homeomorphically onto $S^{n-2}_-$.
\end{lemma}

\begin{proof}
Let $\mathcal{C}_{m_l}$ be the geometric center of $P_{m_l}$, $l=1,\cdots,k$ and $F_{j_l}$, $j_l\in I_l$ the facets of $P_{m_l}$. Let $\mathcal{C}_k=\frac{1}{k+1}(v_1+\cdots+v_{k+1})$. So $\mathcal{C}_k$ is the barycenter of $\Delta^k$. Then $P_{m_l}=\bigcup_{j_l\in I_l}F_{j_l}*\mathcal{C}_{m_l}$ where $*$ is the join operation, and $\Delta^k=\bigcup_{i=1}^{k+1}v_1\cdots\widehat{v_i}\cdots v_{k+1}\mathcal{C}_k$. So $P_{m_1}\times \cdots \times P_{m_k}\times \Delta^k=$ $$\bigcup_{l=1}^k\bigcup_{j_l\in I_l}\bigcup_{i=1}^{k+1}(F_{j_1}*\mathcal{C}_{m_1})\times \cdots \times (F_{j_l}*\mathcal{C}_{m_l})\times \cdots \times (F_{j_k}*\mathcal{C}_{m_k})\times (v_1\cdots\widehat{v_i}\cdots v_{k+1}\mathcal{C}_k).$$

Notice that each product on the right above has $\mathcal{C}=(\mathcal{C}_{m_1},\cdots,\mathcal{C}_{m_k},\mathcal{C}_k)$ as one of its vertices and is the union of line segments from $\mathcal{C}$ to a point of $\bigcup_{l=1}^k(F_{j_1}*\mathcal{C}_{m_1})\times \cdots \times F_{j_l}\times \cdots \times (F_{j_k}*\mathcal{C}_{m_k})\times (v_1\cdots\widehat{v_i}\cdots v_{k+1}\mathcal{C}_k)\bigcup (F_{j_1}*\mathcal{C}_{m_1})\times \cdots \times (F_{j_k}*\mathcal{C}_{m_k})\times (v_1\cdots\widehat{v_i}\cdots v_{k+1})$ and the lines intersect only at $\mathcal{C}$.\\

The line segments of different products having $\mathcal{C}$ as a vertex above agree on the intersections. So $P_{m_1}\times \cdots \times P_{m_k}\times \Delta^k$ is the union of line segments emanating from $\mathcal{C}$ and the union of the end points different from $\mathcal{C}$ of the line segments is $\partial (P_{m_1}\times \cdots \times P_{m_k}\times \Delta^k)$.\\

\begin{figure}[!ht]
		\centering
		\includegraphics[width=150mm]{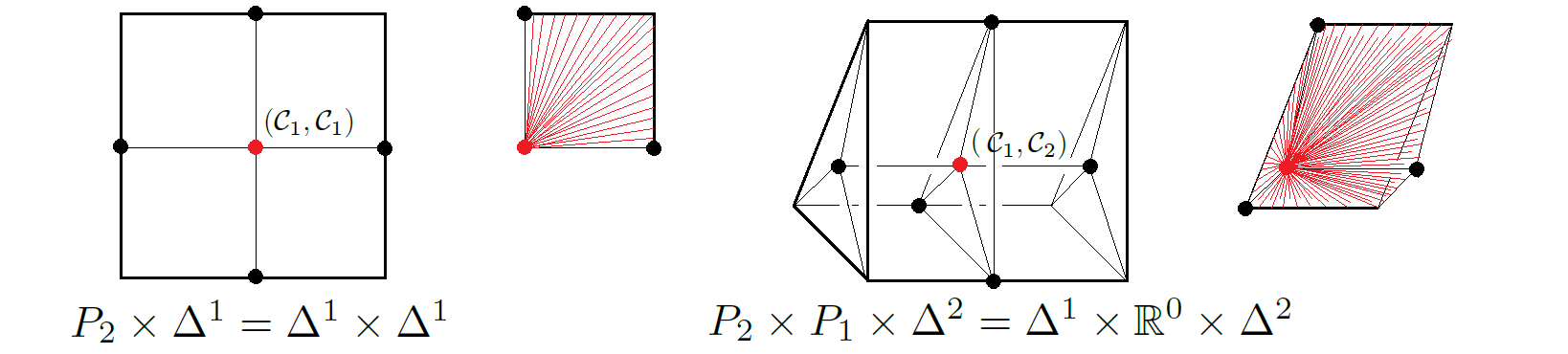}
		\caption{Decompositions of $P_2\times \Delta^1$ and $P_2\times P_2\times \Delta^2$ into line segments joined at their $\mathcal{C}$.}
\end{figure}

We will use a similar proof of that of Lemma 1.1 in \cite{Mun2} four times from now on. Here is the first time.\\

Let $T:P_{m_1}\times \cdots \times P_{m_k}\times \Delta^k\rightarrow \mathbb{R}^{n-1}$ be the translation defined by $T(x)=x-\mathcal{C}$. Let $r':\mathbb{R}^{n-1}\backslash \{(0,\cdots,0)\}\rightarrow S^{n-2}$ be the radial contraction given as $r'(x)=\frac{x}{|x|}$, where $|x|$ is the Euclidean norm of $x$. Since each half open ray emanating from $(0,\cdots,0)$ intersects with $T(\partial (P_{m_1}\times \cdots \times P_{m_k}\times \Delta^k))$ at one and only one point, $r'$ restricts to a continuous bijection $r:T(\partial (P_{m_1}\times \cdots \times P_{m_k}\times \Delta^k))\rightarrow S^{n-2}$. Being the continuous image of a compact space, $T(\partial (P_{m_1}\times \cdots \times P_{m_k}\times \Delta^k))$ is compact, and $S^{n-2}$ is Hausdorff, so $r$ is indeed a homeomorphism.\\

Now we extend $r:T(\partial (P_{m_1}\times \cdots \times P_{m_k}\times \Delta^k))\rightarrow S^{n-2}$ to $R:T(P_{m_1}\times \cdots \times P_{m_k}\times \Delta^k)\rightarrow D^{n-1}$. Define
\begin{equation*}
R(x)=\left\{
\begin{array}{ll}
\displaystyle \frac{x}{|r^{-1}(\frac{x}{|x|})|}& \mbox{ if } x\neq (0,\cdots,0),\\
\displaystyle (0,\cdots,0)&\mbox{ if } x=(0,\cdots,0).
\end{array}\right.
\end{equation*}
Except for at $x\neq (0,\cdots,0)$, $R$ is also continuous at $x=(0,\cdots,0)$. To see this, let $L$ be a lower bound of the Eulidean norm on $\partial (P_{m_1}\times \cdots \times P_{m_k}\times \Delta^k)$. Then for any $\epsilon>0$, if $|x|<L\epsilon$, then $|R(x)-R(0)|=\frac{|x|}{|r^{-1}(\frac{x}{|x|})|}\leq \frac{L\epsilon}{L}=\epsilon$.\\

Since $R$ is a continuous bijection from compact $T(P_{m_1}\times \cdots \times P_{m_k}\times \Delta^k)$ to Hausdorff $D^{n-1}$, $R$ is indeed a homeomorphism. Furthermore, $T$ is also a homeomorphism onto its image. So $R\circ T$ is a homeomorphism from $P_{m_1}\times \cdots \times P_{m_k}\times \Delta^k$ to $D^{n-1}$ mapping $\partial (P_{m_1}\times \cdots \times P_{m_k}\times \Delta^k)$ homeomorphically onto $S^{n-2}$.\\

So far, $P_{m_1}\times \cdots \times P_{m_k}\times (\bigcup_{i=2}^kv_1\cdots \widehat{v_i}\cdots v_{k+1})$ has not been mapped to the lower-hemisphere $S^{n-2}_-$. We will use stereographic projection to achieve this.\\

Recall $\Delta^k=v_1v_2\cdots v_{k+1}$ and $\partial \Delta^k=(\bigcup_{i=2}^kv_1\cdots \widehat{v_i}\cdots v_{k+1})\bigcup(v_1v_2\cdots v_k\bigcup v_2\cdots v_kv_{k+1})$. Then
\begin{enumerate}
\item $\mathcal{N}_k=\frac{1}{k-1}(v_2+\cdots+v_k)$ is in $\mathrm{Int}(v_1v_2\cdots v_k\bigcup v_2\cdots v_kv_{k+1})$.\\

\item $\mathcal{S}_k=\frac{1}{2}(v_1+v_{k+1})$ is in $\mathrm{Int}(\bigcup_{i=2}^kv_1\cdots \widehat{v_i}\cdots v_{k+1})$.\\

\item $\mathcal{N}_k$, $\mathcal{C}_k$ and $\mathcal{S}_k$ lie on the same line and $|\mathcal{N}_k\mathcal{C}_k|:|\mathcal{N}_k\mathcal{S}_k|=2:(k+1)$.
\end{enumerate}

\begin{figure}[!h]
		\centering
		\includegraphics[width=150mm]{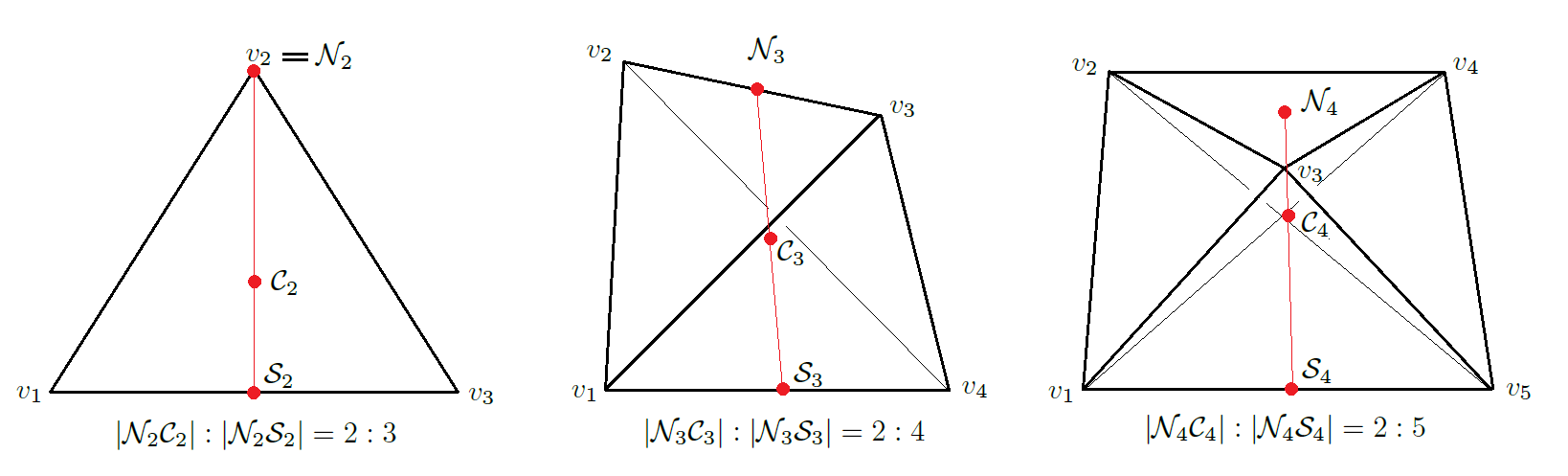}
		\caption{$\mathcal{N}_i$, $\mathcal{C}_i$ and $\mathcal{S_i}$, $i=2,3,5$.}
\end{figure}

Now we let $\mathcal{N}=(\mathcal{C}_{m_1},\cdots,\mathcal{C}_{m_k},\mathcal{N}_k)$ and $\mathcal{S}=(\mathcal{C}_{m_1},\cdots,\mathcal{C}_{m_k},\mathcal{S}_k)$. So
\begin{enumerate}
\item $\mathcal{N}$ is in the relative interior of $\partial(P_{m_1}\times \cdots \times P_{m_k}\times \Delta^k)\backslash \mathrm{Int}(P_{m_1}\times \cdots \times P_{m_k}\times (\bigcup_{i=2}^kv_1\cdots \widehat{v_i}\cdots v_{k+1}))$.\\

\item $\mathcal{S}$ is in the relative interior of $P_{m_1}\times \cdots \times P_{m_k}\times (\bigcup_{i=2}^kv_1\cdots \widehat{v_i}\cdots v_{k+1})$.\\

\item $\mathcal{C}$ is in $\mathrm{Int}(P_{m_1}\times \cdots \times P_{m_k}\times \Delta^k)$.\\

\item $\mathcal{N}$, $\mathcal{C}$ and $\mathcal{S}$ lie on the same line and $|\mathcal{N}\mathcal{C}|:|\mathcal{N}\mathcal{S}|=2:(k+1)$.
\end{enumerate}

Let $\phi$ be an element in $SO(n\mbox{-}1)$ rotating the vector $\mathcal{N}-\mathcal{C}$ so that it is aligned with the positive $x_{n-1}-$axis. Notice that $\phi:D^{n-1}\rightarrow D^{n-1}$ is a homeomorphism. So $b:=\phi\circ R\circ T$ is a homeomorphism.\\

Let $N=(0,\cdots,0,1)\in \mathbb{R}^{n-1}$ and $S=(0,\cdots,0,-1)\in \mathbb{R}^{n-1}$. The stereographic projections $$p'_N:S^{n-2}\backslash N\rightarrow \mathbb{R}^{n-2}$$ $$(x_1,\cdots,x_{n-1})\mapsto (\frac{x_1}{1-x_{n-1}},\frac{x_2}{1-x_{n-1}},\cdots,\frac{x_{n-2}}{1-x_{n-1}})$$ and $$p'_S:S^{n-2}\backslash S\rightarrow \mathbb{R}^{n-2}$$ $$(x_1,\cdots,x_{n-1})\mapsto (\frac{x_1}{1+x_{n-1}},\frac{x_2}{1+x_{n-1}},\cdots,\frac{x_{n-2}}{1+x_{n-1}})$$  are homeomorphisms with inverses $${p'}_N^{-1}(y_1,\cdots,y_{n-2})=(\frac{2y_1}{1+y_1^2+\cdots+y_{n-2}^2},\cdots,\frac{2y_{n-2}}{1+y_1^2+\cdots+y_{n-2}^2},\frac{-1+y_1^2+\cdots+y_{n-2}^2}{1+y_1^2+\cdots+y_{n-2}^2}),$$ $${p'}_S^{-1}(y_1,\cdots,y_{n-2})=(\frac{2y_1}{1+y_1^2+\cdots+y_{n-2}^2},\cdots,\frac{2y_{n-2}}{1+y_1^2+\cdots+y_{n-2}^2},\frac{1-y_1^2-\cdots-y_{n-2}^2}{1+y_1^2+\cdots+y_{n-2}^2}).$$

Since $$D^{n-2}_S:=b(P_{m_1}\times \cdots \times P_{m_k}\times (\bigcup_{i=2}^kv_1\cdots \widehat{v_i}\cdots v_{k+1}))\subset S^{n-2}\backslash N$$ and $$D^{n-2}_N:=b(\partial(P_{m_1}\times \cdots \times P_{m_k}\times \Delta^k)\backslash \mathrm{Int}(P_{m_1}\times \cdots \times P_{m_k}\times (\bigcup_{i=2}^kv_1\cdots \widehat{v_i}\cdots v_{k+1})))\subset S^{n-2}\backslash S,$$ $p'_N$ and $p'_S$ restric to homeomorphisms $p_N$ and $p_S$ from $D^{n-2}_S$ and $D^{n-2}_N$ respectively to their images.\\

$P_{m_1}\times \cdots \times P_{m_k}\times (\bigcup_{i=2}^kv_1\cdots \widehat{v_i}\cdots v_{k+1})$ admits a presentation of the following form
$$\bigcup_{l=1}^k\bigcup_{j_l\in I_l}\bigcup_{i=2}^k\bigcup_{j=1,k+1}(F_{j_1}*\mathcal{C}_{m_1})\times \cdots \times (F_{j_l}*\mathcal{C}_{m_l})\times \cdots \times (F_{j_k}*\mathcal{C}_{m_k})\times (v_1\cdots\widehat{v_i}\widehat{v_j}\cdots v_{k+1}\mathcal{S}_k).$$

Each product on the right has $\mathcal{S}=(\mathcal{C}_{m_1},\cdots,\mathcal{C}_{m_k},\mathcal{S}_k)$ as one of its vertices and it is a union of line segments from $\mathcal{S}$ to a point on
$\bigcup_{l=1}^k(F_{j_1}*\mathcal{C}_{m_1})\times \cdots \times F_{j_l}\times \cdots \times (F_{j_k}*\mathcal{C}_{m_k})\times (v_1\cdots\widehat{v_i}\widehat{v_j}\cdots v_{k+1}\mathcal{C}_k)\bigcup (F_{j_1}*\mathcal{C}_{m_1})\times \cdots \times (F_{j_k}*\mathcal{C}_{m_k})\times (v_1\cdots\widehat{v_i}\widehat{v_j}\cdots v_{k+1})$ and the lines intersect only at $\mathcal{S}$.\\

The line segments of different products above agree on the intersections. So $P_{m_1}\times \cdots \times P_{m_k}\times (\bigcup_{i=2}^kv_1\cdots \widehat{v_i}\cdots v_{k+1})$ is the union of line segments emanating from $\mathcal{S}$ and the union of the end points different from $\mathcal{S}$ of the line segments is $\partial (P_{m_1}\times \cdots \times P_{m_k}\times (\bigcup_{i=2}^kv_1\cdots \widehat{v_i}\cdots v_{k+1}))$. Thus, $p_N(D^{n-2}_S)$ is the closure of an open set in $\mathbb{R}^{n-2}$ containing the origin and $p_N(D^{n-2}_S)$ is a union of line segments emanating from the origin such that each segment intersects $\partial(p_N(D^{n-2}_S))$ at only one point.\\

Therefore, we have a homeomorphism $G_S:p_N(D^{n-2}_S)\rightarrow D^{n-2}\subset \mathbb{R}^{n-2}$ obtained similar to that of $R$.\\

Every line segment $x\mathcal{S}$ above and the point $\mathcal{C}$ determine a half plane. This half plane intersects $\partial (P_{m_1}\times \cdots \times P_{m_k}\times \Delta^k)$ at a piecewise linear path $\mathcal{N}y_1\cdots y_m x \mathcal{S}$ such that $\mathcal{N}y_1\cdots y_m x$ is mapped to a line segment emanating from the origin in $p_S(D^{n-2}_N)$. Thus, $p_S(D^{n-2}_N$) is the closure of an open set in $\mathbb{R}^{n-2}$ containing the origin and $p_S(D^{n-2}_N)$ is a union of line segments emanating from the origin such that each segment intersects $\partial(p_S(D^{n-2}_N))$ at only one point.\\

Therefore, we have a homeomorphism $G_N:p_S(D^{n-2}_N)\rightarrow D^{n-2}\subset \mathbb{R}^{n-2}$ obtained similar to that of $R$.\\

Now we define $f:S^{n-2}\rightarrow S^{n-2}$ by
\begin{equation*}
f(x)=\left\{
\begin{array}{ll}
p_N^{-1}\circ G_S \circ p_N (x), & x\in D^{n-2}_S\\
p_S^{-1}\circ G_N \circ p_S (x), & x\in D^{n-2}_N.
\end{array}\right.
\end{equation*}

Notice that each branch of $f$ is continuous  and they agree on the overlap. (See the figure below.) So $f$ is continuous. Being a continuous bijection from compact Hausdorff $S^{n-2}$ onto itself, $f$ is thus a homeomorphism.\\

\begin{figure}[!ht]
		\centering
		\includegraphics[width=150mm]{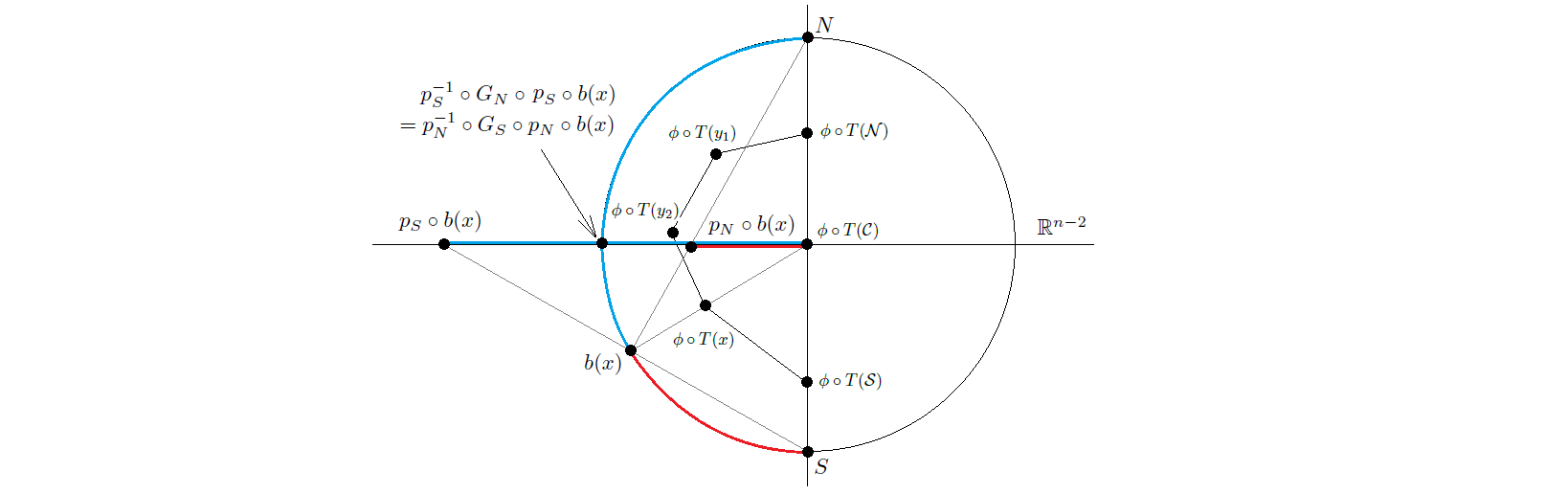}
		\caption{Two stereographic projections.}
\end{figure}

Now we extend $f:S^{n-2}\rightarrow S^{n-2}$ to $F:D^{n-1}\rightarrow D^{n-1}$ by

\begin{equation*}
F(x)=\left\{
\begin{array}{ll}
\displaystyle f(\frac{x}{|x|})|x|, & x\neq (0,\cdots,0)\\
(0,\cdots,0), & x=(0,\cdots,0).
\end{array}\right.
\end{equation*}

Similar to $R$, $F$ is continuous at $x=(0,\cdots,0)$  because for any $\epsilon>0$, if $|x|<\epsilon$, then $|F(x)-F(0)|=|f(\frac{x}{|x|})||x|=|x|<\epsilon$. Being a continuous bijection from compact Hausdorff $D^{n-1}$ onto itself, $F$ is thus a homeomorphism.\\

Therefore, $F\circ b$ is a homeomorphism from $P_{m_1}\times P_{m_2}\times \cdots \times P_{m_k}\times\Delta^k$ to $D^{n-1}$ mapping $P_{m_1}\times P_{m_2}\times \cdots \times P_{m_k}\times(\bigcup_{i=2}^kv_1\cdots \widehat{v_i}\cdots v_{k+1})$ homeomorphically onto $S^{n-2}_-$. Thus, $F\circ b$ also maps $\partial(P_{m_1}\times P_{m_2}\times \cdots \times P_{m_k}\times\Delta^k)\backslash \mathrm{Int}(P_{m_1}\times P_{m_2}\times \cdots \times P_{m_k}\times(\bigcup_{i=2}^kv_1\cdots \widehat{v_i}\cdots v_{k+1}))$ homeomorphically onto $S^{n-2}_+$.
\end{proof}

\begin{proof}[Proof of Proposition \ref{ProductRetract}]
Define $H_{P_{m_1}\times \cdots \times P_{m_k}\times \Delta^k}:(P_{m_1}\times \cdots \times P_{m_k}\times \Delta^k)\times I\rightarrow P_{m_1}\times \cdots \times P_{m_k}\times \Delta^k$ by $$H_{P_{m_1}\times \cdots \times P_{m_k}\times \Delta^k}(x,t)=b^{-1}\circ F^{-1}\circ H_{D^{n-1}}(F\circ b(x),t).$$
\end{proof}

\subsection{Proof of Proposition \ref{extended}}. Now we prove Proposition \ref{extended}.\\

\begin{proof}[Proof of Proposition \ref{extended}]

Let $$S^{n-2}_-\times I_{\epsilon}:=\{(x_1,\cdots,x_{n-2},x_{n-1}-\delta)\in \mathbb{R}^{n-1}\big|(x_1,\cdots,x_{n-1})\in S^{n-2}_-,0\leq \delta\leq \epsilon\},$$ and the extended closed $(n-1)$-cell be
$$\mathrm{Ext}_{\epsilon}(D^{n-1}):=D^{n-1}\bigcup (S^{n-2}_-\times I_{\epsilon}).$$

Define $$H_{S^{n-2}_-\times I_{\epsilon}}:(S^{n-2}_-\times I_{\epsilon})\times I\rightarrow \mathrm{Ext}_{\epsilon}(D^{n-1})$$ by $H_{S^{n-2}_-\times I_{\epsilon}}((x_1,\cdots,x_{n-1}),t)=$ $$(x_1,\cdots,x_{n-2},x_{n-1}+t(x_{n-1}+\sqrt{1-x_1^2-\cdots-x_{n-2}^2}+\epsilon)\frac{2\sqrt{1-x_1^2-\cdots-x_{n-2}^2}}{\epsilon}).$$

Then $H_{S^{n-2}_-\times I_{\epsilon}}$ is a well-defined homotopy. It linearly extends each fiber over $(x_1,\cdots,x_{n-2},x_{n-1}-\epsilon)$ in $S^{n-2}_-\times I_{\epsilon}$ to the fiber in $\mathrm{Ext}_{\epsilon}(D^{n-1})$. For each $t\in I$, $H_{S^{n-2}_-\times I_{\epsilon}}(\cdot,t)$ is a homeomorphism onto its image.

\begin{figure}[!h]
		\centering
		\includegraphics[width=150mm]{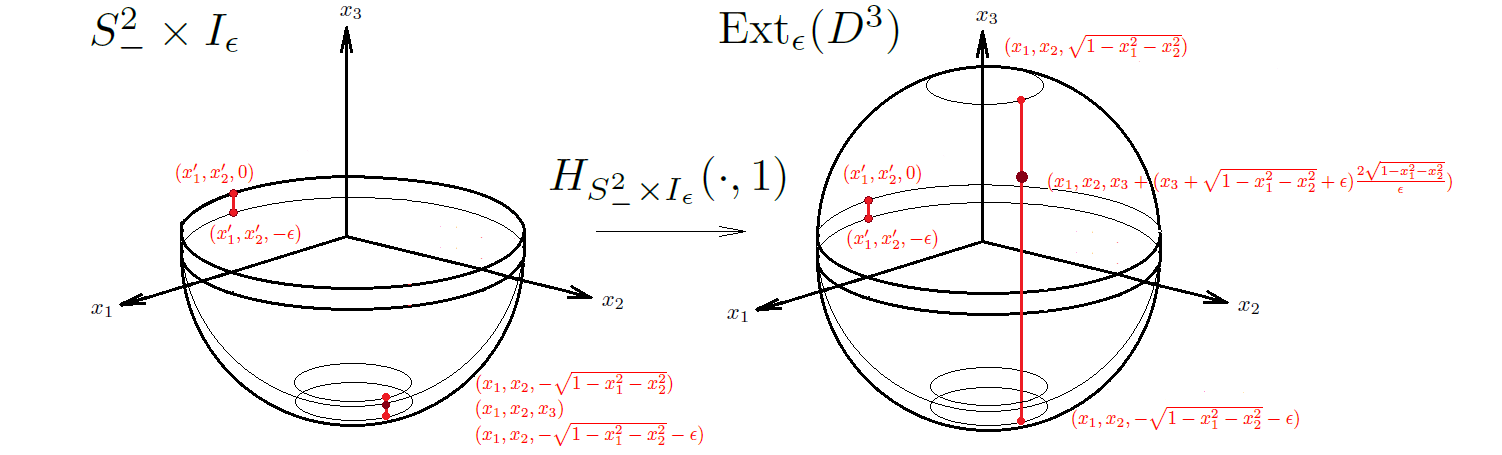}
		\caption{The extended closed 3-cell $\mathrm{Ext}_{\epsilon}(D^3)$.}
\end{figure}

By extending $R$ and possibly perturbing $F$, we get $\widetilde{R}$ and $\widetilde{F}$ such that $$\widetilde{F}\circ \phi\circ \widetilde{R}\circ T:(P_{m_1}\times \cdots \times P_{m_k}\times (\bigcup_{i=2}^kv_1\cdots\widehat{v_i}\cdots v_{k+1}))\times I_{\epsilon}\rightarrow S^{n-2}_-\times I_{\epsilon}$$ is a homeomorphism whose images of $P_{m_1}\times \cdots \times P_{m_k}\times (\bigcup_{i=2}^kv_1\cdots\widehat{v_i}\cdots v_{k+1})\times \{0\}$ and $P_{m_1}\times \cdots \times P_{m_k}\times (\bigcup_{i=2}^kv_1\cdots\widehat{v_i}\cdots v_{k+1})\times \{\epsilon\}$ are $S^{n-2}_-$ and $S^{n-2}_--\{(0,\cdots,0,\epsilon)\}$, respectively. Furthermore, $F\circ b$ ($=F\circ \phi \circ R \circ T$) and $\widetilde{F}\circ \phi\circ \widetilde{R}\circ T$ agree on $P_{m_1}\times \cdots \times P_{m_k}\times (\bigcup_{i=2}^kv_1\cdots\widehat{v_i}\cdots v_{k+1})$.\\

We define $H_{(P_{m_1}\times \cdots \times P_{m_k}\times (\bigcup_{i=2}^kv_1\cdots\widehat{v_i}\cdots v_{k+1}))\times I_{\epsilon}}:$ $$((P_{m_1}\times \cdots \times P_{m_k}\times (\bigcup_{i=2}^kv_1\cdots\widehat{v_i}\cdots v_{k+1}))\times I_{\epsilon})\times I\rightarrow \mathrm{Ext}_{\epsilon}(P_{m_1}\times \cdots \times P_{m_k}\times \Delta^k)$$ by $$H_{(P_{m_1}\times \cdots \times P_{m_k}\times (\bigcup_{i=2}^kv_1\cdots\widehat{v_i}\cdots v_{k+1}))\times I_{\epsilon}}(x,t)=T^{-1}\circ \widetilde{R}^{-1}\circ \phi^{-1}\circ \widetilde{F}^{-1}\circ H_{S^{n-2}_-\times I_{\epsilon}}(\widetilde{F}\circ \phi\circ \widetilde{R}\circ T(x),t).$$

Notice that $H_{P_{m_1}\times \cdots \times P_{m_k}\times \Delta^k}(\cdot,t)$ and $H_{(P_{m_1}\times \cdots \times P_{m_k}\times (\bigcup_{i=2}^kv_1\cdots\widehat{v_i}\cdots v_{k+1}))\times I_{\epsilon}}(\cdot,t)$ agree on $P_{m_1}\times \cdots \times P_{m_k}\times (\bigcup_{i=2}^kv_1\cdots\widehat{v_i}\cdots v_{k+1})$ for each $t\in I$. Then we can define $H_{\mathrm{Ext}_{\epsilon}(P_{m_1}\times \cdots \times P_{m_k}\times \Delta^k)}:\mathrm{Ext}_{\epsilon}(P_{m_1}\times \cdots \times P_{m_k}\times \Delta^k)\times I\rightarrow \mathrm{Ext}_{\epsilon}(P_{m_1}\times \cdots \times P_{m_k}\times \Delta^k)$ by
$H_{\mathrm{Ext}_{\epsilon}(P_{m_1}\times \cdots \times P_{m_k}\times \Delta^k)}(x,t)=H_{P_{m_1}\times \cdots \times P_{m_k}\times \Delta^k}(x,t)$ if $x\in P_{m_1}\times \cdots \times P_{m_k}\times \Delta^k$ and $H_{\mathrm{Ext}_{\epsilon}(P_{m_1}\times \cdots \times P_{m_k}\times \Delta^k)}(x,t)=$ \\
$H_{(P_{m_1}\times \cdots \times P_{m_k}\times (\bigcup_{i=2}^kv_1\cdots\widehat{v_i}\cdots v_{k+1}))\times I_{\epsilon}}(x,t)$ if $x \in (P_{m_1}\times \cdots \times P_{m_k}\times (\bigcup_{i=2}^kv_1\cdots\widehat{v_i}\cdots v_{k+1}))\times I_{\epsilon}$. It can be readily checked that the three conditions are satisfied.\\

\end{proof}

\end{document}